\documentclass[12pt,reqno]{amsart}

\usepackage{amsfonts}
\usepackage{eurosym}
\usepackage{amssymb}
\usepackage{amsthm}
\usepackage{amsmath}
\usepackage{amsaddr}
\usepackage{bm}
\usepackage{cite}
\usepackage{mathrsfs}
\usepackage{xcolor}
\usepackage[OT1]{fontenc}
\usepackage[left=1.8cm, right=1.8cm, top=3cm]{geometry}
\usepackage{hyperref}
\usepackage{graphicx}
\usepackage[title]{appendix}

\usepackage{setspace}
\numberwithin{equation}{section}

\hypersetup{colorlinks=true, linkcolor=blue, citecolor=red, urlcolor=blue}
\RequirePackage{times}
\allowdisplaybreaks
\newtheorem{theorem}{Theorem}[section]

\newtheorem{lemma}[theorem]{Lemma}

\newtheorem{remark}[theorem]{Remark}

\def\n{\textbf{\textit{n}}}
\def\R3{\mathbb{R}^3}
\def\F2o{\overline{F_2}}

\def\d{{\rm d}}
\def \l {\langle}
\def \r {\rangle}

\def\ddt{\frac{\d}{\d t}}

\def\e{\rm e}

\def\L2{L^2(\Omega)}

\def\uloc{\mathrm{uloc}}
\def \au {\rm}
\def \ti {\it}
\def \jou {\rm}
\def \bk {\it}
\def \no#1#2#3 {{\bf #1} (#3), #2.}
\def \eds#1#2#3 {#1, #2, #3.}
\def \nome#1#2 {{\bf #1}, (#2).}

\title[Active Cahn-Hilliard equation]{\large Numerical Analysis and coarsening dynamics \\[5pt] of the Active Cahn-Hilliard equation}
\author{\textsc{Abramo Agosti}} 
\address{Università degli Studi di  Pavia,
Dipartimento di Matematica,
Via Ferrata 5, 27100, Pavia (Italy)}
\email{abramo.agosti@unipv.it}
\author{\textsc{Andrea Giorgini}} 
\address{Politecnico di Milano, Dipartimento di Matematica, Via E. Bonardi 9, 20133, Milano (Italy)}
\email{andrea.giorgini@polimi.it}

\date{\today }

\begin{document}

\begin{abstract}
In this paper, we investigate the analysis and phase ordering dynamics of the active Cahn--Hilliard equation (active Model B), which is a non-integrable variant of the Cahn--Hilliard equation used to model active materials. In particular, we provide novel results beyond the current state of the art concerning the well-posedness and the characterization of its coarsening dynamics. 
We consider both regular polynomial and singular logarithmic (Flory-Huggins) bulk energy potentials. We make use of phase-plane methods, rigorous partial differential equation (PDE) analysis and discrete numerical approximations.
In particular, we exploit a new method based on heteroclinic trajectories in the phase plane to characterize static kink profiles and spherical droplet states. Through phase-plane trajectory analysis at leading perturbative orders, we recover the exact values of key quantities related to static phase-separated configurations which have already been characterized, exactly or numerically, by alternative approaches presented in the literature; moreover, we develop a theory accounting for surface tension modifications driven simultaneously by activity and local interface curvature, which explains the power-law shift $L(t)\sim t^{\frac{1}{z}}$ from $z=3$ to $z=4$ induced by activity for the characteristic domain length conjectured in the literature. This shows that there is a transitory effect before the attainment of a finite saturation length. We also design an efficient numerical scheme, based on piecewise linear continuous finite elements and semi-implicit time discretization, to approximate the model, proving its well-posedness and stability both for regular and singular double-well potentials. In dimensions $d=2,3$ with singular potentials, the convergence analysis of the finite element approximation proves the local-in-time existence and uniqueness of weak solutions for active Model B strictly satisfying the physical constraint $\phi \in (-1, 1)$. In dimension $d=1$ with singular potential, we establish global-in-time well-posedness and regularity of weak solutions under a smallness condition on the activity parameter $\lambda$. Finally, we show numerical simulations for different test cases which prove that our numerical algorithm correctly reproduces the expected phase separation dynamics, verifying that regular and singular potentials produce qualitatively identical phase separation behaviors. Moreover, we show the numerical results for coarsening dynamics at late times which present a power law shift from $z=3$ to $z=4$ prior to reaching late-time length saturation, which confirms our theoretical findings.
\end{abstract}

\maketitle
\tableofcontents

\section{Introduction}

We study the active Cahn-Hilliard equation, also known as active Model B,
\begin{equation}  \label{ac-CH}
\partial_t \phi = \Delta \mu, \quad \mu = -\Delta \phi +\Psi^{\prime}(\phi)+ \lambda |\nabla \phi|^2  \quad 
\text{in } \Omega \times (0,\infty),
\end{equation}
where $\Omega$ is either the interval $(0,L)$ in one dimension or a bounded and smooth domain $\Omega \subset \mathbb{R}^d$, $d=2,3$,
equipped with the following boundary and initial conditions 
\begin{equation}  \label{ac-CH-bc}
\partial_\n \phi = \partial_\n \mu=0 \quad \text{on } \partial \Omega \times (0,T),\quad
\phi(\cdot,0)= \phi_0 \quad \text{in } \Omega.
\end{equation}
Here, $\n$ is the outward normal vector on $\partial \Omega$. 
The active Cahn-Hilliard \eqref{ac-CH} can also be rewritten as
\begin{equation}  \label{ac-CH-2}
\partial_t \phi = -\Delta^2 \phi + \Delta \left( \Psi^{\prime}(\phi)\right) + \lambda \Delta \left( |\nabla \phi|^2\right) \quad 
\text{in } \Omega \times (0,\infty).
\end{equation}
When $\lambda=0$, the model \eqref{ac-CH}-\eqref{ac-CH-bc} corresponds to the classical Cahn-Hilliard equation, also known as passive Model B \cite{CH,CH2}.

The nonlinear function $\Psi$, which represents the bulk free energy, is typically a regular double-well potential given by a quartic polynomial, i.e.
\begin{equation}
\label{pol}
\Psi(s)=\frac{(s^2-1)^2}{4},
\end{equation}
or the singular double-well potential of logarithmic type, called the Flory-Huggins (also Boltzmann-Gibbs entropy) potential
\begin{equation}
\label{log}
\Psi(s)=\frac{\theta}{2}\bigg[(1+s)\ln
(1+s)+(1-s)\ln (1-s)\bigg]- \frac{\theta_0}{2}s^2,\quad s\in [-1,1],
\end{equation}
where the parameters $\theta$ and $\theta_0$ satisfy the conditions $0<\theta<\theta_0$. The potential \eqref{log} is physically relevant since it enforces the solution $\phi$ to take values in $[-1,1]$.
 
 The term $\lambda |\nabla \phi|^2$ is an {\it active} term with strength $\lambda \in \mathbb{R}$, which breaks time-reversal symmetry. 
In the case $\lambda \neq 0$, the model is non-integrable and $\lambda$ can be viewed as the ``distance" to the nearest integrable model. This means that $\mu$ is not derivable as the variational derivative of a Ginzburg-Landau free energy.

Over the past decade, the Statistical Physics community has made significant progress
in the exploration of Active Matter \cite{CT2018,TJHUNG2018, WITT}, which displays extraordinary collective phenomena that are unattainable in passive systems. Active Matter refers to the class of materials encompassing particles that convert energy into directed motion, known as motility, independently of external forces. The resultant self-propulsion
induces particles to exhibit collective behavior, such as aggregation in high-density regions. This
phenomenon is termed Motility-Induced Phase Separation (MIPS). Active Matter is prevalent in biological systems, such as living cells, bacteria, actin filaments, and colloids. The potential applications of Active Matter in biomedicine and engineering are incredibly promising \cite{TJHUNG2015}, including the development of selfhealing materials, bio-inspired robotics, and biofilms.

The most straightforward way to describe MIPS within the Phase Field theory involves integrating an active
term that accounts for the violation of the time-reversal symmetry (TRS). The basic model proposed in \cite{WITT} is the active Cahn-Hilliard equation (or active Model B), which corresponds to \eqref{ac-CH} with the regular double-well potential \eqref{pol}. 

In comparison with the passive Model B, the term $\lambda |\nabla \phi|^2$ is responsible for breaking TRS. Although the active term is new in the context of Phase Separation, the square of gradient term
is common in the realm of PDEs modeling nonlinear physical phenomena and pattern formation. This
term features prominently in models like the Burgers equation or the renowned Kardar-Parisi-Zhang (KPZ) equation. In this context, the fundamental question posed within the Statistical Physics community is: when do the explicit TRS breaking terms control the physics at large scales? A significant ongoing debate in current literature, initially sparked by \cite{WITT} and persisting
through ongoing discussions \cite{PATTANAYAK2021,PATTANAYAK2021-2}, concerns the extent to which active terms impact coarsening dynamics.
So far, empirical simulations indicate, at most, quantitative variations rather than qualitative shifts when compared to the passive case \cite{PATTANAYAK2021}. Here we summarize some known results and open questions on the phase ordering dynamics of Active Model B, which are supported by thorough numerical studies of coarsening kinetics reported in \cite{WITT,PATTANAYAK2021}. 
\begin{itemize}
    \item[-] \textbf{Existence of static kink solutions.} Through a Newton mapping of the static solutions associated to Active Model B, the authors of \cite{WITT} proved the existence of static bulk phase separation characterized by a non-zero constant value of the static chemical potential $\mu$ and by two equilibrium
    coexisting phases separated by a planar interface, whose densities are determined by an uncommon tangent construction, leading to an activity-induced pressure-jump across the interface. 
    \item[-] \textbf{Existence of spherical droplet solutions.} In \cite{WITT} the existence of a particular static configuration of a circular droplet of one phase with a large radius of order $\lambda^{-1}$ immersed in the other phase, for which the active and Laplace pressures are equal and opposite, was also considered. For this particular static configuration, common tangency is restored . The Newton mapping method led the authors of \cite{WITT} to obtain only a numerical approximation of the value of the droplet radius.
    \item[-] \textbf{Power law shift for domain growth.} In \cite{PATTANAYAK2021} it was conjectured that the domain growth kinetics for the active Model B is characterized by a crossover in the power law $L(t)\sim t^{\frac{1}{z}}$ for the characteristic size of separated domains from $z=3$ at early times, corresponding to the classical Lifshit-Slyozov (LS) growth law for passive Model B, to $z=4$ at late times. The authors of \cite{PATTANAYAK2021} were not able to prove this conjecture, posing as an open problem the development of a sophisticated theory which accounts for the modification of the surface tension not only by the activity term but also by the local curvature of the interface in order to prove the conjecture. Also, they provided extensive numerical evidences for its validity.   
\end{itemize}

\noindent
From the mathematical viewpoint, the active CH equation is an evolutionary highly nonlinear fourth-order PDE. Despite conserving total mass akin to the passive case, the active term disrupts the gradient flow structure inherent in the Passive Cahn-Hilliard equation. Furthermore, the energy equation contains an additional unsigned strongly nonlinear term that cannot be directly controlled by either the energy itself or
the dissipation through known techniques. Hence, the energy not only fails to be non-increasing, as in the passive case, but currently, it remains uncertain whether the energy landscape is bounded over the dynamics. As a result, the theoretical exploration of the active Cahn-Hilliard model remained completely open so far.

In this work, we firstly give (in Section $2$) a new theoretical perspective to characterize the static solutions and the phase ordering dynamics of Active Model B, which is alternative to the Newton mapping presented in \cite{WITT}. We will in particular characterize the static bulk phase separation as an \textit{heteroclinic trajectory} in the phase diagram. This theoretical perspective, which has been employed in literature to characterize the traveling wave solutions for the bistable equation (see e.g. \cite[Chapter 6]{KS}), will prove in the present case to be more effective than the Newton mapping approach in deriving exact results and in giving a possible explanation to the observed shift in the power law for domain growth. In particular, we will be able to develop a theory which accounts for the modification of the surface tension not only by the activity term but also by the local curvature of the interface, as demanded in \cite{PATTANAYAK2021} to explain the conjectured power law shift for domain growth. We will find evidences, both theoretically and numerically, that the active term induces \textit{qualitative shifts} in the coarsening dynamics with respect to the passive case, leading to a saturation effect for the characteristic domain size which is similar to the one proved in \cite{TJHUNG2018,Burekovic} for another active model (the active Model B+), caused by a reverse Ostwald mechanism.    

As a second contribution of our work, we will present analytical and numerical results for the active Model B. In particular, we will design (in Section $4$) a well-posed and conditionally stable numerical approximation of active Model B, based on a computationally efficient lowest order finite element and semi-implicit time discrete approximation, both for the cases with the regular and with the singular potential. In the singular potential case, thanks to the physical bound $\phi \in [-1,1]$ enforced by the singularities at the pure phases, we will be able to prove the existence of a unique local in time weak solution to the active Model B, obtained within a convergence analysis for the discrete approximation as the limit point of the sequence of discrete solutions as the discretization parameters tend to zero. For the sake of readability, we will preliminarily show (in Section $3$) the formal a-priori estimates at the continuous level which are rigorously reproduced at the discrete level to prove the convergence of the discrete approximation. The discrete a-priori estimates and convergence properties will be derived by employing highly technical results related to the proposed finite element approximation, like a discrete Gagliardo--Nirenberg inequality and an implicit nonlinear Gronwall inequality, which we will derive here as a further novel result for the present purpose. Moreover, in the case with the singular potential and spatial dimension $d=1$, we will prove that the local in time weak solution is actually global in time.

Finally, we will show (in Section $5$) numerical results which illustrate the phase separation and coarsening dynamics of the active Model B both for the cases with the regular and the singular potentials, proving that they are qualitatively similar. Also, we will show the growth kinetics for the characteristic domain size at late times, comparing the results with those reported in \cite{WITT,PATTANAYAK2021}, proving the shift from $z=3$ to $z=4$ before the attainement of a saturation length, which supports our theoretical findings.

\section{Static kink solutions and coarsening dynamics}
First, we introduce the following dimensional version of the Active Model B:
\begin{equation}  \label{ac-CHdim}
\begin{cases}
\partial_t \phi = D\Delta \mu, \\ 
\mu = -\Gamma \epsilon\Delta \phi +\frac{\Gamma}{\epsilon}\Psi^{\prime}(\phi)+ \gamma \Gamma |\nabla \phi|^2,  
\end{cases}
\end{equation}
where 
\begin{itemize}

    \item $D$ is a positive parameter representing the cells motility, of units $[m^4N^{-1}s^{-1}]$; \item $\Gamma$ is a positive parameter representing surface tension, of units $[N\,m^{-1}]$; 
    \item $\epsilon$ is a positive parameter representing the interface thickness, of units $[m]$; 
    \item $\gamma$ is the activity parameter, representing the signed length associated to interface growth induced by cells accumulation, of units $[m]$. 
\end{itemize}  
All the previous units for the model parameters have been reported in three spatial dimensions.

In order to obtain an adimensionalized version of \eqref{ac-CHdim}, we define the persistence length $l$, which is a characteristic length scale for the active dynamics, and the relaxation time 
\[
\tau=\frac{l^3}{D\Gamma}.
\]
Next, we make the change of variables
\[
\tilde{t}=\frac{t}{\tau}, \quad \tilde{x}=\frac{x}{l},
\]
and we also introduce the adimensionalized variables and parameters
\begin{align*}
&\tilde{\mu}=\frac{l\mu}{\Gamma}, \quad 
\tilde{\epsilon}=\frac{\epsilon}{l}, \quad \tilde{\gamma}=\frac{\gamma}{l}.
\end{align*}
Then, without reporting the tilde superscript on the variables for ease of notation, we  obtain the following adimensionalized version of \eqref{ac-CHdim}:
\begin{equation}  \label{ac-CHadim}
\begin{cases}
\partial_t \phi = \Delta \mu, \\ 
\mu = -\tilde{\epsilon}\Delta \phi +\frac{1}{\tilde{\epsilon}}\Psi^{\prime}(\phi)+ \tilde{\gamma} |\nabla \phi|^2,  
\end{cases}
\end{equation}
Taking $l\equiv \epsilon$, system \eqref{ac-CHadim} coincides with \eqref{ac-CH} with $\lambda:=\frac{\gamma}{\epsilon}$. Hence we conclude that the adimensional parameter $\lambda$ in \eqref{ac-CH} can be identified as the ratio between the activity length scale $\gamma$ and the interface thickness $\epsilon$, when describing the dynamics at the interface length scale.

\subsection{Existence of static kink solutions}

In this section we develop a new methodology to prove the existence of static kink solutions to the active Model B, i. e. of fully phase-separated static states with a planar interface. As already said in the Introduction, this was also proved in \cite{WITT} by means of the Newton mapping method, which gave an explanation of the numerical results reporting bulk demixing as the output of phase separation dynamics for active Model B. In order to prove the validity of our methodology, we will compare our results with the results obtained by the analysis employed in \cite{WITT,PATTANAYAK2021}. To accomplish this, we consider here the system \eqref{ac-CH} with the regular potential \eqref{pol}. We start by introducing the basis for our methodology in the simple case of passive Model B, i.e. in the case with $\lambda=0$. In the latter situation it is known that a static kink solution exists for $\mu\equiv 0$, with bulk densities $\phi=\pm 1$ determined by a common tangent construction. Let us introduce the spatial coordinate $z$ normal to the interface, which varies from $-\infty$ to $+\infty$. The static kink solution, which by the translational invariance depends only on the variable $z$, satisfies the equation
\begin{equation}
    \label{kink1}
    -\phi+\phi^3-\frac{d^2\phi}{dz^2}=0,
\end{equation}
with boundary conditions $\phi(-\infty)\to \mp 1$, $\phi(+\infty)\to \pm 1$. We rewrite the second order ODE \eqref{kink1} as a coupled system of two first order ODEs in the following way:
\begin{equation}
    \label{kink2}
    \begin{cases}
    \displaystyle \frac{d\phi}{dz}=\chi,\\[8pt]
    \displaystyle \frac{d\chi}{dz}=\Psi^{\prime}(\phi).
    \end{cases}
\end{equation}
We observe that the points with coordinates $(-1,0)$ and $(1,0)$ in the phase plane $\{(\phi,\chi)\in \mathbb{R}^2\}$ are saddle points for system \eqref{kink2}, while the point with coordinates $(0,0)$ is a center point for cyclic trajectories. A static kink solution can be identified as an \textit{heteroclinic} trajectory, which is a trajectory in the phase plane, parametrized by the $z$ coordinate, which connects the two saddle points $(-1,0)$ and $(1,0)$, i.e. which approaches $(-1,0)$ for $z\to \mp \infty$ and $(1,0)$ for $z\to \pm \infty$. If we multiply \eqref{kink2}$_2$ by $\chi$ and integrate from $-\infty$ to $z$, we obtain that
\[
\left(\frac{\chi^2}{2}\right)\biggr|_{-\infty}^z=\Psi(\phi)\bigr|_{-\infty}^z.
\]
Since $\chi(-\infty)=\frac{d \phi}{dz}(-\infty)=0$ and $\Psi(\phi(-\infty))=\Psi(\mp 1)=0$, we obtain that
\[
\chi(z)=\frac{d\phi}{dz}(z)=\pm \sqrt{2}\sqrt{\Psi(\phi(z))}= \pm \frac{(\phi^2(z)-1)}{\sqrt{2}},
\]
which gives a parabolic trajectory in the phase plane which connects the saddle points $(-1,0)$ and $(1,0)$, and which has the static kink profile solution
\[
\phi(z)=\pm \tanh\left(\frac{z}{\sqrt{2}}\right),
\]
with the plus or minus signs determining if the trajectory connects $(-1,0)$ for $z\to -\infty$ to $(+1,0)$ for $z\to \infty$ or vice versa.

We extend now the previous arguments to the case with $\lambda \neq 0$. Let us firstly characterize the roots of the cubic polynomial
\begin{equation}
    \label{cubpol}
    -\phi +\phi^3=\mu_s,
\end{equation}
where $\mu_s$ is the static constant value of the chemical potential, which may be nonzero when $\lambda \neq 0$. When $-\frac{2}{3\sqrt{3}}<\mu_s <\frac{2}{3\sqrt{3}}$, the polynomial \eqref{cubpol} admits three distinct real roots $\phi_a<\phi_b<\phi_c$, given by the expression
\begin{equation}
    \label{cubpol2}
    \{\phi_a,\phi_b,\phi_c\}=\sqrt[3]{\frac{\mu_s}{2}+\sqrt{\frac{\mu_s^2}{4}-\frac{1}{27}}}+\sqrt[3]{\frac{\mu_s}{2}-\sqrt{\frac{\mu_s^2}{4}-\frac{1}{27}}}.
\end{equation}
The static kink solution, if it exists, satisfies the equation
\begin{equation}
    \label{kink1lamb}
    -\phi+\phi^3-\frac{d^2\phi}{dz^2}+\lambda \left(\frac{d\phi}{dz}\right)^2=\mu_s,
\end{equation}
with boundary conditions $\phi(-\infty)\to \phi_a$, $\phi(+\infty)\to \phi_c$ or $\phi(-\infty)\to \phi_c$, $\phi(+\infty)\to \phi_a$. In the following, we will always assume to operate in the regime $-\frac{2}{3\sqrt{3}}<\mu_s <\frac{2}{3\sqrt{3}}$. Also, we will assume that $\lambda <<1$ and we will develop a perturbative analysis at order $O(\lambda)$, unless otherwise specified, as done in \cite{WITT,PATTANAYAK2021}. As observed in \cite{WITT,PATTANAYAK2021}, due to the property that $\mu_s(-\lambda)=-\mu_s(\lambda)$, which can be deduced from \eqref{kink1lamb}, we may express $\mu_s$ in powers of $\lambda$, at order $O(\lambda^{2N+1})$, $N\in \mathbb{N}$, as
\begin{equation}
    \label{muslamb}
    \mu_s(\lambda)=\sum_{i=0}^{N}\mu_{s,2i+1}\lambda^{2i+1},
\end{equation}
where the set of real coefficients $\{\mu_{s,2i+1}\}_{i=0}^N$ must be determined in such a way that a kink solution to \eqref{kink1lamb} exists. Plugging \eqref{muslamb} in \eqref{cubpol2} and after some algebraic manipulations it is possible to express the roots $\{\phi_a,\phi_b,\phi_c\}$ as power series in $\lambda$. In particular, considering only $O(\lambda)$ terms, i.e. writing $\mu_s=\mu_{s,1}\lambda$, we obtain that
\begin{equation}
\label{rootslamb}
\phi_a=-1+\frac{\mu_{s,1}}{2}\lambda, \quad \phi_b=-\mu_{s,1}\lambda, \quad \phi_c=1+\frac{\mu_{s,1}}{2}\lambda.
\end{equation}
The same result \eqref{rootslamb} can be obtained by expressing $\{\phi_a,\phi_b,\phi_c\}$ as small perturbations at order $O(\lambda)$ of the values $\{-1,0,1\}$, which are the roots of the cubic polynomial $\eqref{cubpol}$ with $\lambda =0$, i.e. $\phi_a=-1+\lambda \chi_a$, $\phi_b=\lambda \chi_b$, $\phi_c=1+\lambda \chi_c$, and plugging their expressions in $\eqref{cubpol}$ linearized around the values $\{-1,0,1\}$, getting finally that $\chi_a=\chi_c=\frac{\mu_{s,1}}{2}$, $\chi_b=-\mu_{s,1}$. 

We now rewrite the second order ODE \eqref{kink1lamb} as a coupled system of two first order ODEs in the following way:
\begin{equation}
    \label{kink2lamb}
    \begin{cases}
    \displaystyle \frac{d\phi}{dz}=\chi,\\[8pt]
    \displaystyle \frac{d\chi}{dz}=\Psi^{\prime}(\phi)-\mu_{s}+\lambda \chi^2.
    \end{cases}
\end{equation}
We observe that the points with coordinates $(\phi_a,0)$, $(\phi_b,0)$ and $(\phi_c,0)$ are critical points of system \eqref{kink2lamb} in the phase plane $\{(\phi,\chi)\in \mathbb{R}^2\}$. It is possible to prove that, in the regime $-\frac{2}{3\sqrt{3}}<\mu_s <\frac{2}{3\sqrt{3}}$, $(\phi_a,0)$ and $(\phi_c,0)$ are saddle points for system \eqref{kink2lamb}, while $(\phi_b,0)$ is a center point for cyclic trajectories. Without going into long calculations to prove the latter properties, we make the following observations. Given the critical point $(\phi_i,0)$, $i\in \{a,b,c\}$, it is a saddle point if $\phi_i<-\frac{1}{\sqrt{3}}$ or $\phi_i>\frac{1}{\sqrt{3}}$, while it is a centre point of cyclic trajectories if $-\frac{1}{\sqrt{3}}<\phi_i<\frac{1}{\sqrt{3}}$. In the limiting case $\mu_s= \frac{2}{3\sqrt{3}}$, we have that $\left(\phi_a=-\frac{2}{\sqrt{3}},0\right)$ is a saddle point, while $(\phi_b,0)$ and $(\phi_c,0)$ coalesce into an improper double critical point $\left(\phi_b=\phi_c=\frac{1}{\sqrt{3}},0\right)$. For $\mu_s <\frac{2}{3\sqrt{3}}$, $(\phi_b,0)$ and $(\phi_c,0)$ are separated, with $\phi_b<\frac{1}{\sqrt{3}}$ and $\phi_c>\frac{1}{\sqrt{3}}$. Similarly, in the limiting case $\mu_s= -\frac{2}{3\sqrt{3}}$, we have that $\left(\phi_a=\phi_b=-\frac{1}{\sqrt{3}},0\right)$ is a double improper critical point, while $\left(\phi_c=\frac{2}{\sqrt{3}},0\right)$ is a saddle point. For $\mu_s> -\frac{2}{3\sqrt{3}}$, $(\phi_a,0)$ and $(\phi_b,0)$ are separated, with $\phi_a<-\frac{1}{\sqrt{3}}$ and $\phi_b>-\frac{1}{\sqrt{3}}$. Hence, in the regime $-\frac{2}{3\sqrt{3}}<\mu_s <\frac{2}{3\sqrt{3}}$, we can deduce that $(\phi_a,0)$ and $(\phi_c,0)$ are saddle points and $(\phi_b,0)$ is a center point for cyclic trajectories.

In the case $\lambda\neq 0$, a static kink solution can be identified as an \textit{heteroclinic} trajectory in the phase plane, parametrized by the $z$ coordinate, which connects the two saddle points $(\phi_a,0)$ and $(\phi_c,0)$, i.e. which approaches $(\phi_a,0)$ for $z\to \mp \infty$ and $(\phi_c,0)$ for $z\to \pm \infty$. Of course, in the present case we cannot proceed to identify the heteroclinic trajectory by multiplying \eqref{kink2lamb}$_2$ by $\chi$ and integrating from $-\infty$ to $z$. We instead assume that there exists an heteroclinic trajectory which, at order $O(\lambda^N)$, $N\in \mathbb{N}$, is a polynomial curve of degree $N+2$ of the form
\begin{equation}
    \label{heterocliniclambda}  \chi=\sum_{k=0}^N\lambda^k\sum_{\substack{i,j=1\\i+j=k+2}}^{k+1}A_{i,j}(\phi-\phi_a)^i(\phi-\phi_c)^j,
\end{equation}
where $A_{i,j}$ are real coefficients to be determined by inserting \eqref{heterocliniclambda} and \eqref{muslamb} in \eqref{kink2lamb} and using the principle of polynomial identity. The specific guess \eqref{heterocliniclambda} is chosen because, at order $O(\lambda^N)$, it satisfies that $\frac{d \chi}{dz}=\frac{d \chi}{d \phi}\frac{d \phi}{dz}\sim \phi^{N+3}$, and that $\lambda \chi^2\sim \phi^{N+3}$. The matching of the polynomial degrees of the terms $\frac{d \chi}{dz}$ and $\lambda \chi^2$ in \eqref{kink2lamb} is needed to have the existence of an heteroclinic trajectory for $\lambda\neq 0$. We now focus on the calculations at order $O(\lambda)$. Then, \eqref{heterocliniclambda} becomes
\begin{equation}
    \label{heterocliniclambda2}  \chi=A_{1,1}(\phi-\phi_a)(\phi-\phi_c)+\lambda A_{1,2}(\phi-\phi_a)(\phi-\phi_c)^2+\lambda A_{2,1}(\phi-\phi_a)^2(\phi-\phi_c),
\end{equation}
which is a cubic heteroclinic trajectory which connects $(\phi_a,0)$ and $(\phi_c,0)$ in the phase plane.
Using \eqref{rootslamb} in \eqref{heterocliniclambda2}, and keeping only terms at most of order $O(\lambda)$, we obtain
that
\begin{equation}
\label{heterocliniclambda3}
    \frac{d \chi}{dz}= 2A_{1,1}\phi \chi-\lambda A_{1,1}\mu_{s,1}\chi+2\lambda (A_{1,2}+A_{2,1})(\phi^2-1)\chi+\lambda A_{1,2}(\phi-1)^2\chi+\lambda A_{2,1}(\phi+1)^2\chi.
\end{equation}
Substituting \eqref{heterocliniclambda2} in \eqref{heterocliniclambda3}, still keeping only terms at most of order $O(\lambda)$, we ultimately arrive to the expression
\begin{align}
\label{heterocliniclambda4}
     \notag \frac{d \chi}{dz}= & 2A_{1,1}^2\phi(\phi^2-1)-2\lambda A_{1,1}^2\phi^2\mu_{s,1} +2\lambda A_{1,1}A_{1,2}\phi(\phi+1)(\phi-1)^2+2\lambda A_{1,1}A_{2,1}\phi(\phi+1)^2(\phi-1)\\
    & \notag -\lambda A_{1,1}^2(\phi^2-1)\mu_{s,1}
    +2\lambda A_{1,1}(A_{1,2}+A_{2,1})(\phi^2-1)^2\\
    &  +\lambda A_{1,1}A_{1,2}(\phi+1)(\phi-1)^3+\lambda A_{1,1}A_{2,1}(\phi+1)^3(\phi-1).
\end{align}
Also, substituting \eqref{heterocliniclambda2} in \eqref{kink2lamb}$_{2}$ and keeping only terms at most of order $O(\lambda)$, we obtain that 
\begin{equation}
\label{heterocliniclambda5}
    \frac{d \chi}{dz}= \phi(\phi^2-1)-\mu_{s,1}\lambda+\lambda A_{1,1}^2(\phi^2-1)^2.
\end{equation}
Equating \eqref{heterocliniclambda4} and \eqref{heterocliniclambda5}, at order $O(1)$ we obtain that $A_{1,1}=\pm \frac{1}{\sqrt{2}}$, while at order $O(\lambda)$ we apply the principle of polynomial identity and obtain that
\[
A_{1,2}=A_{2,1}=\pm \frac{1}{10\sqrt{2}}, \quad \mu_{s,1}=\frac{4}{15}.
\]
We have thus gained the same result $\mu_{s,1}=\frac{4}{15}$ obtained in \cite{WITT} with the method of Newton mapping. Substituting \eqref{rootslamb}, together with the values of $A_{1,1}, A_{1,2}, A_{2,1}$ and $\mu_{s,1}$, in \eqref{heterocliniclambda2}, we conclude that, at order $O(\lambda)$,
\begin{equation}
    \label{chiode}
    \chi(z)=\frac{d\phi}{dz}=\pm \frac{1}{\sqrt{2}}(\phi^2-1)\pm \lambda \left(\frac{1}{5\sqrt{2}}\phi(\phi^2-1)-\frac{4}{15\sqrt{2}}\phi \right).
\end{equation}
The static kink solution thus must solve the following boundary value problem:
\begin{equation}
    \label{chiodebv}
    \begin{cases}
    \displaystyle \frac{d\phi}{dz}=\pm \frac{1}{\sqrt{2}}(\phi^2-1)\pm \frac{\lambda}{5\sqrt{2}} \phi \left(\phi^2-\frac{7}{3}\right),\\
    \phi(-\infty)\to \pm 1+\frac{2}{15}\lambda,\\
    \phi(+\infty)\to \mp 1+\frac{2}{15}\lambda.
    \end{cases}
\end{equation}
In order to solve \eqref{chiodebv}, we assume that its solution can be written as $\phi(z)=\phi_0(z)+\lambda \phi_1(z)+o(\lambda)$. Then, at order $O(\lambda)$ we have that
\[
\frac{d \phi}{dz}=\frac{d \phi_0}{dz}+\lambda \frac{d \phi_1}{dz}=\pm \frac{1}{\sqrt{2}}(\phi_0^2-1)\pm \lambda \left(\sqrt{2}\phi_0\phi_1+\frac{1}{5\sqrt{2}}\phi_0 \left(\phi_0^2-\frac{7}{3}\right)\right).
\]
Hence, at order $O(1)$ we need to solve the boundary value problem
\begin{equation*}
    \begin{cases}
    \displaystyle \frac{d\phi_0}{dz}=\pm \frac{1}{\sqrt{2}}(\phi_0^2-1),\\
    \phi(-\infty)\to \pm 1,\\
    \phi(+\infty)\to \mp 1,
    \end{cases}
\end{equation*}
which has the solution
\[
\phi_0(z)=\mp \tanh\left(\frac{z}{\sqrt{2}}\right).
\]
At order $O(\lambda)$ we need to solve the boundary value problem
\begin{equation*}
    \begin{cases}
   \displaystyle \frac{d\phi_1}{dz}=- \sqrt{2}\tanh\left(\frac{z}{\sqrt{2}}\right)\phi_1(z)-\frac{1}{5\sqrt{2}}\tanh\left(\frac{z}{\sqrt{2}}\right)\left(\tanh^2\left(\frac{z}{\sqrt{2}}\right)-\frac{7}{3}\right),\\
    \phi(-\infty)\to \frac{2}{15},\\
    \phi(+\infty)\to \frac{2}{15},
    \end{cases}
\end{equation*}
which has the solution
\[
\phi_1(z)=\frac{1}{5}\frac{\log \left(\cosh \left(\frac{z}{\sqrt{2}}\right)\right)}{\cosh^2\left(\frac{z}{\sqrt{2}}\right)}+\frac{2}{15}.
\]
We finally obtain the following expression for the static kink solution at order $O(\lambda)$:
\begin{equation}
    \label{statkinklambda}
    \phi(z)=\mp \tanh\left(\frac{z}{\sqrt{2}}\right)+\frac{\lambda}{5}\frac{\log \left(\cosh \left(\frac{z}{\sqrt{2}}\right)\right)}{\cosh^2\left(\frac{z}{\sqrt{2}}\right)}+\frac{2}{15}\lambda.
\end{equation}
\subsection{Existence of spherical droplet solutions}
We now consider the existence of a large circular or spherical droplet, centered in the origin of the system of coordinates, of one pure phase immersed in a background of the other phase as a static solution of active Model B for $\mu_s\equiv 0$. Following \cite{WITT}, we indicate with $R^*$ the radius of the droplet and we set
\[
\xi:=\frac{d-1}{R^*},
\]
where $d=2,3$ is the spatial dimension. Through the method of the Newton mapping, a perturbative calculation to leading order in $R^*=O(\lambda^{-1})$ gave in \cite{WITT} the approximated result $\xi \approx 0.566 |\lambda|$, which approximates the expected exact result $\xi=\frac{\sqrt{8}}{5}|\lambda|$. The latter result was obtained in \cite{WITT} within theoretical arguments by formally balancing the Laplace pressure at the curved interface of the droplet with the pressure at the interface induced by the $\lambda$ term. Here, we are able to obtain the exact result for $\xi$ employing the method of search of heteroclinic trajectories in the phase plane. Due to the rotational symmetry of the problem, the static solution depends only on the radial coordinate $r$. To leading order in $R^*$, the static solution satisfies the equation
\[
-\phi+\phi^3-\frac{d^2\phi}{dr^2}-\xi \frac{d\phi}{dr}+\lambda\left(\frac{d\phi}{dr}\right)^2=0,
\]
which can be written as a coupled system of two first order ODEs as
\begin{equation}
    \label{kink2lambxi}
    \begin{cases}
    \displaystyle \frac{d\phi}{dr}=\chi,\\[8pt]
    \displaystyle \frac{d\chi}{dr}=\Psi^{\prime}(\phi)-\xi \chi+\lambda \chi^2.
    \end{cases}
\end{equation}
We observe that the points with coordinates $(-1,0)$ and $(1,0)$ are saddle points for system \eqref{kink2lambxi}, while the point $(0,0)$ is a stable spyral point for $\xi<2$ or a stable node for $\xi\geq 2$ (improper for $\xi=2$). We assume that there exists an heteroclinic trajectory wich connects the saddle points $(-1,0)$ and $(1,0)$ in the phase plane of the form
\begin{equation}
    \label{heterocliniclambda2xi}  \chi=A_{1,1}(\phi^2-1)+\lambda A_{1,2}(\phi+1)(\phi-1)^2+\lambda A_{2,1}(\phi+1)^2(\phi-1).
\end{equation}
Making similar calculations to those employed in \eqref{heterocliniclambda4}, keeping only terms at most of order $O(\lambda)$, we get
\begin{align}
\label{heterocliniclambda4xi}
     \notag \frac{d \chi}{dz}= & 2A_{1,1}^2\phi(\phi^2-1)+2\lambda A_{1,1}A_{1,2}\phi(\phi+1)(\phi-1)^2+2\lambda A_{1,1}A_{2,1}\phi(\phi+1)^2(\phi-1)\\
    & +2\lambda A_{1,1}(A_{1,2}+A_{2,1})(\phi^2-1)^2+\lambda A_{1,1}A_{1,2}(\phi+1)(\phi-1)^3+\lambda A_{1,1}A_{2,1}(\phi+1)^3(\phi-1).
\end{align}
Also, substituting \eqref{heterocliniclambda2xi} in \eqref{kink2lambxi}$_{2}$, assuming that $\xi=O(\lambda)$ and keeping only terms at most of order $O(\lambda)$, we obtain that 
\begin{equation}
\label{heterocliniclambda5xi}
    \frac{d \chi}{dz}= \phi(\phi^2-1)-\xi A_{1,1}(\phi^2-1)+\lambda A_{1,1}^2(\phi^2-1)^2.
\end{equation}
Equating \eqref{heterocliniclambda4xi} and \eqref{heterocliniclambda5xi}, at order $O(1)$ we obtain that $A_{1,1}=\pm \frac{1}{\sqrt{2}}$, while at order $O(\lambda)$ we apply the principle of polynomial identity and obtain that
\[
A_{1,2}=A_{2,1}=\pm \frac{1}{10\sqrt{2}}, \quad \xi=\frac{\sqrt{8}}{5}|\lambda|,
\]
the latter of which is the expected exact result.

\subsection{Power law shift for domain growth}
\label{pls}
We conclude this section by developing tools within our theoretical framework to gain insights into the crossover in the growth law for active Model B conjectured in \cite{PATTANAYAK2021}. In the latter reference, the authors tried to generalize the shrinking dynamics of a bubble of phase $\phi\equiv \phi_a$ in a background of phase $\phi\equiv \phi_c$ in space dimension $d=3$ to the case of non-zero background chemical potential and considering the presence of activity, finding that the bubble radius $R(t)$ obeys the ODE:
\begin{equation}
    \label{lslambda}
    \frac{dR}{dt}=-\frac{\sigma+O(\lambda^2)}{2}\frac{1}{R^2},
\end{equation}
where $\sigma$ is the surface tension of passive Model B, i.e. $\sigma=\frac{2\sqrt{2}}{3}$, and the $O(\lambda^2)$ represents the lower order corrections to the surface tensions due to activity. Since \eqref{lslambda} yields to a LS growth law with a $\lambda^2$-dependent prefactor, it is clear that this argument is not adequate to capture the crossover in the growth law observed for the Active Model B. As observed in \cite{PATTANAYAK2021}, a more sophisticated theory is needed which accounts for the modification of the surface tension not only by the activity term but also by the local curvature of the interface. Here we aim to develop such a theory, following the arguments developed in the previous sections based on the analysis of phase-plane dynamics. To improve readability, the detailed calculations are presented in the Appendix \ref{ap:shift}.

We consider the shrinking dynamics of a single spherical domain of phase $\phi\equiv \phi_a$, with radius $R(t)$, immersed in a sea of phase $\phi\equiv \phi_c$ in space dimension $d=3$. As observed in \cite{PATTANAYAK2021}, the growth laws are expected to be the same for $d\geq 2$. As a first step, we search to represent the local curvature effects for a static solution of active Model B given by a spherical droplet phase $\phi\equiv \phi_a$, centered in the origin, immersed in a sea of phase $\phi\equiv \phi_c$, as perturbative corrections to the static flat profile. For reasons which will be clear later, we need to develop a perturbative analysis up to second order in the perturbative parameters. We observe that, due to the spherical symmetry, a static droplet solution of active Model B does not necessarily require $\mu$ to be constant. Indeed, $\mu$ can be a multiple of $\frac{1}{r}$, where $r$ is the radial coordinate, since
\[
\Delta \left(\frac{1}{r}\right)=\frac{1}{r^2}\frac{d }{d r}\left(r^2\frac{d }{d r}\left(\frac{1}{r}\right)\right)=0, \quad \frac{d }{d r}\left(\frac{1}{r}\right)\to 0 \;\; \text{for} \;\; r\to \infty.
\]
Hence, a static droplet solution of Active Model B, if it exists, satisfies the equation
\[
-\phi+\phi^3-\frac{d^2\phi}{dr^2}-\frac{2}{r}\frac{d\phi}{dr}+\lambda \left(\frac{d\phi}{dr}\right)^2=\mu_s(\lambda) +\rho \frac{1}{r},
\]
where we have considered $\mu=\mu_s(\lambda)+\rho \frac{1}{r}$, $\rho \in \mathbb{R}$. We observe that, in order to avoid a singularity in the origin, the static value of the chemical potential should be expressed as $\mu=\mu_s(\lambda)+\rho \max\left\{\frac{1}{R},\frac{1}{r}\right\}$. Since we are only interested in obtaining perturbative corrections in terms of $R^{-1}$ to the static flat profile, and since in the perturbative analysis we can assume that the factor $\frac{1}{r}$ is approximately constant through the interface, we will disregard the singularity issue in the forthcoming analysis.
In order to proceed, we make the assumption that $R>>1$. This means that the droplet radius $R$ is much larger than the interface width, which is typically satisfied. Indeed, we have seen that \eqref{ac-CH} corresponds to \eqref{ac-CHadim} with $l\equiv \epsilon$, i.e. the interface length is considered to be of order $O(1)$ in \eqref{ac-CH}. Under this hypothesis, we expect that a static droplet solution has the form $\phi(r)=f(r-R)$, with the function $f$ changing from $\phi_a$ to $\phi_c$ in a small region of width $O(1)$ near $r=R$.
Since the droplet interface width is much smaller than its radius, as already said the factor $\frac{1}{r}$ is approximately constant through the interface, and a perturbative solution of the static droplet equation can be derived by solving the equation
\begin{equation}
\label{dropletsolutionlamb}
-f+f^3-\frac{d^2f}{d\tilde{r}^2}-\xi\frac{df}{d\tilde{r}}+\lambda \left(\frac{df}{d\tilde{r}}\right)^2=\mu_s(\lambda)+\rho \xi,
\end{equation}
where we set $\xi:=\frac{2}{R}$ and we introduced the variable $\tilde{r}:=r-R$. The roots $\{f_a,f_b,f_c\}$ of the cubic potential
\begin{equation}
\label{cubicxilamb}
-f+f^3=\mu_s(\lambda)+\rho \xi
\end{equation}
can be obtained as small perturbations of the values $\{-1,0,1\}$. As already anticipated, we need to perform the perturbation analysis up to the second order in the parameters $\lambda,\xi$, hence from \eqref{muslamb} we take $\mu_s(\lambda)=\mu_{s,1}\lambda$ and we express $f_a=-1+(\mu_{s,1}\lambda+\rho \xi) \chi_a+ (\mu_{s,1}\lambda+\rho \xi)^2 \zeta_a$, $f_b=(\mu_{s,1}\lambda+\rho \xi) \chi_b+ (\mu_{s,1}\lambda+\rho \xi)^2 \zeta_b$ and $f_c=1+(\mu_{s,1}\lambda+\rho \xi) \chi_c+ (\mu_{s,1}\lambda+\rho \xi)^2 \zeta_c$. Inserting these expressions in \eqref{cubicxilamb}, expanded to second order in $\lambda,\xi$, we obtain that
\begin{align}
    \label{rootsxilamb}
    & \notag f_a=-1+\frac{\mu_{s,1}\lambda}{2}+\frac{\rho \xi}{2}+\frac{3}{8}\mu_{s,1}^2\lambda^2+\frac{3}{8}\rho^2\xi^2+\frac{3}{4}\mu_{s,1}\rho \lambda \xi, \\
    & \notag f_b=-\mu_{s,1}\lambda-\rho \xi, \\
    & f_c=1+\frac{\mu_{s,1}\lambda}{2}+\frac{\rho \xi}{2}-\frac{3}{8}\mu_{s,1}^2\lambda^2-\frac{3}{8}\rho^2\xi^2-\frac{3}{4}\mu_{s,1}\rho \lambda \xi.
\end{align}
We now rewrite the static droplet equation as a coupled system of two first order ODEs in the following way:
\begin{equation}
    \label{kink2xilamb}
    \begin{cases}
    \displaystyle \frac{df}{d\tilde{r}}=\chi,\\[8pt]
    \displaystyle \frac{df}{d\tilde{r}}=\Psi^{\prime}(f)-\xi \chi +\lambda \chi^2 - \mu_{s,1}\lambda -\rho \xi.
    \end{cases}
\end{equation}
We observe that the points with coordinates $(f_a,0)$, $(f_b,0)$ and $(f_c,0)$ are critical points of system \eqref{kink2xilamb} in the phase plane $\{(f,\chi)\in \mathbb{R}^2\}$. 
A static droplet solution can be identified as an \textit{heteroclinic} trajectory in the phase plane, parametrized by the $\tilde{r}$ coordinate, which connects the two saddle points $(f_a,0)$ and $(f_c,0)$, i.e. which approaches $(f_a,0)$ for $\tilde{r}\to - \infty$ and $(f_c,0)$ for $\tilde{r}\to + \infty$. Let us recall that the static droplet solution $\phi(r)=f(r-R)$ is sharply peaked around $r=R$, going from $f_a$ to $f_c$ in a region of width $O(1)$; since $R>>1$, we can assume that $f\to f_a$ for $\tilde{r} \to - \infty$ and $f\to f_c$ for $\tilde{r} \to + \infty$. The details of the calculations of the heteroclinic trajectory and of the associated static droplet profile, expressed up to second order $O(\lambda^2,\xi^2,\lambda \xi)$ as
\begin{equation}
\label{f012345text}
f(\tilde{r})=f_0(\tilde{r})+\lambda f_1(\tilde{r})+\xi f_2(\tilde{r})+\lambda^2f_3(\tilde{r})+\xi^2 f_4(\tilde{r})+\lambda \xi f_5(\tilde{r})+o(\lambda^2)+o(\xi^2)+o(\lambda\xi),
\end{equation}
are reported in Appendix \ref{ap:shift}.
We are thus able to describe the local curvature effects of the interface for a static droplet solution up to order $O(\lambda^2,\xi^2,\lambda\xi)$. It is possible to prove that this static solution is unstable. Since this is a standard argument, we do not enter into details here (see e.g. \cite[Supplementary Note 2]{WITT}). Following \cite{Bray}, (see also \cite{PATTANAYAK2021}), we now characterize the interface dynamics at late times to obtain the domain growth law. 
The bulk domains are saturated to the equilibrium values $\phi_a$ and $\phi_c$, with small fluctuations that drive domain growth by bulk diffusion of the order parameter from interfaces of high curvature to regions of low curvature. 
Linearizing the equations around $\phi_a$, i.e. putting $\phi=\phi_a+\tilde{\phi}$, by standard arguments we obtain that $\Delta \tilde{\phi}=0$ in the bulk domains, and also $\Delta \mu=0$, since $\mu$ and $\tilde{\phi}$ are proportional in the bulk domains. The boundary conditions on $\mu$ at the interface determine the interface dynamics.
\noindent
As in the seminal work of Lifshitz and Slyozov \cite{LS}, we assume that the majority phase is supersaturated with the dissolved minority species, which have not yet reached equilibrium, and thus we may express the bulk equilibrium values as
\begin{equation}
    \label{supersat}
    \phi_a=-1+\frac{2}{15}\lambda+\frac{2}{75}\lambda^2, \quad \phi_c=1+\frac{2}{15}\lambda-\frac{2}{75}\lambda^2-\frac{\epsilon(t)}{2},
\end{equation}
where $\epsilon(t)$ is a time-dependent small supersaturation, with $\epsilon(t)<<1$. We also assume that the interface profile is almost equilibrated to the static droplet solution \eqref{f012345}. Note that this means that
\[
\lim_{\tilde{r}\to 0^-}\phi(\tilde{r})=f_a=- 1+\frac{2}{15}\lambda-\frac{\sqrt{2}}{3}\frac{1}{R}+\frac{2}{75}\lambda^2+\frac{1}{3R^2}-\frac{2\sqrt{2}}{15}\frac{\lambda}{R}=-1\left(1+\frac{1}{4}\Delta p\right)+O(\lambda^2,\xi^2,\lambda\xi),
\]
where $\Delta p$ is the excess pressure at the interface, given by the sum of the Laplace pressure $\Delta p_L=\frac{4\sqrt{2}}{3R}$ and the activity-induced pressure $\Delta p_{\gamma}=-\frac{8}{15}\lambda$, which drives the Ostwald ripening mechanism. At first order, we obtain the same expression as that reported in \cite[Supplementary Note 3]{WITT}.

\noindent
Near the interface, the equation for the chemical potential is
\begin{equation}
    \label{muinterface}
    \mu=\Psi'(\phi)-\frac{d^2\phi}{d\tilde{r}^2}-\frac{2}{\tilde{r}}\frac{d\phi}{d \tilde{r}}+\lambda \left(\frac{d\phi}{d \tilde{r}}\right)^2.
\end{equation}
The value of $\mu$ at the interface can be obtained by multiplying equation \eqref{muinterface} by $\frac{d \phi}{d \tilde{r}}$, which is sharply peaked at the interface, and integrating in $\tilde{r}$ from $-\infty$ to $+\infty$. We obtain that
\begin{equation}
    \label{muinterface2}
    \mu \Delta \phi= \Delta \Psi-\frac{2}{R}\sigma+\lambda \beta,
\end{equation}
where $\Delta \phi=\phi_c-\phi_a$, $\Delta \Psi=\Psi(\phi_c)-\Psi(\phi_a)$, $\sigma=\int_{-\infty}^{+\infty}\left(\frac{d \phi}{d \tilde{r}}\right)^2d\tilde{r}$ and $\beta=\int_{-\infty}^{+\infty}\left(\frac{d \phi}{d \tilde{r}}\right)^3d\tilde{r}$. Moreover, the value of $\mu$ at $+\infty$ is $\Psi^{\prime}(\phi_c)$.

\noindent
At order $O(\lambda^2,\xi^2,\lambda\xi,\epsilon)$, the previous quantities are
\[
\Delta \phi=2-\frac{4}{75}\lambda^2-\frac{\epsilon(t)}{2}, \quad \Delta \Psi=O(\lambda^3,\lambda \epsilon(t),\epsilon^2(t)), \quad \Psi^{\prime}(\phi_c)=\frac{4}{15}\lambda-\epsilon(t),
\]
which are obtained from \eqref{supersat},
\begin{align*}
\sigma&=\int_{-\infty}^{+\infty}\left(\left(\frac{df_0}{d\tilde{r}}\right)^2+\lambda^2\left(\left(\frac{df_1}{d\tilde{r}}\right)^2+2\frac{df_0}{d\tilde{r}}\frac{df_3}{d\tilde{r}}\right)+2\xi^2\frac{df_0}{d\tilde{r}}\frac{df_4}{d\tilde{r}}+2\lambda\xi\frac{df_0}{d\tilde{r}}\frac{df_5}{d\tilde{r}}\right)d\tilde{r}\\
&=\frac{2\sqrt{2}}{3}-\frac{103\sqrt{2}}{1125}\lambda^2-\frac{2\sqrt{2}}{3}\frac{1}{R^2}+\frac{2}{15}(4-\sqrt{2})\frac{\lambda}{R},
\end{align*}
which is expressed as the surface tension of the Passive Model B plus perturbative corrections coming from activity and from the interface curvature, and
\begin{align*}
\beta&=\int_{-\infty}^{+\infty}\left(\left(\frac{df_0}{d\tilde{r}}\right)^3+3\lambda^2\left(\frac{df_0}{d\tilde{r}}\left(\frac{df_1}{d\tilde{r}}\right)^2+\left(\frac{df_0}{d\tilde{r}}\right)^2\frac{df_3}{d\tilde{r}}\right)+3\xi^2\left(\frac{df_0}{d\tilde{r}}\right)^2\frac{df_4}{d\tilde{r}}+3\lambda\xi\left(\frac{df_0}{d\tilde{r}}\right)^2\frac{df_5}{d\tilde{r}}\right)d\tilde{r}\\
&=\frac{8}{15}-\frac{208}{1575}\lambda^2-\frac{8}{9}\frac{1}{R^2}+\frac{16}{225}(5\sqrt{2}-3)\frac{\lambda}{R}.
\end{align*}
The quantities $\sigma$ and $\beta$ are obtained using \eqref{f0}--\eqref{f5}. We observe that the latter quantities do not contain linear contributions in $\lambda$ and $\xi$: indeed, it is easy to prove that the latter contributions are null due to symmetry considerations. This is the reason why we have developed the perturbative analysis for the static droplet solution up to the second order in $\lambda$ and $\xi$.

\noindent
The right hand side of \eqref{muinterface} is expressed up to order $O(\lambda^2,\xi^2,\lambda\xi)$. When multiplied by the factor $\frac{d\phi}{d\tilde{r}}$, it acquires an extra $\frac{1}{R}$ factor. With the previous expressions, the formula \eqref{muinterface2} becomes
\begin{align*}
\mu\left(2-\frac{4}{75}\lambda^2-\frac{\epsilon(t)}{2}\right)&=\lambda\left(\frac{8}{15}-\frac{8}{9}\frac{1}{R^2}+\frac{16}{225}(5\sqrt{2}-3)\frac{\lambda}{R}\right)\\
&-\frac{2}{R}\left(\frac{2\sqrt{2}}{3}-\frac{103\sqrt{2}}{1125}\lambda^2-\frac{2\sqrt{2}}{3}\frac{1}{R^2}+\frac{2}{15}(4-\sqrt{2})\frac{\lambda}{R}\right)+O(\lambda^3,\lambda\epsilon(t),\epsilon^2(t)),
\end{align*}
which we rewrite as
\begin{equation}
\label{muinterface3}
\mu=\mu_{\beta}-\frac{2}{R}\mu_{\sigma},
\end{equation}
where we define
\begin{align*}
&\mu_{\beta}:=\frac{\lambda}{2-\frac{4}{75}\lambda^2-\frac{\epsilon(t)}{2}}\left(\frac{8}{15}-\frac{8}{9}\frac{1}{R^2}+\frac{16}{225}(5\sqrt{2}-3)\frac{\lambda}{R}\right), \\ 
&\mu_{\sigma}:=\frac{1}{2-\frac{4}{75}\lambda^2-\frac{\epsilon(t)}{2}}\left(\frac{2\sqrt{2}}{3}-\frac{103\sqrt{2}}{1125}\lambda^2-\frac{2\sqrt{2}}{3}\frac{1}{R^2}+\frac{2}{15}(4-\sqrt{2})\frac{\lambda}{R}\right).
\end{align*}
A solution of the Laplace equation $\Delta \mu=0$ with boundary condition \eqref{muinterface3} for $r=R$ and $\mu=\Psi^{\prime}(\phi_c)$ for $r\to +\infty$, and which is regular at the origin, is
\begin{equation}
\label{muinterface4}
\mu=
\begin{cases}
\Psi^{\prime}(\phi_c)+(\mu_{\beta}-\Psi^{\prime}(\phi_c))\frac{R}{r}-\frac{2}{r}\mu_{\sigma}, \quad \text{for} \; r\geq R,\\
\mu_{\beta}-\frac{2}{r}\mu_{\sigma}, \quad \text{for} \; r\leq R.
\end{cases}
\end{equation}
The movement of the interface, and hence the time rate $\dot{R}$, is determined by the imbalance between the current flowing into it and the current flowing out of it, which is proportional to the negative gradient of the chemical potential at the interface \cite{Bray,PATTANAYAK2021}. Hence we have that
\begin{align}
    \notag \dot{R}(t)&= -\frac{1}{\Delta \phi}\frac{d\mu}{dr}\biggr|_{r=R}=\frac{1}{\Delta \phi R}\left(\mu_{\beta}-\Psi'(\phi_c)-\frac{2}{R}\mu_{\sigma}\right)\\
    \notag &=\frac{1}{\Delta \phi R}\biggl[\left(\frac{8}{15}\frac{\lambda}{\Delta \phi}-\frac{4}{15}\lambda+\epsilon(t)\right)\\
    \label{rdot} & \quad +\left(\frac{(202\sqrt{2}-80)}{375}\frac{\lambda^2}{\Delta \phi}-\frac{4\sqrt{2}}{3\Delta \phi}\right)\frac{1}{R}-\frac{(88-12\sqrt{2})}{45}\frac{\lambda}{\Delta \phi}\frac{1}{R^2}+\frac{4\sqrt{2}}{3\Delta \phi}\frac{1}{R^3}\biggr].
\end{align}
In the case with $\lambda=0$, i.e. for the passive Model B, \eqref{rdot} becomes
\begin{equation}
    \label{rdot0}
    \dot{R}(t)=\frac{4\sqrt{2}}{3(\Delta \phi)^2}\frac{1}{R}\left(\bar{\epsilon}(t)-\frac{1}{R}+\frac{1}{R^3}\right),
\end{equation}
with $\bar{\epsilon}(t):=\frac{3\Delta \phi}{4\sqrt{2}}\epsilon(t)$. For $\bar{\epsilon}(t)$ sufficiently small, the right hand side of \eqref{rdot0} admits three real roots, one negative and two positive. The smaller positive root has a value $\sim 1$ and is a stable critical point for \eqref{rdot0}. As already said, lengths of magnitude $\sim 1$ are of the same order as the interface length, and they are not associated to the droplet size. The bigger positive root is $R(t)=\bar{R}(t)\sim \frac{1}{\bar{\epsilon}(t)}$, and it is an unstable critical point for \eqref{rdot0}, i.e. we have that $\dot{R}(t)<0$ for $R<\bar{R}(t)$ and $\dot{R}(t)>0$ for $R>\bar{R}(t)$. The latter property represents the classical Ostwald ripening mechanism: a droplet with a radius less than $\bar{R}(t)$ decreases, while a droplet with a radius greater than $\bar{R}(t)$ increases. As long as the system reaches equilibrium, i.e. as $\epsilon(t)\to 0$, the metastable radius $\bar{R}(t)$ increases in time to $+\infty$. The typical droplet size is comparable to $\bar{R}(t)$, and for a system of non-interacting droplets it can be identified as the value $R_0(t)$ defined in the following way: $R_0(t)$ is the initial droplet radius such that the droplet evaporates at time $t$, considering $\epsilon=0$ (see e.g. \cite{Bray,PATTANAYAK2021}). 
The latter quantity satisfies the equation 
\[
\dot{R}_0(t)=\frac{4\sqrt{2}}{3(\Delta \phi)^2}\left(\frac{1}{R_0^2}-\frac{1}{R_0^4}\right),
\]
with initial condition $R_0(0)>1$.
The latter equation implies that $R_0(t)\sim t^{\frac{1}{3}}$ as $t\to +\infty$, which is the LS domain growth law.

\noindent
In the case $\lambda>0$, for $\epsilon(t)$ and $\lambda$ sufficiently small the right hand side of \eqref{rdot} still admits a positive real root with magnitude $\sim 1$, which is a stable critical point for \eqref{rdot}, and a positive real root $\bar{R}(t)$ which is an unstable critical point for \eqref{rdot}. The new fact about active Model B is that $\bar{R}(t)$ tends to a finite constant, which depends non-linearly on $\lambda^{-1}$, as the system reaches equilibrium, i.e. as $\epsilon(t)\to 0$. This means that the typical domain size for a system of non-interacting droplets undergoing Ostwald ripening does not grow indefinitely with time, but it reaches a saturation value depending on $\lambda^{-1}$. Hence, the domain growth law may present power laws for finite times, but eventually reaches a plateau for $t\to +\infty$. This is a new and unexpected property of active Model B, which partially explains the shifts of exponents in the LS power law observed in \cite{WITT,PATTANAYAK2021}. In table \ref{tab:1} we report the approximate values for $\bar{R}(t)$, obtained as $\epsilon(t)\to 0$, calculated in the particular cases $\lambda=0.1,1,2$.
\begin{table}
\begin{center}
\begin{tabular}{ |c|c|c|c|} 
 \hline
   & $\lambda = 0.1$ &  $\lambda = 1$ & $\lambda = 2$ \\
 \hline
 & & &  \\
 $\bar{R}$ & $\approx 132252$ & $\approx 95.24$ & $\approx 3.65$ \\  
 \hline
\end{tabular}
\end{center}
\caption{Saturation values of $\bar{R}(t)$, as $\epsilon(t)\to 0$, in the particular cases $\lambda=0.1,1,2$.}
\label{tab:1}
\end{table}
The equation for $R_0(t)$, which as already said can be identified as the typical droplet size for a system of non-interacting droplets, is 
\begin{align}
    \notag \dot{R}_0(t)&= -\frac{1}{\Delta \phi R}\biggl[\left(\frac{8}{15}\frac{\lambda}{\Delta \phi}-\frac{4}{15}\lambda\right)\\
    \label{rzerolamb} & \quad +\left(\frac{(202\sqrt{2}-80)}{375}\frac{\lambda^2}{\Delta \phi}-\frac{4\sqrt{2}}{3\Delta \phi}\right)\frac{1}{R_0}-\frac{(88-12\sqrt{2})}{45}\frac{\lambda}{\Delta \phi}\frac{1}{R_0^2}+\frac{4\sqrt{2}}{3\Delta \phi}\frac{1}{R_0^3}\biggr].
\end{align}
In figure \ref{fig:growth1} we plot the values of $R_0$ which solve \eqref{rzerolamb} versus time for the cases $\lambda=0.1$ and $\lambda=1$, obtained by choosing $R_0(0)>1$ in order to avoid the stable point of the solution trajectory of \eqref{rzerolamb} at $\sim 1$, together with the power growth laws $R_0(t)\propto t^{\frac{1}{3}}$ and $R_0(t)\propto t^{\frac{1}{4}}$.
\begin{figure}[ht!]
\includegraphics[width=0.7\linewidth]
{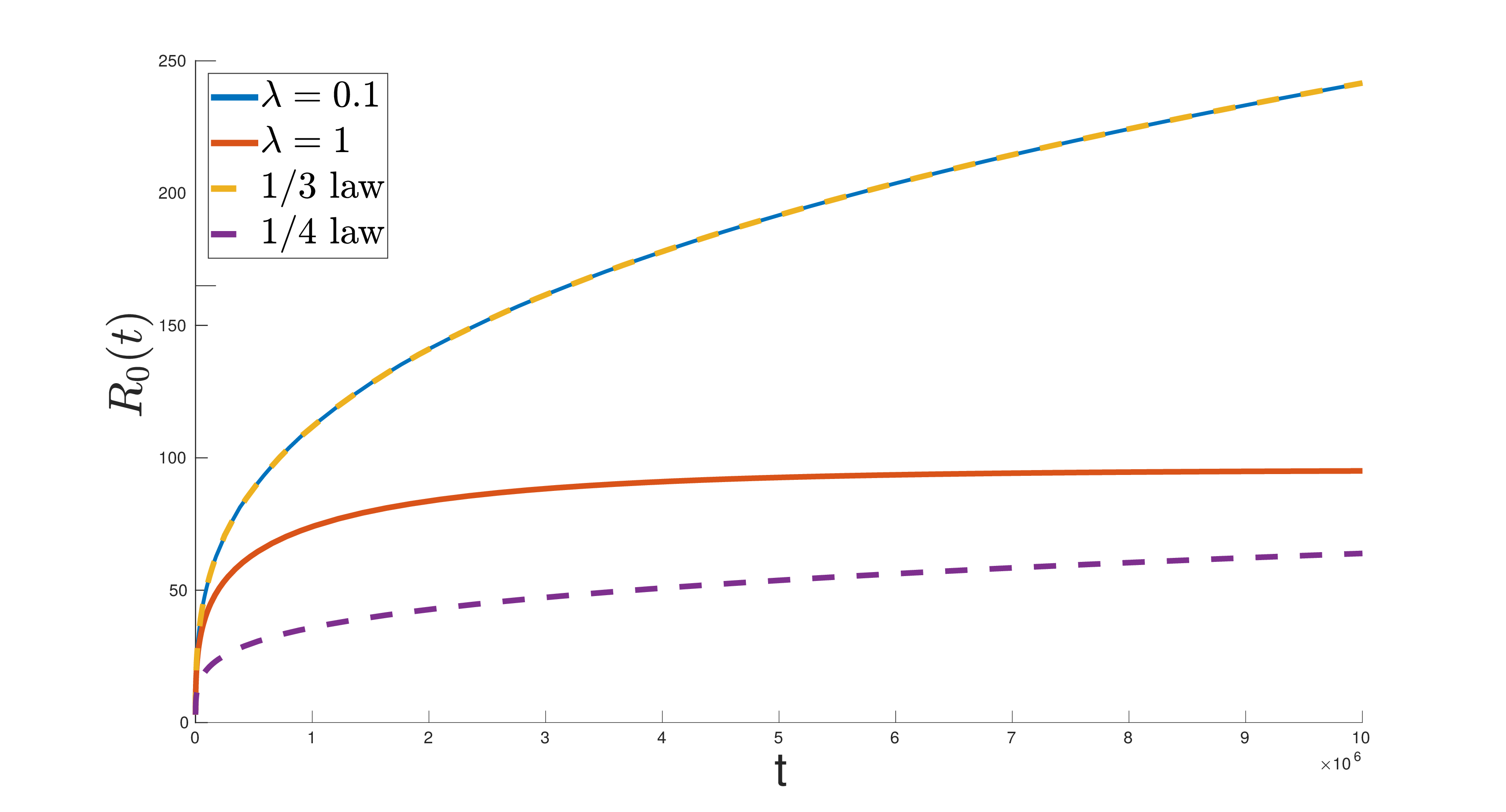}
\centering
\caption{Plots of the solution $R_0$ of \eqref{rzerolamb} versus time, with an initial condition $R_0(0)>1$, for the cases $\lambda=0.1$ and $\lambda=1$, together with the plots of the power growth laws $R_0(t)\propto t^{\frac{1}{3}}$ and $R_0(t)\propto t^{\frac{1}{4}}$.}
\label{fig:growth1}
\end{figure}
We observe that, in the time frame $t\in [0,10^6]$ reported in figure \ref{fig:growth1}, the profile of $R_0(t)$ for $\lambda=0.1$ follows a $ t^{\frac{1}{3}}$ growth law, which remains far from the saturation value, while for $\lambda=1$ it initially grows following a $t^{\frac{1}{3}}$ law, then shifts at later times to a $t^{\frac{1}{4}}$ profile, and lately it reaches a plateau at the saturation value. We highlight this behavior in figure \ref{fig:growth2}, where we plot the $\log (R_0(t))$ versus $\log (R(0)^3+t)$ for the solution of of \eqref{rzerolamb} with $\lambda=1$ and for the power growth laws $R_0(t)\propto t^{\frac{1}{3}}$ and $R_0(t)\propto t^{\frac{1}{4}}$.  
\begin{figure}[ht!]
\includegraphics[width=0.7\linewidth]
{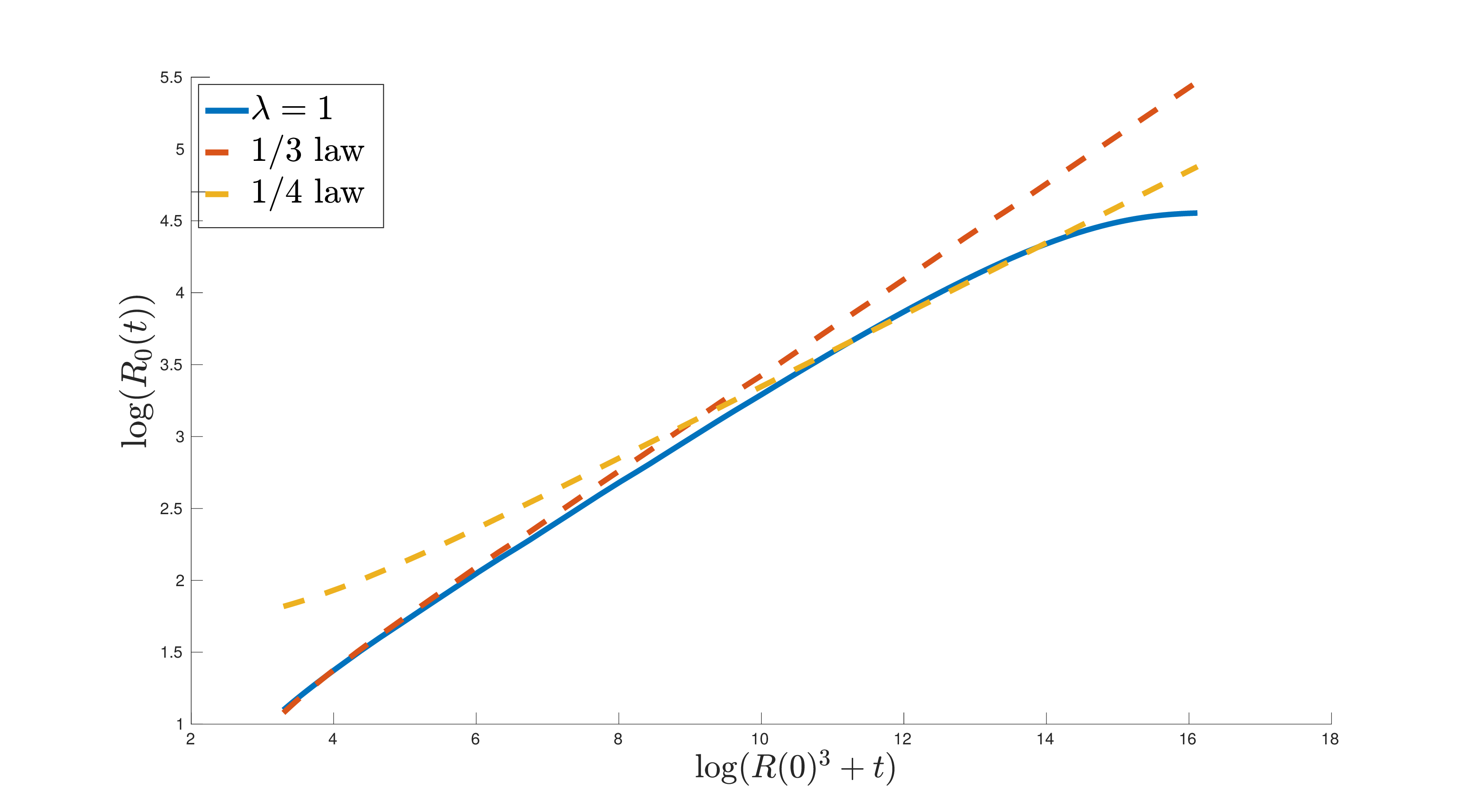}
\centering
\caption{Plots of the $\log (R_0(t))$ versus $\log (R(0)^3+t)$ for the solution of \eqref{rzerolamb} with $\lambda=1$ and for the power growth laws $R_0(t)\propto t^{\frac{1}{3}}$ and $R_0(t)\propto t^{\frac{1}{4}}$.}
\label{fig:growth2}
\end{figure}
We observe from figure \ref{fig:growth2} that at finite times the profile of $R_0(t)$ for $\lambda=1$ shifts from the power growth law $t^{\frac{1}{3}}$ at early times to $t^{\frac{1}{4}}$ at later times, as observed and conjectured in \cite{WITT,PATTANAYAK2021}, and asymptotically reaches a saturation plateau.
\begin{remark}
We observe that the term
\[
\xi_0:=\frac{8}{15}\frac{\lambda}{\Delta \phi}-\frac{4}{15}\lambda
\]
in the first round brackets on the right hand side of \eqref{rzerolamb} is of order $O(\lambda^3)$, since $\Delta \phi=2-\frac{4}{75}\lambda^2$ (for $\epsilon(t)=0$). This term is the $O(R^{0})$ contribution of $\mu_{\beta}-\Psi^{\prime}(\phi_c)$, where $\frac{8}{15}\frac{\lambda}{\Delta \phi}$ comes from $\mu_{\beta}$, which is defined in \eqref{muinterface3} and obtained from the $O(\lambda^2,\xi^2,\lambda\xi,\epsilon)$ expression of $\mu$ in \eqref{muinterface}, while $\frac{4}{15}\lambda$ comes from the $O(\lambda)$ contribution in $\Psi^{\prime}(\phi_c)$. In order to appropriately represent $\xi_0$ at the leading order, we should consider also the $O(\lambda^3)$ contribution in $\Psi^{\prime}(\phi_c)$. We report here for the sake of completeness the results obtained by considering the latter contribution, without reporting the undergoing calculations which are extremely long. We highlight that these results do not change the qualitative behavior of the growth law derived in the text; they only change the approximate values of $\bar{R}$ reported in Table \ref{tab:1}. It is possible to prove that, at order $O(\lambda^3)$, we have, for $\epsilon(t)=0$, that
\[
\phi_c=1+\frac{2}{15}\lambda-\frac{2}{75}\lambda^2+\frac{1}{2}\left(\mu_{s,1}^3+\mu_{s,3}\right)\lambda^3,
\]
and 
\[
\Psi^{\prime}(\phi_c)=\mu_{s,1}\lambda+\mu_{s,3}\lambda^3,
\]
with $\mu_{s,1}=\frac{4}{15}$ and $\mu_{s,3}=-\frac{139}{3780}$. The latter value is obtained by considering \eqref{muslamb} and \eqref{heterocliniclambda} up to order $\lambda^3$. Hence, we have that
\[
\xi_0:=\frac{8}{15}\frac{\lambda}{\Delta \phi}-\frac{4}{15}\lambda+\frac{139}{3780}\lambda^3\sim \frac{4}{15}\frac{\lambda}{\Delta}+\frac{8}{1125}\lambda^3-\frac{4}{15}\lambda+\frac{139}{3780}\lambda^3=\frac{4147}{94500}\lambda^3.
\]
\end{remark}

\section{Formal estimates for the case with singular potential}
In this section, we derive formal a-priori estimates for solutions to \eqref{ac-CH}-\eqref{ac-CH-bc}, which will be rigorously justified in the context of the finite element approximation of the problem in Section $4$. They laid the basis for proving existence results via the convergence analysis of the discretized problems. In particular, in the case with singular potential and spatial dimension $d=2,3$, assuming a proper smallness condition on $\lambda$ and the existence of weak solutions which satisfy the physical bound $\phi \in [-1,1]$ a.e. in space and time, we are able to prove local-in-time high order regularity via a-priori estimates. In the context of our finite element and time discrete approximation, thanks to inverse inequalities and energy estimates we will prove the well-posedness of the discrete scheme, both in the cases with regular and singular potentials, and, in the singular case, the property that the discrete solution remains in the interval $(-1,1)$ at all times. Thanks to this property, and to some technical lemmas regarding piecewise linear finite elements and time discrete semi-implicit schemes, some of which we will derive here for the first time, we will rigorously obtain high order a-priori estimates. Then, by convergence analysis of the discrete scheme, they will let us prove the existence of a local in time weak solution for active Model B in Theorem \ref{thm:limitpoint}. Moreover, in the case with singular potential and spatial dimension $d=1$, we will show here by further a-priori estimates that the local in time weak solution is actually global in time.
\newline

\subsection{Notation and preliminary lemmas}
Let $\Omega \subset \mathbb{R}^d$, $d=1,2,3$, be an open bounded domain with smooth boundary, with $\Omega=(0,L)$ in the case $d=1$. For a normed space $X$, the associated norm and seminorm are denoted by $\lVert \cdot \rVert_X$ and $\lvert \cdot \rvert_X$ respectively. The dual space of a Banach space $X$ is denoted by $X'$.
We denote by $L^p(\Omega;K)$ and $W^{r,p}(\Omega;K)$ the standard Lebesgue and Sobolev spaces of functions defined on $\Omega$ with values in a set $K$, where $K$ may be $\mathbb{R}$ or a multiple power of $\mathbb{R}$. If $K\equiv \mathbb{R}$,
we simply write $L^p(\Omega)$ and $W^{r,p}(\Omega)$.  In the case $p = 2$, we use the notations $H^1 := W^{1,2}$ and
$H^2 := W^{2,2}$, and we denote by $(\cdot,\cdot)$ and $\lVert \cdot \rVert$ the $L^2$ scalar product and induced norm between
functions with scalar or vectorial values. For $f\in L^1(\Omega)$, we define $\bar{f}:=|\Omega|^{-1}(f,1)$.
The duality pairing between $H^1(\Omega; K)$ and $(H^1(\Omega; K))'$ is denoted by $< \cdot, \cdot>$.

For any $1\leq p \leq \infty$, we denote by $L^p(0, T; V)$ and $L_{\uloc}^p([0,\infty); V)$ the Bochner
spaces of functions from $(0,T)$ and $[0,\infty)$, respectively, with values in a Banach space $V$. Their norms are defined as
$$
\| f\|_{L^p(0,T; X)}:= \left(\int_0^T \|f(s)\|^p_X \mathrm{d}s \right)^\frac{1}{p}
\quad \text{and} \quad
\|f\|_{L^p_{\rm uloc}([0,\infty); X)}:= \sup_{t\ge 0}\left( \int_t^{t+1}\|f(s)\|^p_X \mathrm{d}s\right)^\frac{1}{p}.
$$
We will also use the notation $C^k(\overline{\Omega}; K)$ and $C^k ([0, T]; V)$, $k\geq 0$, for the spaces of continuously differentiable functions up to order $k$ defined on $\Omega$ with values in a set $K$ or
from $[0, T]$ to the space $V$, respectively.

In the following, $C$ denotes a generic positive constant independent of the unknown variables, the discretization and the regularization parameters, the value of which might change from line to line. $C_1, C _2 , \dots$ indicate generic positive constants whose particular value must be tracked through the calculations.
$C(a, b, . . . )$ denotes a constant depending on the nonnegative parameters $a, b, \dots$.

\medskip

We introduce the \textit{inverse Laplacian operator} $\mathcal{G}:
\noindent\mathcal{F} \to V$, such that
\[
(\nabla \mathcal{G}v, \nabla \eta) =< v, \eta > \quad \forall \, \eta \in H^1(\Omega),
\]
where $\mathcal{F} := \{v \in (H^1(\Omega))': \ < v, 1 >= 0\}$ and $V := \{ v \in H^1(\Omega) : (v, 1) = 0 \}$. For any $v \in \mathcal{F}$, the existence and uniqueness of an element
$\mathcal{G}v \in V$ follows from the Lax--Milgram Theorem. Then, we can define a norm on $\mathcal{F}$ by setting $\lVert v\rVert_{\mathcal{F}}:=\lVert \nabla \mathcal{G}v\rVert$ for any $v\in \mathcal{F}$.
\medskip

Let us recall the following form of the Gagliardo--Nirenberg inequality \cite{Leoni,Brezis}.
\begin{lemma}
\label{lem:gnexp}
    Let $\Omega \subset \mathbb{R}^d$, $d=1,2,3$, be a bounded domain with Lipschitz boundary and $f\in W^{m,r}(\Omega)\cap L^q(\Omega)$, $1\leq q\leq +\infty$, $1\leq r\leq +\infty$.
    For any integer $j$ with $0 \leq j < m$, suppose there are $p\geq 1$ and $\alpha \in \mathbb{R}$ such that
    \begin{equation}
    \label{gnexp}
    j-\frac{d}{p}=\left(m-\frac{d}{r}\right)\alpha+(1-\alpha)\left(-\frac{d}{q}\right), \;\; \frac{j}{m}\leq \alpha \leq 1.
    \end{equation}
    Then, there exists a positive constant $C$ depending on $\Omega, d, m, j, q, r, \alpha$, such that
    \begin{equation}
    \label{gngeneral}
    \lVert D^jf\rVert_{L^p(\Omega)}\leq C\lVert f\rVert_{W^{m,r}(\Omega)}^{\alpha}\lVert f\rVert_{L^q(\Omega)}^{1-\alpha}.
    \end{equation}
\end{lemma}
\noindent
The form of the Gagliardo--Nirenberg inequality which will be relevant in the sequel is obtained for $j=1, m=r=2$, $d\leq 3$, $p=4$ and $q=+\infty$, i.e.
\begin{equation}
    \label{gn4}
    \lVert \nabla f\rVert_{L^4(\Omega)}\leq C\lVert f\rVert_{H^2(\Omega)}^{\frac{1}{2}}\lVert f\rVert_{L^{\infty}(\Omega)}^{\frac{1}{2}} \quad \forall \; f \in H^2(\Omega).
\end{equation}

\subsection{Formal estimates for $d=2,3$.}
Let us assume that $\phi_0\in H^1(\Omega)\cap L^\infty(\Omega)$ with $|\phi_0|\leq 1$ a.e. in $\Omega$ and $\bar{\phi}_0\in (-1,1)$,
and that $\mu_0\in H^1(\Omega)$, where $\mu_0=-\Delta \phi_0+\Psi'(\phi_0)+\lambda |\nabla \phi_0|^2$.
Let us moreover formally assume that there exist sufficiently regular solutions of \eqref{ac-CH}-\eqref{ac-CH-bc} such that $\phi \in (-1,1)$ a.e. in $\Omega \times (0,T)$, for a given $T>0$. We will use the notation $\Psi(\phi)=F(\phi)-\frac{\theta_0}{2}\phi^2$, where $F(\cdot)$ is the convex singular part in \eqref{log}.
We derive the following a-priori estimates. 

\medskip

\textbf{First estimate.}
Let us multiply \eqref{ac-CH}$_1$ by $\phi$ and \eqref{ac-CH}$_2$ by $\Delta \phi$, integrate over $\Omega$ and sum the contributions. Upon integration by parts and implementation of the boundary conditions \eqref{ac-CH-bc}, we obtain 
  \begin{equation*}
    \frac{1}{2}\frac{\mathrm{d}}{\mathrm{d}t}\|\phi\|^2 + \|\Delta \phi\|^2 
    + \int_\Omega F''(\phi)|\nabla \phi|^2 \, \mathrm{d}x
    = -\theta_0 \int_{\Omega}\phi \Delta \phi \, \mathrm{d}x
    + \lambda \int_\Omega |\nabla \phi|^2 \Delta \phi \, \mathrm{d}x 
  \end{equation*}
  Thanks to \eqref{gn4}, to elliptic regularity and to the fact that $\phi\in (-1,1)$ almost everywhere, we have that 
  \[
  |\lambda| \int_\Omega |\nabla \phi|^2 |\Delta \phi| \, \mathrm{d}x 
  \leq |\lambda| {\| \nabla \phi\|_{L^4}^2}\| \Delta \phi\|\leq |\lambda| {C_1\| \Delta \phi\|^2},
  \]
  where $C_1$ only depends on $\Omega$. 
 Hence, using the Cauchy--Schwarz and Young inequalities, together with the property $\phi \in (-1,1)$, in the first term on the right-hand side, as well as the convexity of $F(\cdot)$, we infer that
  \begin{equation*}
    \frac{1}{2}\frac{\mathrm{d}}{\mathrm{d}t}\|\phi\|^2 + {(1-\epsilon-|\lambda| C_1)}\|\Delta \phi\|^2 \leq C,
  \end{equation*}
  for any $\epsilon >0$, where $C$ depends now on $\epsilon$. 
  Integrating the previous inequality in time over the interval $(0,T)$, using the assumptions on $\phi_0$  and assuming that $|\lambda|<\frac{1}{C_1}$, we can choose $\epsilon$ in a suitable way such that
  \[
  \lVert \phi \rVert_{L^{\infty}(0,T;L^2(\Omega))}\leq C, \quad \text{and} \quad \lVert \Delta \phi \rVert_{L^2(\Omega \times (0,T))}\leq C.
  \]
 Moreover, multiplying \eqref{ac-CH}$_2$ by $-\Delta \phi$ and integrating over $\Omega$, there exists $\omega:=1-\epsilon-|\lambda| C_1>0$ such tat 
  \begin{equation}
  \label{fe1}
   {\omega}\|\Delta \phi\|^2 \leq \|\nabla \mu\|\|\nabla \phi\| + \theta_0 \| \nabla \phi \|^2.
  \end{equation}

\medskip

\textbf{Second and Third estimates.}
Let us multiply \eqref{ac-CH}$_1$ by $\mu$ and \eqref{ac-CH}$_2$ by $-\partial_t \phi$, integrate over $\Omega$ and sum the contributions. We obtain that
 \begin{align}
     \frac{\mathrm{d}}{\mathrm{d}t} \left(\frac{1}{2}\|\nabla \phi\|^2+ \int_\Omega \Psi(\phi) \, \mathrm{d}x \right) 
    + \|\nabla \mu \|^2 
    &= - \lambda \int_\Omega |\nabla \phi|^2 \partial_t \phi \, \mathrm{d}x.
     \label{fe2}
\end{align}
By using the Poincar\'{e}--Wirtinger inequality, together with the fact that $(\partial_t \phi,1)=0$, the Cauchy--Schwarz and Young inequalities, as well as \eqref{gn4} and \eqref{fe1}, we deduce that
\begin{align}
\left| \lambda \int_\Omega |\nabla \phi|^2 \partial_t \phi \, \mathrm{d}x \right|
& \leq \frac{1}{4}\|\nabla \partial_t \phi\|^2+C\|\Delta \phi\|^2
\notag
\\
    & \leq \frac{1}{4}\|\nabla \partial_t \phi\|^2+C\|\nabla \mu\|
    \|\nabla \phi\| 
    + C \|\nabla \phi\|^2.
    \label{fe2-2}
  \end{align}

Next, multiplying \eqref{ac-CH}$_1$ by $\partial_t \mu$, and the time derivative of \eqref{ac-CH}$_2$ by $-\partial_t \phi$, integrating over $\Omega$ and summing the contributions, we have
\begin{equation*}
    \frac{1}{2}\frac{\mathrm{d}}{\mathrm{d}t} \|\nabla \mu\|^2+ \|\nabla \partial_t \phi \|^2+ \underbrace{\int_\Omega F''(\phi)|\partial_t \phi|^2 \, \mathrm{d}x}_{\geq 0} = \underbrace{\theta_0 \|\partial_t \phi\|^2}_{I_1}-\underbrace{2\lambda \int_\Omega \nabla \partial_t \phi\cdot \nabla \phi \, \partial_t \phi \, \mathrm{d}x}_{I_2}. 
  \end{equation*}
   Observe that
   \[
   I_1=\theta_{0}(\nabla \mathcal{G}\partial_t \phi,\nabla \partial_t \phi)\leq \theta_0\|\partial_t \phi\|_{\left(H^1(\Omega)\right)'}\|\nabla \partial_t \phi\|.
   \]
   From comparison in \eqref{ac-CH}$_1$, we get that $\|\partial_t \phi\|_{\left(H^1(\Omega)\right)'}\leq \|\nabla \mu\|$. Hence, using the Cauchy--Schwarz and Young inequalities, we infer that
   \[
   I_1\leq \frac{1}{8}\|\nabla \partial_t \phi\|^2+C\|\nabla \mu\|^2.
   \]
   For what concerns $I_2$, using \eqref{gngeneral} with $j=0, p=4, m=1, r=2, q=2$, the Poincar\'{e}--Wirtinger, Cauchy--Schwarz and Young inequalities, we obtain that 
   \begin{align*}
       & |I_2|\leq 2|\lambda|\|\nabla \partial_t \phi\|\|\partial_t \phi\|_{L^4}\|\nabla \phi\|_{L^4}\leq 2|\lambda|\|\nabla \partial_t \phi\|^{\frac{4+d}{4}}\|\partial_t \phi\|^{\frac{4-d}{4}}\|\Delta \phi\|^{\frac12}.
   \end{align*}
   As for the bound for $I_1$, by exploiting that $\lVert \partial_t \phi \rVert\leq \lVert \nabla \mu\rVert^{\frac12}\lVert \nabla \partial_t \phi\rVert^{\frac12}$, we then find 
   \begin{align*}
       & |I_2|
       \leq 
       2|\lambda|\|\nabla \partial_t \phi\|^{\frac{12+d}{8}}
       \|\nabla \mu\|^{\frac{4-d}{8}}
       \|\Delta \phi\|^{\frac12}
       \leq \frac{1}{8}\|\nabla \partial_t \phi\|^2
       +C\|\nabla \mu\|^2\|\Delta \phi\|^{\frac{8}{4-d}}.
   \end{align*}
   Collecting these results and using \eqref{fe1}, we conclude that
\begin{equation}
\label{fe3}
    \frac{1}{2}\frac{\mathrm{d}}{\mathrm{d}t} \|\nabla \mu\|^2+ \frac{3}{4}\|\nabla \partial_t \phi \|^2\leq C\|\nabla \mu\|^2 \left( 1+ \| \nabla \phi\|^\frac{8}{4-d} \right)
    +C\|\nabla \mu\|^{\frac{12-2d}{4-d}}
    \|\nabla \phi\|^{\frac{4}{4-d}}. 
  \end{equation}
 Let us set 
\[
H(t):=\frac{1}{2}\|\nabla \phi(t)\|^2+\frac{1}{2}\|\nabla \mu(t)\|^2.
\]
Summing \eqref{fe2} and \eqref{fe3}, and integrating in time over the interval $(0,t)$, with $t<T$,
we end up with 
  \begin{equation}
  \label{fe23}
   H(t)
   + \frac12 \int_0^t \|\nabla \mu\|^2+\|\nabla \partial_t\phi\|^2 \mathrm{d}s 
   \leq C+H(0)+\left|\int_\Omega\Psi(\phi_0)\, \mathrm{d}x \right|
   +C\int_0^t\left( 1+H^{\frac{8-d}{4-d}}(s)\right)\, \mathrm{d}s. 
  \end{equation}
Here we have used that $\Psi(\cdot)$ is bounded from below. We observe that $\frac{8-d}{4-d}>1$. Hence,
thanks to the Bihari inequality \cite{Bihari} and to the assumptions on $\phi_0$ and $\mu_0$, we conclude that there exists a $0<\overline{T}<T$, which depends only on $\Omega$ and on the data of the problem, such that 
$$ 
\sup_{t\in (0,\overline{T})}H(t)\leq C.
$$

In addition, a further estimate can be obtained by integrating \eqref{ac-CH}$_2$ over $\Omega$ and using the classical monotonicity argument of the singular nonlinear part (see below \eqref{va-F'}. In particular, this entails a bound on $\overline{\mu}$ in $L^{\infty}(0,\overline{T})$. 
We do not reproduce this estimate here since it will be detailed in Section $4$ for the discrete system (see also below for the case $d=1$).
Owing to this, together with \eqref{fe1}, and exploiting the convergence of mass $\bar{\phi}=\bar{\phi_0}$ and \eqref{gn4}, we conclude that
 \begin{equation}
 \label{fe23b}
  \lVert \phi \rVert_{L^{\infty}(0,\overline{T};H^2(\Omega))\cap W^{1,\infty}(0,\overline{T};H^1(\Omega))}\leq C, \quad \text{and} \quad \lVert \mu \rVert_{L^{\infty}(0,\overline{T};H^1(\Omega))}\leq C.
  \end{equation}

\medskip

\textbf{Uniqueness.}
Let $(\phi_1,\mu_1)$ and $(\phi_2,\mu_2)$ be two local solutions satisfying \eqref{fe23b}, and originating from the same initial datum. 
We define $\Phi:=\phi_1-\phi_2$, $\Sigma:=\mu_1-\mu_2$. We multiply the difference of the equations \eqref{ac-CH}$_1$ satisfied by $\phi_1$ and $\phi_2$, respectively, by $\mathcal{G}\Phi$, and integrate over $\Omega$. Since $(\Phi,1)=0$, by using the relations of $\mu_1$ and $\mu_2$, we obtain that
  \begin{equation*}
    \frac{1}{2}\frac{\mathrm{d}}{\mathrm{d}t}
    \|\Phi\|^2_{\left(H^1(\Omega)\right)'}+ \|\nabla \Phi \|^2+\underbrace{\left(F'(\phi_1)-F'(\phi_2),\Phi\right)}_{\geq 0}= \theta_0\lVert \Phi \rVert^2+ \lambda
    \int_\Omega \nabla \Phi \cdot \nabla \left( \phi_1+ \phi_2\right) \Phi \, \mathrm{d}x. 
    \end{equation*}
    We control the last term on right-hand side by interpolation as follows
    \begin{align*}\left|
    \lambda
    \int_\Omega \nabla \Phi \cdot \nabla \left( \phi_1+ \phi_2\right) \Phi \, \mathrm{d}x
    \right|
    &\leq C\|\nabla(\phi_1+\phi_2)\|_{L^4(\Omega)}\|\Phi\|_{L^4(\Omega)}\lVert \nabla \Phi\rVert
    \\
    &\leq C\|\nabla(\phi_1+\phi_2)\|_{L^4(\Omega)}
    \|\Phi\|^\frac{4-d}{4}\lVert \nabla \Phi\rVert^{\frac{4+d}{4}}.
  \end{align*}
  Here we have used \eqref{gngeneral} with $j=0, p=4, m=1, r=2, q=2$. By the interpolation inequality $\lVert \Phi\rVert\leq \lVert \Phi\rVert_{\left(H^1(\Omega)\right)'}^{\frac12}\lVert \nabla \Phi \rVert^{\frac12}$, we arrive at
 \begin{equation}
 \label{fe4}
    \frac{1}{2}\frac{\mathrm{d}}{\mathrm{d}t}
    \|\Phi\|^2_{\left(H^1(\Omega)\right)'}+ \frac{1}{2}\|\nabla \Phi \|^2\leq C(1+\|\nabla(\phi_1+\phi_2)\|_{L^4(\Omega)}^{\frac{16}{4-d}})\|\Phi\|_{\left(H^1(\Omega)\right)'}^2. 
  \end{equation}
Then, integrating in time over the interval $(0,t)$, for $t\leq \bar{T}$, thanks to \eqref{fe23b} and to the Gronwall inequality, we conclude that the solution is unique.

\subsection{The one dimensional case: global well-posedness}

We now consider the one dimensional case. We are able to prove that the local-in-time weak solution
is actually global in time. 

Setting $\Omega=(0,L)$ for some $L>0$, the model reads as follows
\begin{equation}  
\label{ac-CH-1D}
\partial_t \phi = \partial_{xx} \mu, \quad \mu = -\partial_{xx} \phi +\Psi^{\prime}(\phi)+ \lambda |\partial_x \phi|^2  \quad 
\text{in } (0,L) \times (0,\infty),
\end{equation}
which is equipped with the following boundary and initial conditions 
\begin{equation} 
 \label{ac-CH-1D-bc}
\partial_x \phi (0) = \partial_x \phi (L)= \partial_x \mu(0)= \partial_x \mu(L)=0,\quad \text{and} \quad \phi|_{t=0}= \phi_0 \quad \text{in } (0,L).
\end{equation}
To state our result, we introduce $\widetilde{\Psi}(s)=\Psi(s)+C_\Psi$, where $C_\Psi$ is a non-negative constant, such that $\widetilde{\Psi}(s)\geq0$ for any $s \in [-1,1]$ and $\widetilde{\Psi}(s^{\pm})=0$ where $s^{\pm}$ are the two global minima of $\Psi(s)$.

Our main result reads as follows
\begin{theorem}
\label{theorem 1d}
For any $R>0$, there exists $\lambda_0=\lambda_0(R, L, \theta_0)>0$ defined in \eqref{lambda-val} such that, for any $0<\lambda\leq \lambda_0$ and any $\phi_0 \in H^1(0,L)$ with $\overline{\phi_0}=0$, $\| \phi_0\|_{L^\infty(0,L)}\leq 1$ and 
$$
\int_0^L \frac12 |\partial_x \phi_0|^2 + \widetilde{\Psi}(\phi_0) \, \d x\leq R,
$$
there exists a unique global weak solution $\phi: \Omega \times [0,\infty) \to [-1,1]$ to \eqref{ac-CH-1D}-\eqref{ac-CH-1D-bc} in the following sense:
\begin{itemize}
    \item[(i)] The global weak solution satisfies the regularity
    \begin{equation}
        \begin{split}
            &\phi \in L^\infty(0,\infty; H^1(0,L)) \cap L_{\uloc}^4([0,\infty); H^2(0,L))\cap H_{\uloc}^1([0,\infty);H^1(0,L)'),
            \\
            &\phi \in L^\infty((0,L) \times (0,\infty)) \text{ with } 
            |\phi(x,t)|<1 \text{ a.e. in } (0,L) \times (0,\infty),
            \\
            & \mu= -\partial_{xx} \phi +\Psi^{\prime}(\phi)+ \lambda |\partial_x \phi|^2 \in L_{\rm uloc}^2([0,\infty); H^1(0,L)),
        \end{split}
    \end{equation}

    \item[(ii)] For any $v \in H^1(0,L)$, the variational formulation
    \begin{equation}
        \l \partial_t \phi, v\r+ (\nabla \mu, \nabla v)=0
    \end{equation}
    holds almost everywhere in $(0,\infty)$,
    
    \item[(iii)] $\phi(\cdot,0)=\phi_0$ almost everywhere in $\Omega$.
\end{itemize}
\end{theorem}

\begin{remark}
  For the sake of presentation, we consider an initial condition $\phi_0$ with mass $\overline{\phi_0}=0$ in Theorem \ref{theorem 1d}. Nevertheless, the analysis can be easily extended to the case $\overline{\phi_0}\in (-1,1)$.
\end{remark}

\begin{proof}
We aim to perform some preliminary formal estimates by assuming the existence of a regular solution to the model \eqref{ac-CH-1D}-\eqref{ac-CH-1D-bc} in $(0,L)\times (0,T)$. This entails that $\| \phi\|_{L^\infty((0,L) \times (0,T))}\leq 1$. 
Moreover, since the total mass is conserved, we assume without loss of generality that $\overline{\phi(t)}=\overline{\phi_0}=0$ for any $t\geq 0$.

\medskip

\textbf{First estimate.} 
Multiplying \eqref{ac-CH-1D}$_1$ by 
$\phi$ and integrating over $(0,L)$, we obtain
\begin{equation}
\begin{split}
&\ddt \left( \frac12 \int_0^L |\phi|^2 \, \d x \right)
+\int_0^L |\partial_{xx} \phi|^2 \, \d x
+ \int_0^L F''(\phi) |\partial_x \phi|^2 \, \d x
\\
&\quad = 
\theta_0 \int_0^L |\partial_x \phi|^2 \, \d x 
+\lambda \int_0^L |\partial_x \phi|^2 \partial_{xx} \phi \, \d x.
\end{split}
\end{equation}   
Concerning the first term on the right-hand side, we notice that
\begin{equation}
\label{h1-bound-1D}
\begin{split}
\theta_0 \int_0^L |\partial_x \phi|^2 \, \d x 
&= - \theta_0 \int_0^L \phi \, \partial_{xx} \phi \, \d x
\\
&\leq  \theta_0 \| \phi\|_{L^2(0,L)} \| \partial_{xx}\phi\|_{L^2(0,L)}
\leq \frac12 \int_0^L |\partial_{xx} \phi|^2 \, \d x + 
\frac{\theta_0^2}{2}\int_0^L |\phi|^2 \, \d x.
\end{split}
\end{equation}
On the other hand, we observe that 
\begin{equation}
\label{1D-trick}
 \int_0^L |\partial_x \phi|^2 \partial_{xx} \phi \, \d x
 = \frac13 \int_0^L \partial_x \left( (\partial_x \phi)^3 \right) \, \d x
 =0.
\end{equation}
Thus, we end up with 
\begin{equation}
\label{EST0-1D}
\begin{split}
&\ddt \left( \frac12 \int_0^L |\phi|^2 \, \d x \right)
+\frac12 \int_0^L |\partial_{xx} \phi|^2 \, \d x
+ \int_0^L F''(\phi) |\partial_x \phi|^2 \, \d x
\leq  
\frac{\theta_0^2}{2}\int_0^L |\phi|^2 \, \d x.
\end{split}
\end{equation}  
An application of the Gronwall lemma gives 
\begin{equation}
\label{EST1-1D}
 \int_0^L | \phi(t) |^2 \, \d x \leq 
 \left(  \int_0^L |\phi_0|^2 \, \d x\right) \mathrm{e}^{\theta_0^2 t}, \quad \forall \, t \in [0,T],
\end{equation}
and 
\begin{equation}
\int_0^T \int_0^L |\partial_{xx}\phi|^2 \, \d x \d \tau \leq 
 \left(  \int_0^L |\phi_0|^2 \, \d x\right) \e^{\theta_0^2 T}.
\end{equation}
In particular, we have
\begin{align}
\label{REG1-1D}
&\phi \in L^\infty(0,T; L^2(\Omega))\cap L^2(0,T;H^2(\Omega)).
\end{align}
In addition, by exploiting the global boundedness of any solution (due to the logarithmic potential), the global bound in \eqref{EST1-1D} can be enhanced as follows. First, we report the one dimensional Poincar\'{e} inequality for zero-mean functions (see \cite[Page 233]{BREZIS2010})
\begin{equation}
\label{PI-1D}
\| f\|_{L^\infty(0,L)}\leq \|\partial_x f\|_{L^1(0,L)}, \quad \forall \, f \in W^{1,1}(0,L), \ \overline{f}=0, 
\end{equation}
which simply entails that 
 \begin{equation}
\label{PI-1D-2}
\|f\|_{L^2(0,L)}\leq L \|\partial_x f \|_{L^2(0,L)}, \quad \forall \, f \in W^{1,2}(0,L), \ \overline{f}=0.
\end{equation}
Since $\| \phi\|_{L^\infty((0,L) \times (0,T))}\leq 1$, and $F''(s)\geq \theta$ for any $s \in (-1,1)$, we deduce from \eqref{EST0-1D} that
\begin{equation}
\label{EST3-1D}
\begin{split}
&\ddt \left(  \int_0^L |\phi|^2 \, \d x \right)
+ \int_0^L |\partial_{xx} \phi|^2 \, \d x
+\frac{2 \theta}{L^2} \int_0^L |\phi|^2 \, \d x
\leq  
\theta_0^2 L.
\end{split}
\end{equation}  
Hence, the Gronwall lemma entails that
\begin{equation}
\label{EST4-1D}
\int_0^L |\phi(t)|^2 \, \d x \leq \left( \int_0^L |\phi_0|^2 \, \d x\right) \mathrm{e}^{-\frac{2 \theta}{L^2} t} + \frac{\theta_{0}^2 L^3}{2 \theta}, \quad \forall \, t \in [0,T], 
\end{equation}
and 
\begin{equation}
\label{EST5-1D}
\int_0^T \int_0^L |\partial_{xx} \phi|^2 \, \d x \leq 
 \int_0^L |\phi_0|^2 \, \d x + \theta_0^2 L T.
\end{equation}

\medskip

\textbf{Second estimate.}
Multiplying \eqref{ac-CH-1D}$_1$ by $\mu$ and integrating over $(0,L)$, we find
\begin{equation}
\label{EE0-1D}
\ddt \left( \int_0^L \frac12 |\partial_x \phi|^2 + \Psi(\phi) \, \d x \right) +
\int_0^L |\partial_x \mu|^2 \, \d x  
= - \lambda \int_0^L |\partial_x \phi|^2 \partial_t \phi \, \d x.
\end{equation}
By using \eqref{ac-CH-1D}$_1$, we have
\begin{equation}
\label{EE-1D}
\ddt \left( \int_0^L \frac12 |\partial_x \phi|^2 + \Psi(\phi) \, \d x \right) +
\int_0^L |\partial_x \mu|^2 \, \d x
 = - \lambda \int_0^L |\partial_x \phi|^2 \partial_{xx} \mu \, \d x.
\end{equation}
We observe that an integration by parts yields  
\begin{align*}
-\lambda \int_0^L |\partial_x \phi|^2 \, \partial_{xx} \mu \, \d x
=  \lambda \int_0^L \partial_x \left( | \partial_x \phi|^2 \right) \partial_x \mu \, \d x
= 2 \lambda \int_0^L \partial_{xx}\phi \, \partial_x \phi \, \partial_x \mu \, \d x.
\end{align*}
By exploiting the one-dimensional Sobolev embedding $W^{1,2}(0,L)\subset L^\infty(0,L)$ and, in particular, the inequality $\| \partial_x \phi\|_{L^\infty(0,L)}\leq \sqrt{L}\| \partial_{xx}\phi\|_{L^2(0,L)}$, which follows from the Neumann boundary condition on $\phi$, it follows that 
\begin{equation}
\label{AC-1-1D}
\begin{split}
\left| \lambda \int_0^L |\partial_x \phi|^2 \, \partial_{xx} \mu \, \d x \right| 
&\leq 
2 |\lambda| \| \partial_{xx}\phi\|_{L^2(0,L)} 
\| \partial_x \phi\|_{L^\infty(0,L)}
 \| \partial_x \mu \|_{L^2(0,L)}
\\
&\leq 2 |\lambda| \sqrt{L} 
\| \partial_{xx}\phi\|_{L^2(0,L)}^2 
\| \partial_x \mu \|_{L^2(0,L)}.
\end{split}
\end{equation}
Now, multiplying \eqref{ac-CH}$_2$ by $-\partial_{xx}\phi $ and integrating on $(0,L)$, we also obtain
\begin{equation}
\label{h2-bound0-1D}
\begin{split}
\int_0^L |\partial_{xx}\phi |^2 \, \d x 
+ \int_0^L F''(\phi) |\partial_x \phi|^2 \, \d x
&= 
\int_0^L \partial_x \mu \, \partial_x \phi \, \d x
+ \theta_0 \int_0^L |\partial_x \phi|^2 \, \d x 
+\lambda \underbrace{\int_{\Omega} |\partial_x \phi|^2 \partial_{xx} \phi \, \d x}_{=0 \text{ by } \eqref{1D-trick}}
\\
&\leq \| \partial_x \mu \|_{L^2(0,L)} \| \partial_x \phi\|_{L^2(0,L)}
+ \theta_0 \| \partial_x \phi\|_{L^2(0,L)}^2. 
\end{split}
\end{equation}
In particular, by \eqref{h1-bound-1D} and the global bound on $\phi$, we further get
\begin{equation}
\label{h2-bound-1D}
\begin{split}
 \int_0^L |\partial_{xx}\phi |^2 \, \d x 
+ 2\int_0^L F''(\phi) |\partial_x \phi|^2 \, \d x
\leq 2 \| \partial_x \mu \|_{L^2(0,L)} \| \partial_x \phi\|_{L^2(0,L)}
+ \theta_0^2 L. 
\end{split}
\end{equation}
Combining \eqref{AC-1-1D} and \eqref{h2-bound-1D}, we deduce that
\begin{equation}
\label{AC-2-1D}
\begin{split}
\left| \lambda \int_0^L |\partial_x \phi|^2 \, \partial_{xx} \mu \, \d x \right| 
&\leq 2 |\lambda| \sqrt{L} 
\left( 2 \| \partial_x \mu \|_{L^2(0,L)} \| \partial_x \phi\|_{L^2(0,L)}
+ \theta_0^2 L \right)
\| \partial_x \mu \|_{L^2(0,L)}
\\
& \leq \frac12 \int_0^L |\partial_x \mu|^2 \, \d x 
+ 4 |\lambda| \sqrt{L} \| \partial_x \phi\|_{L^2(0,L)} 
\left(\int_0^L |\partial_x \mu|^2 \, \d x\right)
+ 2 \lambda^2 \theta_0^4 L^3.
\end{split}
\end{equation}
By using \eqref{AC-2-1D} in \eqref{EE-1D}, we arrive at
\begin{equation}
\label{EE2-1D}
\ddt \left( \int_0^L \frac12 |\partial_x \phi|^2 + \Psi(\phi) \, \d x \right) +
\left( \frac12 -4 |\lambda| \sqrt{L} \| \partial_x \phi\|_{L^2(0,L)} \right) \int_0^L |\partial_x \mu|^2 \, \d x 
\leq  2 \lambda^2 \theta_0^4 L^3.
\end{equation}
It is easily seen from \eqref{EE2-1D} that 
\begin{align*}
\label{EE3-1D}
&\ddt \left( \int_0^L \frac12 |\partial_x \phi|^2 + \widetilde{\Psi}(\phi) \, \d x \right) 
\\
&\quad +
\left( \frac12 -4 \sqrt{2} |\lambda| \sqrt{L} \left( \int_0^L \frac12 |\partial_x \phi|^2 + \widetilde{\Psi}(\phi) \, \d x \right)^\frac12 \right) \int_0^L |\partial_x \mu|^2 \, \d x 
\leq  2 \lambda^2 \theta_0^4 L^3.
\end{align*}
Setting $X(t)= \int_0^L \frac12 |\partial_x \phi|^2 + \widetilde{\Psi}(\phi) \, \d x $ and $Y(t)=\int_0^L |\partial_x \mu|^2 \, \d x $, we end up with the differential inequality
\begin{equation}
\label{DI1-1D}
X'(t)+ \left( \frac12 -4 \sqrt{2} |\lambda| \sqrt{L} \sqrt{X(t)}\right) Y(t)\leq 
2 \lambda^2 \theta_0^4 L^3.
\end{equation}


We now assume that 
\begin{equation}
\label{X0}
|\lambda| < \frac{1}{2^5\sqrt{2}\sqrt{L}\sqrt{R}},
\end{equation}
which equivantly implies that 
$$
X(0)\leq R < \frac{1}{2^{11} \lambda^2 L}.
$$
Since $X(t)$ is differentiable and thus continuous, there exists $\tau \in (0,\infty]$ (notice $\tau$ is strictly positive) such that 
\begin{equation}
\label{barrier}
\tau= \sup \left\lbrace t \in [0,\infty): \quad \max_{t\in [0,\tau]} X(t) \leq \frac{4}{2^{11} \lambda^2 L} \right\rbrace. 
\end{equation} 
Owing to this, for any $t \in (0,\tau)$, the differential inequality \eqref{DI1-1D} becomes 
\begin{equation}
\label{DI2-1D}
X'(t)+ \frac14 Y(t)\leq 
2 \lambda^2 \theta_0^4 L^3, 
\end{equation}
which, in turn, entails that
\begin{equation}
\label{DI3-1D}
X(t) + \frac14 \int_0^t Y(s) \, \d s
\leq X(0)+ 2 \lambda^2 \theta_0^4 L^3 t , \quad \forall \, t \in [0, \tau].
\end{equation}
As a consequence, we infer that
\begin{equation}
\label{DI4-1D}
\max_{t\in [0,\tau] } X(t) + \frac14 \int_0^\tau Y(s) \, \d s
\leq \frac{1}{2^{9} \lambda^2 L} + 2 \lambda^2 \theta_0^4 L^3 \tau.
\end{equation}
Moreover, using the definition of $\tau$, we can deduce a lower bound on $\tau$ as follows
\begin{equation}
\label{est_time}
\frac{1}{2^{9} \lambda^2 L} \leq X(0) + 2 \lambda^2 \theta_0^4 L^3 \tau 
\quad \Rightarrow \quad \tau \geq 
\frac{3}{2^{12}\lambda^4 \theta_0^4L^4}.
\end{equation}
In particular, by setting 
$$
|\lambda|
< \min \left\lbrace \frac{1}{2^5\sqrt{2}\sqrt{L}\sqrt{R}}, \ \frac{\sqrt[4]{3}}{2^3 \theta_0L}\right\rbrace, 
$$
we know that $\tau>1$.

\medskip

Next, we try to improve the differential inequality \eqref{DI1-1D}. 
First of all, by using \eqref{h2-bound0-1D} in \eqref{AC-1-1D}, and combining it with \eqref{EE-1D}, we have
\begin{align}
\ddt \left( \int_0^L \frac12 |\partial_x \phi|^2 + \Psi(\phi) \, \d x \right) 
& +
\left( \frac12 -2 |\lambda| \sqrt{L} \| \partial_x \phi\|_{L^2(0,L)} \right) \int_0^L |\partial_x \mu|^2 \, \d x  \notag
\\
&\leq  2\lambda^2  \theta_0^2 L \| \partial_x \phi\|_{L^2(0,L)}^4,
\label{EE4-1D}
\end{align}
namely 
\begin{equation}
\label{DI5-1D}
X'(t)+ \left( \frac12 -4 \sqrt{2} |\lambda| \sqrt{L} \sqrt{X(t)}\right) Y(t)
\leq 
2^3 \lambda^2 \theta_0^2 L X^2(t).
\end{equation}
In addition, owing to \eqref{barrier}, 
we also obtain
\begin{equation}
\label{DI5-2-1D}
X'(t)+ \frac14 Y(t)
\leq 
2^3 \lambda^2 \theta_0^2 L X^2(t), \quad \text{in } (0,\tau). 
\end{equation}

Now multiplying \eqref{ac-CH-1D}$_2$ by $\phi$ and integrating on $(0,L)$, we find 
\begin{equation}
\label{est-muphi}
\int_0^L |\partial_x \phi|^2 \, \d x + \int_0^L F'(\phi) \phi \, \d x -\theta_0 \int_0^L \phi^2 \, \d x + \lambda \int_0^L |\partial_x \phi|^2 \phi \, \d x = 
\int_0^L \mu \phi \, \d x.
\end{equation}
Since $F$ is convex and $F(0)=0$, we know that $F'(s)s \geq F(s)$ for any $s \in (-1,1)$. Also, exploiting the global boundeness and the zero-total mass, we deduce that
\begin{equation}
\begin{split}
\left( 1-|\lambda| \right) 
&\int_0^L |\partial_x \phi|^2 \, \d x + \int_0^L F(\phi) \, \d x  
\\
&\leq 
\frac{\theta_0}{2} \int_0^L \phi^2 \, \d x  
+ \| \mu-\overline{\mu}\|_{L^2(0,L)}\| \phi\|_{L^2(0,L)}
\\
&\leq \frac{\theta_0}{2}L
+ L^\frac32 \| \partial_x \mu\|_{L^2(0,L)}. 
\end{split}
\end{equation}
Thus, assuming that $1-|\lambda|=\omega>0$, we conclude that
\begin{equation}
\label{H1-stat-1D}
\begin{split}
\omega 
\int_0^L |\partial_x \phi|^2 \, \d x + \int_0^L \Psi(\phi) \, \d x 
&\leq \frac{\theta_0}{2} L
+ L^\frac32 \| \partial_x \mu\|_{L^2(0,L)}, 
\end{split}
\end{equation}
which in turn can be rewritten as 
\begin{equation}
\label{H1-stat-1D}
\begin{split}
\min \lbrace 2\omega, 1 \rbrace 
\left( \int_0^L \frac12 |\partial_x \phi|^2 \, \d x + \int_0^L \widetilde{\Psi}(\phi) \, \d x \right)
&\leq \left( \frac{\theta_0}{2}  + C_\Psi\right) L + L^\frac32 \| \partial_x \mu\|_{L^2(0,L)}.
\end{split}
\end{equation}
Following the notation above, this is equivalent to
\begin{equation}
\label{H1-stat2-1D}
\begin{split}
\min \lbrace 2\omega, 1 \rbrace 
X(t)
&\leq \left( \frac{\theta_0}{2}  + C_\Psi\right) L + L^\frac32 \sqrt{Y(t)}, 
\end{split}
\end{equation}
namely
\begin{equation}
\label{H1-stat3-1D}
\begin{split}
\min \lbrace 2\omega, 1 \rbrace^2 
X^2(t)
&\leq 2\left( \frac{\theta_0}{2}  + C_\Psi\right)^2 L^2 + 2L^3 Y(t). 
\end{split}
\end{equation}
Exploiting \eqref{H1-stat3-1D} in \eqref{DI5-2-1D}, we arrive at 
\begin{align}
X'(t)+ \left(\frac14
\frac{\min \lbrace 2\omega, 1 \rbrace^2 }{2L^3} 
- 2^3 \lambda^2 \theta_0^2 L\right)
X^2(t)
 \leq 
\frac{1}{2L} \left( \frac{3\theta_0}{2}  + C_\Psi\right)^2. 
\label{DI6-1D}
\end{align}
Assuming that $|\lambda|<<1$, we have $\min \lbrace 2\omega, 1 \rbrace=1$. Then, we rewrite the inequality above as
\begin{equation}
\label{DI7-1D}
X'(t)+ 
\left( \frac14 -2^4\lambda^2 \theta_0^2 L^4 \right) 
\frac{1}{2L^3} X^2(t)
\leq 
\frac{1}{2L} \left( \frac{3\theta_0}{2}  + C_\Psi\right)^2. 
\end{equation}
We now further assume that 
\begin{equation}
    |\lambda|
< \min \left\lbrace \frac12, \ \frac{1}{2^5\sqrt{2}\sqrt{L}\sqrt{R}}, \ \frac{\sqrt[4]{3}}{2^3 \theta_0L}, \  \frac{1}{2^\frac52 \theta_0 L^2} \right\rbrace.
\end{equation}
This implies that 
\begin{equation}
\label{DI8-1D}
X'(t)+ 
\frac{1}{8L^3} X^2(t) 
\leq 
\frac{1}{2L} \underbrace{\left( \frac{3\theta_0}{2}  + C_\Psi\right)^2}_{C_s}, \quad \forall \, t \in (0,\tau).
\end{equation}
For any initial condition $X(0)$, we can choose $\lambda$ sufficiently small such that $\tau>1$ from \eqref{est_time}. 


We report the following result from \cite[Chapter III, Section 5, Lemma 5.1]{T}:

\begin{lemma}[Superlinear Gronwall lemma]
\label{SGL}
Let $f$ be an absolutely continuous function on $[0,T]$ which satisfies 
$$
\ddt f(t) + \gamma f(t)^p \leq \delta,
$$
with $p>1$, $\gamma>0$, $\delta\geq 0$. Then, we have
$$
f(t) \leq \left( \frac{\delta}{\gamma} \right)^\frac{1}{p} + \frac{1}{\left(\gamma(p-1)t \right)^\frac{1}{p-1}}, \quad \forall \, t \in [0,T].
$$
\end{lemma}
\noindent
Applying this result to \eqref{DI8-1D} on $[1,\tau]$, we have
\begin{equation}
X(t)\leq 2\sqrt{C_s}L^2+ 8L^3.
\end{equation}
In order to show that $\tau=\infty$, it is sufficient that 
$$
2\sqrt{C_s}L^2+ 8L^3 < \frac{4}{2^{11} \lambda^2 L}, 
$$
namely
$$
\lambda^2 < \frac{1}{2^{11} L \left( 2\sqrt{C_s}L^2+ 8L^3  \right)}.
$$
Therefore, 
by assuming that
\begin{equation}
\label{lambda-val}
    |\lambda|
< \lambda_0:= \min \left\lbrace \frac12, \ \frac{1}{2^5\sqrt{2}\sqrt{L}\sqrt{R}}, \ \frac{\sqrt[4]{3}}{2^3 \theta_0L}, \  \frac{1}{2^\frac52 \theta_0 L^2}, \ 
\frac{1}{2^\frac{11}{2} \sqrt{\left( 2\sqrt{C_s}L^3+ 8L^4  \right)}}
\right\rbrace,
\end{equation}
we deduce that 
\begin{equation}
\label{bound-X}
X(t)\leq \frac{4}{2^{11}\lambda^2 L} \quad  \forall \, t \in [0,1], \quad \text{and} \quad X(t)\leq  2\sqrt{C_s}L^2+ 8L^3 \quad  \forall \, t \in [1,\infty),
\end{equation}
namely $X(t)$ is bounded globally in time.
In particular, there exists $K_1=K_1(L,R, \theta_0, \lambda_0)>0$ such that
$$
X(t) \leq K_1, \quad \forall \, t \geq 0.
$$
Furthermore, by using \eqref{DI5-1D} and the above estimates, we also obtain
\begin{equation}
X'(t)+ \frac14 Y(t)\leq 2^3 \lambda^2 \theta_0^2 L \left( 2\sqrt{C_s}L^2+ 8L^3\right)^2, \quad t \in (1,\infty).
\end{equation}
This entails that there exists $K_2=K_2(L,R, \theta_0, \lambda_0)>0$
\begin{equation}
\label{bound-Y}
\sup_{T \in [1,\infty)}\int_T^{T+1} Y(t) \, \d t \leq 
K_2.
\end{equation}
Finally, by exploiting \eqref{h2-bound-1D}, we also deduce that
\begin{equation}
\label{bound-phiH2}
\sup_{T \in [0,\infty)} \int_T^{T+1} \left( \int_0^L |\partial_{xx}\phi|^2 \, \d x\right)^2 \, \d t
\leq 2^2\left( \max_{t \in [0,\infty)} X(t) \right) \sup_{T\in [0, \infty)} \int_T^{T+1} Y(t) \, \d t + \theta_0^4 L^2.
\end{equation}
Thus, we conclude that
\begin{equation}
\label{glob-bounds}
\phi \in L^\infty(0,\infty; H^1(0,L)), \quad \phi \in L_{\rm uloc}^4([0,\infty); H^2(0,L)), \quad \nabla \mu \in L_{\rm uloc}^2([0,\infty); L^2(0,L)).
\end{equation}

\smallskip

\textbf{Third estimate.} 
In light of \eqref{ac-CH-1D}, we have 
\begin{equation}
\overline{\mu}= 
\frac{1}{L} \int_0^L F^\prime(\phi) \, \d x 
+ \frac{\lambda}{L} \int_0^L |\partial_x \phi|^2 \, \d x 
=
\frac{1}{L} \int_0^L F^\prime(\phi) \, \d x 
+ \frac{\lambda}{L} \| \partial_x \phi\|_{L^2(0,L)}^2,
\end{equation}
which gives 
\begin{equation}
\label{mass-mu}
|\overline{\mu}|\leq \frac{1}{L} \left| \int_0^L F^\prime(\phi) \, \d x \right|
+ \frac{2 \lambda}{L} X(t).
\end{equation}
By \eqref{est-muphi} and the assumption $0<|\lambda|<1$, we know that
\begin{equation}
\int_0^L F'(\phi) \phi \, \d x \leq \theta_0 \int_0^L \phi^2 \, \d x +
\int_0^L \mu \phi \, \d x
\leq \theta_0 X(t)+ \| \partial_x \mu\|_{L^2(0,L)}.
\end{equation}
We  now recall that (see e.g. \cite{MZ,Mbook,GMS})
\begin{equation}
\label{va-F'}
\int_0^L |F'(\phi)| \, \d x \leq C_1 \int_0^L F'(\phi) \phi \, \d x + C_2,
\end{equation}
where the positive constants $C_1$ and $C_2$ only depend on $\theta$ and $\overline{\phi}$. Owing to this, we derive that
\begin{equation}
\begin{split}
\sup_{T \in [0,\infty)} &\int_T^{T+1} \left(\int_0^L |F'(\phi)| \, \d x \right)^2 \, \d t
\\
&\leq 2^3C_1^2 \theta_0^2 \sup_{t \in [0,\infty)} X(t)^2+2^3C_1^2 \sup_{T\in [0,\infty)}\int_T^{T+1} \| \partial_x \mu(t)\|_{L^2(0,L)}^2\, \d t +2^3C_2^2.
\end{split}
\end{equation}
Then, by using the above estimate in \eqref{mass-mu}, we find
\begin{equation}
\label{est-mass-mu}
\begin{split}
\sup_{T \in [0,\infty)} &\int_T^{T+1} \left| \overline{\mu(t)}\right|^2\, \d t
\\
&\leq \left( \frac{2^4C_1^2 \theta_0^2}{L^2}+\frac{8\lambda^2}{L^2}\right) \sup_{t \in [0,\infty)} X(t)^2+\frac{2^4 C_1^2}{L^2} \sup_{T\in [0,\infty)}\int_T^{T+1} \| \partial_x \mu(t)\|_{L^2(0,L)}^2\, \d t +\frac{2^4C_2^2}{L^2}.
\end{split}
\end{equation}
Combining \eqref{glob-bounds} and \eqref{est-mass-mu}, and exploiting \eqref{PI-1D-2}, we arrive at
\begin{equation}
\mu \in L_{\rm uloc}^2([0,\infty); H^1(0,L)).
\end{equation}

Finally, it immediately follows from \eqref{ac-CH-1D} that
$$
\| \partial_t \phi\|_{H^1(0,L)'} \leq \| \nabla \mu\|_{L^2(0,L)}, 
$$
which, in light of \eqref{glob-bounds}, entails that $\partial_t \phi \in L^2_{\uloc}([0,\infty);H^1(0,L)')$.

\vspace{0.3cm}

\textbf{Uniqueness of global weak solutions.} Let $\phi_1$, $\phi_2$ be two
global weak solutions corresponding to the initial data
$\phi_{0}^1$, $\phi_{0}^2$, respectively.
Set $\Phi=\phi_1-\phi_2$ and $M=\mu_1-\mu_2$.
Testing the equation written for $\Phi$  by
$\mathcal{G}\Phi$ yields
\begin{equation*}
\frac12 \ddt \| \partial_x \mathcal{G} \Phi\|_{L^2(0,L)}^{2}
+\int_0^L M \, \Phi \, \d x=0.
\end{equation*}
According to \eqref{AC-1-1D}$_2$, we have
\begin{align*}
\int_0^L M \, \Phi \, \d x 
&= \int_0^L |\partial_x \Phi|^2\, \d x 
+\int_0^L \left( \Psi'(\phi_1)-\Psi'(\phi_2) \right) \Phi \, \d x 
+ \lambda \int_0^L  \left( |\partial_x \phi_1|^2-  |\partial_x \phi_2|^2\right) \Phi \, \d x
\\
&\geq \int_0^L |\partial_x \Phi|^2\, \d x -(\theta_0-\theta) \int_0^L |\Phi|^2 \, \d x + \lambda \int_0^L  \left( \partial_x \phi_1 +  \partial_x \phi_2\right)\partial_x \Phi \, \Phi \, \d x.
\end{align*}
In a classical fashion, we know that
\begin{align}
\label{interpolino}
(\theta_0-\theta) \int_0^L |\Phi|^2 \, \d x 
&=(\theta_0-\theta)  \int_0^L \partial_x \Phi \, \partial_x \mathcal{G} \Phi \, \d x
\\
&\leq \frac12 \int_0^L |\partial_x \Phi|^2\, \d x  +
\frac{(\theta_0-\theta)^2}{2}\| \partial_x \mathcal{G} \Phi\|_{L^2(0,L)}^{2}.
\end{align}
On the other hand, we infer that
\begin{equation}
\begin{split}
\lambda \int_0^L  &\left( \partial_x \phi_1  +  \partial_x \phi_2 \right) 
\partial_x \Phi \, \Phi \, \d x 
\\
&\leq |\lambda| \left( \| \partial_x \phi_1\|_{L^\infty(0,L)}+ \| \partial_x \phi_2\|_{L^\infty(0,L)} \right) \| \partial_x \Phi\|_{L^2(0,L)} \| \Phi\|_{L^2(0,L)}
\\
& \leq|\lambda| \sqrt{L} \left( \| \partial_{xx} \phi_1\|_{L^2(0,L)}+ \| \partial_{xx} \phi_2\|_{L^2(0,L)} \right) \| \partial_x \Phi\|_{L^2(0,L)}^\frac32 \| \partial_x \mathcal{G} \Phi\|_{L^2(0,L)}^\frac12
\\
& \leq \frac12 \int_0^L |\partial_x \Phi|^2\, \d x 
+ \frac{\lambda^4 L^2}{4} \left( \| \partial_{xx} \phi_1\|_{L^2(0,L)}+ \| \partial_{xx} \phi_2\|_{L^2(0,L)} \right)^4 \| \partial_x \mathcal{G} \Phi\|_{L^2(0,L)}^2.
\end{split}
\end{equation}
Then, we end up with
\begin{equation}
\frac12 \ddt \| \partial_x \mathcal{G} \Phi\|_{L^2(0,L)}^{2}
\leq 
G(t)
 \| \partial_x \mathcal{G} \Phi\|_{L^2(0,L)}^2,
\end{equation}
where
$$
G(t)=\left( \frac{(\theta_0-\theta)^2}{2} + \frac{\lambda^4 L^2}{4} \left( \| \partial_{xx} \phi_1\|_{L^2(0,L)}+ \| \partial_{xx} \phi_2\|_{L^2(0,L)} \right)^4 \right).
$$
Therefore, since $G\in L^1(0,T)$ in light of \eqref{glob-bounds}, an application of the Gronwall lemma entails the uniqueness of weak solutions.

\end{proof}

\section{Finite Element approximation}
In this section we introduce the finite element and time discretization of \eqref{ac-CHadim}, and we study its well posedness, stability and convergence to the continuous solution, separating the analysis between the cases with the logarithmic potential \eqref{log} and with the polynomial potential \eqref{pol}. The convergence results will also give existence results for the solution of the continuous problem \eqref{ac-CHadim}.
\subsection{Preliminaries and technical lemmas}
Let $T>0$ and $N\in \mathbb{N}^*$. We split the interval $[0,T]$ into $N$ equidistant subintervals $[t^n,t^{n+1}]$, with $t^n=n\Delta t$, where $\Delta t:=\frac{T}{N}$, $n\in\{0,\dots,N-1\}$. Let $\Omega\subset \mathbb{R}^d$, $d=1,2,3$, be a bounded domain of $\mathbb{R}^d$. We require $\{\mathcal{T}_h\}_{h>0}$ to be a family of conforming partitions of $\Omega$ into disjoint $d-$simplices $K$, with $h_K:=\text{diam}(K)$, $h:=\max_{K\in \mathcal{T}_h}h_K$, such that $\overline{\Omega}=\cup_{K\in \mathcal{T}_h}\overline{K}$. We will make the following assumptions on the domain $\Omega$ and on the meshes:
\begin{itemize}
\item[(A0):] $\Omega\subset \mathbb{R}^d$, $d=1,2,3$, is a bounded convex domain with polygonal shape for $d=2$ and polyhedral shape with $d=3$;
\item[(A1):] $\mathcal{T}_h$ is shape-regular, i.e., denoting $\rho_K$ the diameter of the largest inscribed ball in the simplex $K\in \mathcal{T}_h$, it holds that
\[
\sup_{K\in \mathcal{T}_h}h_K\rho_K^{-1}\leq C.
\]
Moreover, $\mathcal{T}_h$ is quasi-uniform, i.e. it holds that
\[
h_K\geq Ch, \;\; \forall \; K\in \mathcal{T}_h.
\]
\item[(A2):] $\mathcal{T}_h$ is an acute partitioning, which means that the angles of any triangle in the case $d=2$ is less then or equal to $\frac{\pi}{2}$, while the angle made by any two faces of any tetrahedron is less then or equal to $\frac{\pi}{2}$ for $d=3$.
\end{itemize}
\begin{remark}
    \label{rem:erpol}
    In $(A0)$ we are assuming a polygonal/polyhedral shape for $\Omega$.
    This technical constraint arises specifically from our finite-element-based existence theory. We pass to the limit in the discretization parameters to prove existence at the continuous level within a fixed domain of polygonal or polyhedral shape. We could bypass this requirement by working with isoparametric finite elements (see e.g. \cite[Chapter 10.4]{Brenner}), in which a base polygonal-polyedral domain is mapped through piece-wise polynomial functions to a domain approximating the smooth domain $\Omega$. In this setting, we can pass to the limit also in the approximating geometry. In this situation, the integrals of finite element functions over the approximated domain can be mapped back to the original domain $\Omega$. We will not delve into these technical details here, assuming $(A0)$ for simplicity.
\end{remark}
We introduce the following finite element space:
\begin{align*}
 & S^{h} := \{\chi \in C(\overline{\Omega}):\chi |_{K}\in \mathbb{P}_{1}(K) \; \forall K\in \mathcal{T}_{h}\}\subset W^{1,\infty}(\overline{\Omega}),
\end{align*}
where $\mathbb{P}_{1}(K)$ indicates the space of polynomials of total order one on $K$.

Let $\mathcal{I}$ be the set of nodes of $\mathcal{T}_{h}$ and $\{\mathbf{x}_j\}_{j\in \mathcal{I}}$ be the set of their coordinates. Moreover, let $\{\chi_j\}_{j\in \mathcal{I}}$ be the Lagrangian basis functions associated with each node $j\in \mathcal{I}$.
Denoting by $\pi^h:C(\overline{\Omega})\rightarrow S^h$ the standard Lagrangian interpolation operator we define the lumped scalar product as
\begin{equation}
\label{eqn:lump}
(\eta_1,\eta_2)^h=\int_{\Omega}\pi^h(\eta_1(\mathbf{x})\eta_2(\mathbf{x}))d\mathbf{x}\equiv \sum_{j\in \mathcal{I}}(1,\chi_j)\eta_1(\mathbf{x}_j)\eta_2(\mathbf{x}_j),
\end{equation}
for all $\eta_1,\eta_2\in C(\overline{\Omega})$. We observe that, since $\chi_j\geq 0$ for all $j\in \mathcal{I}$, the lumped scalar product induces a norm on $S^h$, which we indicate as $\lVert \cdot \rVert_h$. We also observe that
\[(f_h,1)^h=(f_h,1) \;\, \text{for any}\;\, f_h\in S^h.\]
We then introduce the $L^2$ projection operators $P_h,\hat{P}_h:L^2(\Omega)\to S^h$ such that
\begin{equation}
    \label{Projh}
    (P_h v,\chi)=(\hat{P}_h v,\chi)^h=(v,\chi)\quad \forall \, \chi \in S^h.
\end{equation}

\noindent
We recall the following well-known inverse estimates and interpolation results (see e.g. \cite{Barrett1,Barrett2,Brenner}).
\begin{lemma}
Under Assumptions $(A0)-(A1)$, the following properties hold:
\begin{align}
\label{eqn:interp1}
&|\chi_h|_{W^{m,p}(\Omega)}  \leq  Ch^{s-m+\frac{d}{p}-\frac{d}{q}}|\chi_h|_{W^{s,q}(\Omega)} \quad  \forall \, \chi_h \in S^h, \; s,m\in\{0,1\}, \; s\leq m,\; 1\leq q \leq p \leq \infty; \\
\label{eqn:interp2}
&||\chi_h||^{2} \leq \lVert \chi_h \rVert_{h}^2 \leq (d+2)||\chi_h||^{2} \quad \; \; \forall \, \chi_h \in S^h;\\
\label{eqn:interp3}
&|(\chi_h,\xi_h)^h-(\chi_h,\xi_h)|\leq Ch^{1+m}| \chi^h|_{H^m(\Omega)}\lVert \nabla \xi^h\rVert \quad \; \; \forall \, \chi_h,\xi_h \in S^h,\;\; m\in\{0,1\};\\
\label{eqn:interp4}
&\lVert (I-\pi^h)\eta\rVert_{L^p(\Omega)}+h\lVert \nabla(I-\pi^h)\eta\rVert_{L^p(\Omega)} \leq Ch^2 |\eta|_{W^{2,p}(\Omega)}\quad \; \; \forall \eta \in W^{2,p}(\Omega),\;\; p\in [2,+\infty];\\
\label{eqn:interp5}
&\lim_{h\to 0} \, \lVert (I-\pi^h)\eta\rVert_{L^{\infty}(\Omega)} =0\quad \; \; \forall \, \eta \in C(\overline{\Omega}).
\end{align}
\end{lemma}

 We introduce the spaces $\mathcal{F}^h=\{v\in {C}(\bar{\Omega}) : (v,1)=0\}$ and $V^h=\{v^h\in S^h:(v^h,1)^h=0\}$. Note that, for any $v^h\in S^h$, we have that $(v^h,1)^h=(v^h,1)$. Then, we define the discrete Green operator 
${\mathcal{G}}_h:\mathcal{F}^h\rightarrow V^h$ as follows
\begin{align}
\label{eqn:greendiscr}
(\nabla {\mathcal{G}}_h v,\nabla \chi)&=(v,\chi )^h \quad \forall \, \chi \in S^h,
\end{align}
which is well-posed as a consequence of the Lax--Milgram theorem. 
Taking $\chi=v$ in \eqref{eqn:greendiscr} and using \eqref{eqn:interp2} and the Cauchy--Schwarz inequality we deduce that 
\begin{equation}
\label{l2ggh}
\lVert v\rVert\leq \lVert \nabla \mathcal{G}_hv\rVert^{\frac{1}{2}}\lVert \nabla v\rVert^{\frac{1}{2}}.
\end{equation}
We also define the discrete Laplacian operator $\Delta_h:S^h\rightarrow V^h$ as follows:
\begin{align}
\label{eqn:lapldiscr}
-(\Delta_h v,\chi)^h&=(\nabla v,\nabla \chi ) \quad \forall \, \chi \in S^h,
\end{align}
whose well-posedness is a consequence of the Riesz representation theorem. Taking $\chi=-\Delta_hv$ in \eqref{eqn:lapldiscr} and using \eqref{eqn:interp2}, the Cauchy--Schwarz inequality and \eqref{eqn:interp1}, we obtain that
\begin{equation}
\label{discrinv}
    \lVert \Delta_h v\rVert^2\leq \lVert \nabla \Delta_h v\rVert \, \lVert \nabla v\rVert\leq Ch^{-1}\lVert \Delta_h v\rVert\,\lVert \nabla v\rVert,
\end{equation}
from which, applying again \eqref{eqn:interp1}, we find 
\begin{equation}
\label{deltainv}
    \lVert \Delta_h v\rVert\leq Ch^{-2}\lVert v\rVert.
\end{equation}

We will need the following Lemma (see, e.g., \cite{Ciavaldini}).
\begin{lemma}
    Under Assumption $(A2)$, given $g\in W^{1,\infty}(\mathbb{R})$, with $g(0)=0$ and $0\leq g'(s)\leq L_g<+\infty$ for a.e. $s\in \mathbb{R}$, the following estimate holds
    \begin{equation}
        \label{eqn:convg}
        \lVert \nabla \pi^h \left(g(\chi)\right) \rVert^2\leq L_g\left(\nabla \chi,\nabla \pi^h g(\chi)\right) \;\; \forall \; \chi \in S^h. 
    \end{equation}
\end{lemma}
A crucial argument in the forthcoming numerical analysis will be the validity of discrete Gagliardo--Nirenberg inequalities \eqref{gngeneral} with $j=1, m=r=2$. These are not guaranteed at the discrete level, since elements of $S^h$ do not have second order derivatives. 
In the following Lemma \ref{lemgndisc}, we will derive the discrete analogue of \eqref{gn4} in the space $S^h$. 
To the best of our knowledge, this is a novel result. Similar discrete Gagliardo-Nirenberg inequalities are available in literature for the particular case of piecewise constant discrete functions defined on a cartesian grid for finite volume discretization schemes \cite{Calgaro}. For finite elements, as far as we know only particular forms of the discrete Gagliardo-Nirenberg inequality with $q=2$ are available in literature \cite{Suli}.
\begin{lemma}
    \label{lemgndisc}
    Under the Assumptions $(A0)-(A1)$, there exists a $C>0$ independent on $h$ such that
\begin{equation}
    \label{gndisc2}
    \lVert \nabla \chi\rVert_{L^4(\Omega)}\leq C\lVert \Delta_h \chi\rVert^{\frac{1}{2}}\lVert \chi\rVert_{L^{\infty}(\Omega)}^{\frac{1}{2}} \quad \forall \; \chi \in S^h.
\end{equation}
\end{lemma}
The proof is inspired by the arguments in the proof of \cite[Theorem 2.8]{Liu}, which considers discrete versions of Gagliardo--Nirenberg inequalities different from \eqref{gngeneral}. In particular, the extension of these techniques to the case $q=\infty$ in \eqref{gngeneral} is not straightforward. 
\begin{proof}
Let us introduce the discrete $L^2$ laplacian $\Delta^h:S^h\rightarrow V^h$ as follows:
\begin{align}
\label{eqn:lapldiscr2}
-(\Delta^h v,\chi)&=(\nabla v,\nabla \chi ) \quad \forall \, \chi \in S^h.
\end{align}
Comparing with \eqref{eqn:lapldiscr}, we observe that $(\Delta^h v,\chi)=(\Delta_h v,\chi)^h$. Taking $\chi\equiv \Delta^h v$ and thanks to \eqref{eqn:interp2}, we deduce that
\begin{equation}
    \label{apb:1}
    \lVert \Delta^h v\rVert\leq C\lVert \Delta_h v\rVert.
\end{equation}
Let us also introduce the $H^1$ projection $P_h^1:H^1(\Omega)\to S^h$ defined, for $v\in H^1(\Omega)$, by
\begin{equation}
\label{projP1}
\left(\nabla P_h^1(v),\nabla \chi\right)=\left(\nabla v,\nabla \chi\right) \quad  \forall \, \chi \in S^h,
\end{equation}
with $(P_h^1(v),1)=(v,1)$. The following interpolation inequality holds under the assumptions $(A0)-(A1)$ (see e.g. \cite{Liu}):
\begin{equation}
    \label{apb:2}
    \lVert v-P_h^1v\rVert+h\lVert \nabla (v-P_h^1v)\rVert\leq Ch^2\lvert v \rvert_{H^2(\Omega)} \quad \forall \, v\in H^2(\Omega).
\end{equation}
We also need the following refined interpolation inequalities for $\pi^h$ (see e.g. \cite{Liu}), valid under the assumptions $(A0)-(A1)$:
\begin{equation}
    \label{apb:3}
    \lVert v-\pi^h(v)\rVert_{W^{m,q}(\Omega)}\leq Ch^{\frac{d}{q}-\frac{d}{2}}h^{2-m}\lvert v \rvert_{H^2(\Omega)} \quad \forall \, v\in H^2(\Omega), \, m\in \{0,1\},\, 2\leq q \leq \infty.
\end{equation}
Given $\chi \in S^h$, by $L^2$ elliptic regularity on convex polygonal or polyhedral domains \cite{Dauge,Grisvard}, there exists a unique $u\in H_N^2(\Omega):=\{v\in H^2(\Omega): \partial_{\mathbf{n}}v|_{\partial \Omega}=0\}$ such that
\[
(\nabla u,\nabla v)=-(\Delta^h \chi,v) \quad \forall \, v\in H^1(\Omega),
\]
with $(u,1)=(\chi,1)$ and
\begin{equation}
    \label{apb:4}
    \lVert u\rVert_{H^2(\Omega)}\leq C\lVert \Delta^h \chi\rVert.
\end{equation}
We observe from the definition \eqref{projP1} and from \eqref{eqn:lapldiscr2} that $\chi \equiv P_h^1(u)$ in $S^h$.

We now write
\[
\lVert \nabla \chi\rVert_{L^4(\Omega)}\leq \underbrace{\lVert \nabla u\rVert_{L^4(\Omega)}}_{I_1}+\underbrace{\lVert \nabla (u-\chi)\rVert_{L^4(\Omega)}}_{I_2}.
\]
For what concerns $I_1$, thanks to \eqref{gn4} and to \eqref{apb:4}, we have
\[
I_1\leq C \lVert u\rVert_{H^2(\Omega)}^{\frac{1}{2}}\lVert u\rVert_{L^{\infty}(\Omega)}^{\frac{1}{2}} \leq C\lVert \Delta^h\chi\rVert^{\frac{1}{2}}\lVert u\rVert_{L^{\infty}(\Omega)}^{\frac{1}{2}}.
\]
By exploiting \eqref{apb:3} with $m=0$, $q=\infty$ and \eqref{eqn:interp1}, we get
\begin{align*}
\lVert u\rVert_{L^{\infty}(\Omega)}&\leq \lVert \chi \rVert_{L^{\infty}(\Omega)}+\lVert u-\chi\rVert_{L^{\infty}(\Omega)}\leq \lVert \chi \rVert_{L^{\infty}(\Omega)}+\lVert u-\pi^h u\rVert_{L^{\infty}(\Omega)}+\lVert \pi^hu-\chi\rVert_{L^{\infty}(\Omega)}\\
& \leq \lVert \chi \rVert_{L^{\infty}(\Omega)}+Ch^{2-\frac{d}{2}}\lVert u\rVert_{H^2(\Omega)}+Ch^{-\frac{d}{2}}\left(\lVert u-\pi^h u\rVert+\lVert u-P_h^1 u\rVert\right).
\end{align*}
Then, thanks to \eqref{eqn:interp4}, \eqref{apb:2} and \eqref{apb:4}, we conclude that
\[
I_1\leq C\lVert \Delta^h\chi\rVert^{\frac{1}{2}}\left(\lVert \chi \rVert_{L^{\infty}(\Omega)}+Ch^{2-\frac{d}{2}}\lVert \Delta^h \chi \rVert\right)^{\frac{1}{2}}.
\]
For what concerns $I_2$, thanks to \eqref{apb:3} with $m=1,q=4$, as well as \eqref{eqn:interp1} and \eqref{apb:4}, we find
\begin{align*}
I_2&\leq \lVert \nabla(u-\pi^h u) \rVert_{L^{4}(\Omega)}+\lVert \nabla (\pi^h u-\chi)\rVert_{L^{4}(\Omega)}\\
& \quad \leq Ch^{1-\frac{d}{4}}\lVert \Delta^h \chi\rVert+Ch^{-\frac{d}{4}}\left(\lVert \nabla(u-\pi^h u)\rVert+\lVert \nabla(u-P_h^1 u)\rVert\right).
\end{align*}
Then, thanks to \eqref{eqn:interp4}, \eqref{apb:2} and \eqref{apb:4}, we infer that
\[
I_2\leq Ch^{1-\frac{d}{4}}\lVert \Delta^h \chi \rVert=C
\lVert \Delta^h\chi\rVert^{\frac{1}{2}}\left(h^{2-\frac{d}{2}}\lVert \Delta^h \chi \rVert\right)^{\frac{1}{2}}.
\]
We now crucially observe that, thanks to \eqref{apb:1} and to \eqref{deltainv}, we have
\begin{align*}
& h^{2-\frac{d}{2}}\lVert \Delta^h \chi \rVert\leq Ch^{-\frac{d}{2}}\lVert \chi \rVert=Ch^{-\frac{d}{2}}\left(\sum_{K\in \mathcal{T}_h}\int_K|\chi|^2\right)^{\frac{1}{2}}\\
& \quad \leq Ch^{-\frac{d}{2}}\lVert \chi \rVert_{L^{\infty}(\Omega)}\left(\sum_{K\in \mathcal{T}_h}\int_K1\right)^{\frac{1}{2}}\leq Ch^{-\frac{d}{2}}h^{\frac{d}{2}}\lVert \chi \rVert_{L^{\infty}(\Omega)}.
\end{align*}
Collecting the results we finally obtain that
\[
\lVert \nabla \chi\rVert_{L^4(\Omega)}\leq C\lVert \Delta^h\chi\rVert^{\frac{1}{2}}\lVert \chi \rVert_{L^{\infty}(\Omega)}^{\frac{1}{2}},
\]
from which, using \eqref{apb:1}, we get \eqref{gndisc2}.
\end{proof}

The last preliminary result we will need in the forthcoming analysis is the following discrete Bihary-type inequality. 
\begin{lemma}
    \label{bihari}
Let $\{y_n\}$ and $\{a_n\}$ be sequences of non-negative real numbers, and let $C \ge 0$ be a constant. Let moreover $r>1$, and assume that $a_k<\bar{\Delta}$ for any $k=1,\dots,n$, with $\bar{\Delta}>0$ a constant which depends only on $C$ and $r$ defined in \eqref{ap:6}. If the sequence $\{ y_n\}$ satisfies, for all integers $n$ with $1\leq n \leq N$, $N\in \mathbb{N}^*$, the inequality
\begin{equation}
\label{bihari1}
\displaystyle y_n \leq C + \sum_{k=1}^n a_k y_k^r,
\end{equation}
then there exist $\bar{R}>0$ and $n_1\in \mathbb{N}^*$, with $n_1\leq N$, which depend only on $C$ and $r$, such that, for $1\leq n\leq n_1$,  it holds that
\begin{equation}
\label{bihari2}
y_n \leq \frac{C}{\left[ 1 - (r-1)C^{r-1}\bar{R}^r\sum_{k=1}^{n_1} a_k \right]^{\frac{1}{r-1}}},
\end{equation}
where $n_1$ is such that $1 - (r-1)C^{r-1} \bar{R}^r\sum_{k=1}^{n} a_k>0$ for any $n\leq n_1$.
\end{lemma}
There are different versions of the discrete explicit Bihari inequality, i.e. inequality of type \eqref{bihari1} where the sum on the right hand side is up to $n-1$, reported in literature, see e.g. \cite{Pachpatte}. To the best of our knowledge, there are no reported proofs of Lemma \ref{bihari} in the implicit case. For this reason, we prove Lemma \ref{bihari}, 
providing in the proof the specific form of the constants $\bar{\Delta}$ and $\bar{R}$.
\begin{proof}
Let us define $z_n:= C + \sum_{k=1}^n a_k y_k^r$ for $n \geq 1$ and $z_0 = C$. We note that $y_n \le z_n$. By definition, it follows that $z_n - z_{n-1} = a_n y_n^r$, which implies that 
\begin{equation}
\label{ap:1}
z_n - z_{n-1} \leq a_n z_n^r.
\end{equation}
Let us now introduce the strictly decreasing function $f(x) = x^{1-r}$. By the mean value theorem, for any interval $[z_{n-1}, z_n]$, there exists a value $\xi \in (z_{n-1}, z_n)$ such that 
$f(z_{n-1}) - f(z_n) = f'(\xi)(z_{n-1} - z_n)$, which implies that
\[
z_{n-1}^{1-r} - z_n^{1-r} = (r-1)\xi^{-r}(z_n - z_{n-1}),
\]
and, since $\xi \ge z_{n-1}$,
\[z_{n-1}^{1-r} - z_n^{1-r} \leq (r-1)z_{n-1}^{-r}(z_n - z_{n-1}).
\]
Substituting \eqref{ap:1} in the previous inequality, we obtain that
\[
z_{n-1}^{1-r} - z_n^{1-r} \leq (r-1)a_n\rho_n^r,
\]
where $\rho_n:=z_n/z_{n-1}$. Note that in the explicit case the previous inequality would have been $z_{n-1}^{1-r} - z_n^{1-r} \leq (r-1)a_n$. In the implicit case we need to control the amplification factor $\rho_n$.
Summing the previous inequalities over $k$, we get 
\[
C^{1-r}-z_n^{1-r}=
\sum_{k=1}^n \left( z_{k-1}^{1-r} - z_{k}^{1-r} \right)\le (r-1)
\sum_{k=1}^n a_k\rho_k^r,
\]
which implies, raising to the power $(1-r)^{-1}$ and noting that $y_n\leq z_n$, that
\begin{equation}
\label{ap:2}
y_n \leq \frac{C}{\left[ 1 - (r-1)C^{r-1}\sum_{k=1}^n a_k\rho_k^r \right]^{\frac{1}{r-1}}}.
\end{equation}
In order to bound the amplification factor $\rho_k$, we observe from \eqref{ap:1} that, for any $k=1,\dots,n$, it satisfies the inequality
\begin{equation}
    \label{ap:3}
    \rho_k-a_kz_{k-1}^{r-1}\rho_k^r\leq 1.
\end{equation}
We also note that $\rho_k\geq 1$ and $\rho_k\to 1$ as $a_k\to 0$. Setting $\delta_k:=a_kz_{k-1}^{r-1}$, the maximum of the function $f(\rho):=\rho-\delta_k\rho^r$ is attained at $\bar{\rho}=(r\delta_k)^{-\frac{1}{r-1}}$, with $f(\bar{\rho})=\frac{r-1}{r}(r\delta_k)^{-\frac{1}{r-1}}$. Imposing the smallness condition 
\begin{equation}
\label{ap:4}
\delta_k<\frac{(r-1)^{r-1}}{r^r}, 
\end{equation}
we obtain that $f(\bar{\rho})>1$. Since $f(0)=0$ and $f(\bar{\rho})>1$, the graph of $f$ crosses the line $y=1$ twice at the points $R_{1,k}$ and $R_{2,k}$ such that $1\leq R_{1,k} < \bar{\rho}<R_{2,k}$. Thanks to \eqref{ap:3}, the condition \eqref{ap:4} implies that $1\leq \rho_k \leq R_{1,k}$, with $R_{1,k}$ finite and depending only on $r$ and $\delta_k$. Now, if there exists a $\bar{\delta}>0$ such that
\begin{equation}
    \label{ap:5}
    \delta_k\leq \bar{\delta}<\frac{(r-1)^{r-1}}{r^r} \;\, \text{for any}\;\, k=1,\dots,n,
\end{equation}
setting $\bar{R}$ to be the first positive root such that $\bar{R}-\bar{\delta}\bar{R}^r=1$ implies that $R_{1,k}\leq \bar{R}$ for all $k=1,\dots,n$. Indeed, since $f(\rho)$ is increasing in the first branch for $0\leq \rho \leq \bar{\rho}$, we have that $f'(\rho)=1-r\delta_k \rho^{r-1}>0$ for any $0\leq \rho < \bar{\rho}$. Hence, given that $R_{1,k}-\delta_kR_{1,k}^r=1$, and considering the implicit dependence of $R_{1,k}$ on $\delta_k$, we calculate that
\[
\frac{\mathrm{d}\,R_{1,k}}{\mathrm{d}\,\delta_k}=\frac{R_{1,k}^r}{1-r\delta_k R_{1,k}^{r-1}}> 0,
\] 
which means that the $R_{1,k}$ is increasing in $\delta_k$.
%
This gives \eqref{bihari2}. 

Since $\delta_k$ depends on $z_{k-1}$, which is defined through the $y_1,\dots,y_{k-1}$, whose finiteness is determined by \eqref{ap:2} and hence by the validity of \eqref{ap:4} at $1,\dots,k-1$, we need to employ an iterative argument to show that it is possible to ensure \eqref{ap:5} uniformly in $k$. This will be possible by requiring, via \eqref{ap:4}, smallness conditions on the coefficients $a_k$, for $k=1,\dots,n_1$, in terms of $\bar{\delta}$ and on the $y_k$, for $k=1,\dots,n_1-1$, and then choosing the infimum of such smallness conditions to be valid for any $a_k$. We follow the following constructive reasoning starting from $n=1$. 

\textit{Step $1$}: for $n=1$, we have that $\delta_1=a_1z_0^{r-1}=a_1C^{r-1}$. Choosing $\bar{\delta}<\frac{(r-1)^{r-1}}{r^r}$ and imposing that $a_1<C^{1-r}\bar{\delta}$, so that \eqref{ap:4} is satisfied, we deduce  from \eqref{bihari2} that $y_1\leq C\left(1-(r-1)C^{r-1}a_1\bar{R}^r\right)^{-1}$. Note that the denominator in the previous inequality is greater than zero, since $\bar{R}<\frac{r}{r-1}$. Hence, $z_1$ is non-negative and finite. 

\textit{Step $2$}: assuming that $z_{k-1}$ is non-negative and finite, we impose $a_k<z_{k-1}^{1-r}\bar{\delta}$, so that \eqref{ap:4} is satisfied, and then obtain from \eqref{bihari2} that $z_k$ is non-negative and finite, as long as the denominator of \eqref{bihari2} is greater than zero.
Setting
\begin{equation}
    \label{ap:6}
\bar{\Delta}:=\min_{k=1,\dots,n_1}\frac{\bar{\delta}}{z_{k-1}^{r-1}},
\end{equation}
and imposing that $a_k<\bar{\Delta}$ for any $k=1,\dots,n$, we have that \eqref{ap:5} is uniformly satisfied. Note that the bound on the sequence $a_k$ is implicit and depends only on $r$ and $C$.
\end{proof}
\subsection{Discretisation scheme}
In order to formulate the finite element and time discretization of \eqref{ac-CHadim} we introduce the following convex splitting of the potential $\Psi(\phi)$, where we distinguish between the regular and the singular cases:
\medskip

\noindent
\textbf{Regular case:}
\begin{equation}
\label{csplitr}
\Psi(\phi)=\Psi_{1,R}(\phi)+\Psi_{2,R}(\phi), \;\, \text{where}\;\,\Psi_{1,R}(\phi)=\frac{\phi^4}{4}, \;\; \Psi_{2,R}(\phi)=-\frac{\phi^2}{2}.
\end{equation}
\textbf{Singular case:}
\begin{align}
\label{csplits}
&\Psi(\phi)=\Psi_{1,S}(\phi)+\Psi_{2,S}(\phi), \;\,\text{where} \\ & \notag \Psi_{1,S}(\phi)=\frac{\theta}{2}\left((1+\phi)\log(1+\phi)+(1-\phi)\log(1-\phi)\right), \;\; \Psi_{2,S}(\phi)=-\frac{\theta_0}{2}\phi^2.
\end{align}
We also introduce the regularization of the singular potential, depending on the parameter $\alpha>0$, given by
\begin{equation}
    \label{reg}
    \Psi''_{1,S,\alpha}(s):=
    \begin{cases}
        \Psi''_{1,S}(-1+\alpha) \quad \text{for} \;\, s\leq -1+\alpha,\\
        \Psi''_{1,S}(s) \quad \text{for} \;\, -1+\alpha<s<1-\alpha,\\
        \Psi''_{1,S}(1-\alpha) \quad \text{for} \;\, s\geq 1-\alpha,
    \end{cases}
\end{equation}
with $\Psi'_{1,S,\alpha}(\mp 1\pm\alpha)=\Psi'_{1,S}(\mp 1\pm \alpha)$, $\Psi_{1,S,\alpha}(\mp 1\pm \alpha)=\Psi_{1,S}(\mp 1\pm \alpha)$. Integrating \eqref{reg} twice, we get
\begin{align}
    \label{regpsi}
    &\Psi_{1,S,\alpha}(s):=\\
    &\notag \quad \begin{cases}
        \frac{\theta}{4\alpha(2-\alpha)}(s+1)^2-\frac{\theta}{2}\left(\log\left(\frac{2-\alpha}{\alpha}\right)+\frac{1}{2-\alpha}\right)(s+1)+\frac{\theta}{2}\left(2\log(2-\alpha)+\frac{\alpha}{2(2-\alpha)}\right) \quad \text{for} \;\, s\leq -1+\alpha,\\[5pt]
        \frac{\theta}{2}\left((1+s)\log(1+s)+(1-s)\log(1-s)\right) \quad \text{for} \;\, -1+\alpha<s<1-\alpha,\\[5pt]
         \frac{\theta}{4\alpha(2-\alpha)}(s-1)^2+\frac{\theta}{2}\left(\log\left(\frac{2-\alpha}{\alpha}\right)+\frac{1}{2-\alpha}\right)(s-1)+\frac{\theta}{2}\left(2\log(2-\alpha)+\frac{\alpha}{2(2-\alpha)}\right) \quad \text{for} \;\, s\geq 1-\alpha,
    \end{cases}
\end{align}
and
\begin{equation}
    \label{regpsip}
    \Psi'_{1,S,\alpha}(s):=\begin{cases}
        \frac{\theta}{2\alpha(2-\alpha)}(s+1)-\frac{\theta}{2}\left(\log\left(\frac{2-\alpha}{\alpha}\right)+\frac{1}{2-\alpha}\right) \quad \text{for} \;\, s\leq -1+\alpha,\\[5pt]
        \frac{\theta}{2}\log\left(\frac{1+s}{1-s}\right) \quad \text{for} \;\, -1+\alpha<s<1-\alpha,\\[5pt]
         \frac{\theta}{2\alpha(2-\alpha)}(s-1)+\frac{\theta}{2}\left(\log\left(\frac{2-\alpha}{\alpha}\right)+\frac{1}{2-\alpha}\right) \quad \text{for} \;\, s\geq 1-\alpha.
    \end{cases}
\end{equation}
Finally, we introduce the concave-preserving extension $\bar{\Psi}_{2,S} \in C^1(\mathbb{R})$ of $\Psi_{2,R}\in C^1([-1,1])$ as
\begin{equation}
    \label{reg2}
    \bar{\Psi}_{2,S}(s):=
    \begin{cases}
        \Psi_{2,S}(-1)+(s+1)\Psi'_{2,S}(-1) \quad \text{for} \;\, s\leq -1,\\
        \Psi_{2,S}(s) \quad \text{for} \;\, -1<s<1,\\
        \Psi_{2,S}(1)+(s-1)\Psi'_{2,S}(1) \quad \text{for} \;\, s\geq 1.
    \end{cases}
\end{equation}
It is easy to observe that there exists a sufficiently small value $\alpha_0>0$ such that
\begin{equation}
    \label{psibelow}
    \Psi_{1,S,\alpha}(s)+\bar{\Psi}_{2,S}(s)+\frac{\theta_0}{2}\geq \frac{\theta}{4\alpha(2-\alpha)}\left([s-1]_+^2+[-1-s]_+^2\right) \;\; \forall \, s\in \mathbb{R},\;\,\forall \, \alpha \leq \alpha_0, 
\end{equation}
where $[\cdot]_+$ denoted the positive part. In particular we have, for all $s\in \mathbb{R}$,
\begin{equation}
    \label{psibelow2}
    \Psi_{1,S,\alpha}(s)+\bar{\Psi}_{2,S}(s)\geq -\frac{\theta_0}{2}. 
\end{equation}
At the same time, in the case with the regular potential \eqref{csplitr} we know that, for all $s\in \mathbb{R}$, 
\begin{equation}
    \label{psibelow3}
    \Psi_{1,R}(s)+\Psi_{2,R}(s)\geq -\frac{1}{4}. 
\end{equation}
The $\alpha-$regularization of the singular potential will be employed to prove the well-posedness of the discretization scheme by studying the $\alpha\to 0$ limit of a regularized scheme. 

\noindent
Let us consider the following assumption for the initial condition:
\begin{itemize}
    \item[(IC)] $\phi_0\in H^1(\Omega)$ and, in the case with the singular potential, $|\phi_0|\leq 1$ a.e. in $\Omega$ and $\overline{\phi}_0\in (-1,1)$.
\end{itemize}
We note that the bound $|\phi_0|\leq 1$ a.e. in $\Omega$ implies that $\Psi(\phi_0)\in L^1(\Omega)$. We set $\phi_h^0=\hat{P}_h(\phi_0)$. Thanks to \eqref{Projh}, it follows that $|\phi_h^0|\leq 1$ in $\Omega$ and $\overline{\phi}_h^0=\overline{\phi}_0\in (-1,1)$ as well.
We consider the following fully-discretized problem associated to \eqref{ac-CHadim}:
\medskip

\noindent
\textbf{Problem} $\mathbf{P^h}$: for $n=1,\cdots,N$, given $\phi_h^{n-1}\in S^h$ with $\lVert \phi_h^{n-1}\rVert_{H^1(\Omega)}\leq C_{n-1}$ and, in the case with the singular potential, with $-1\leq \phi_h^{n-1} \leq 1$ and $\overline{\phi}_h^{n-1}\in (-1,1)$, find $(\phi_h^n,\mu_h^n)\in S^h\times S^h$ such that, for all $(\chi,\xi)\in S^h\times S^h$, 
\begin{equation}
    \label{ph}
    \begin{cases}
        \displaystyle \left(\frac{\phi_h^n-\phi_h^{n-1}}{\Delta t},\chi\right)^h+(\nabla \mu_h^{n},\nabla \chi)=0,\\[8pt]
        (\mu_h^{n},\xi)^h=(\nabla \phi_h^n,\nabla \xi)+\left(\Psi'_1(\phi_h^n)+\Psi'_2(\phi_h^{n-1}),\xi\right)^h+\lambda\left(\nabla \phi_h^n\cdot \nabla \phi_h^{n-1},\xi\right).
    \end{cases}
\end{equation}

In the sequel, the constant $C_{n-1}$ may change value for different values of $n=1\dots,N$.
We observe that we address both the cases with a regular and a singular potential in a unified way employing the notations \eqref{csplitr} and \eqref{csplits}.
We now prove the well-posedness and conditional stability of Problem $P^h$.
\subsection{Well-posedness and conditional stability for Problem $P^h$}
We prove the following theorem.
\begin{theorem}
    \label{thm:wps}
   Let the Assumptions $(IC)-(A0)-(A1)-(A2)$ be satisfied, and let us also assume that $\Delta t < \lambda^{-4}(F(h))^{2}$, where $F(h)$ will be defined in \eqref{fh}. Then, for any $n=1,\dots,N$, there exists a unique solution $(\phi_h^n,\mu_h^n)\in S^h\times S^h$ to \eqref{ph}, which satisfies the stability estimate
   \begin{align}
   \label{stes}
   &\notag \frac{1}{2\sqrt{\Delta t}}\lVert \phi_h^n- \phi_h^{n-1}\rVert^2+\frac{\Delta t}{2} \lVert \nabla \mu_h^n \rVert^2+(\Psi(\phi_h^n),1)+\frac{1}{2}\left(1-\sqrt{\Delta t}\frac{\lambda^2}{F(h)}\right)\lVert \nabla \phi_h^n\rVert^2\\
   & \quad \leq \frac{1}{2}\lVert \nabla \phi_h^{n-1}\rVert^2 +(\Psi(\phi_h^{n-1}),1).
   \end{align}
   Besides, in the case with the singular potential, the solution $\phi_h^n$ satisfies the property $-1< \phi_h^n< 1$. 

    \smallskip
   
   In addition, in the case with the singular potential, assuming $(IC)-(A0)-(A1bis)-(A2)$ and the smallness condition $|\lambda|<(\sqrt{2}C_g^2)^{-1}$ (where $C_g$ is the constant appearing in \eqref{gndisc2}), the following estimate, which is uniform in the discretization parameters $\Delta t$ and $h$, is valid:
   \begin{equation}
       \label{stes2}
       \sup_{n\in \{1,\dots,N\}}\lVert \phi_h^n\rVert^2+\Delta t \sum_{n=1}^{N}\lVert \Delta_h \phi_h^{n}\rVert^2\leq C.
   \end{equation}
\end{theorem}
\begin{proof}
    The proof is divided into three steps:
    \begin{itemize}
        \item[-] \textit{Step 1:} Proof of the well-posedness for a regularized version of \eqref{ph} for a fixed $n$, given $\phi_h^{n-1}$ which satisfies the hypotheses in the statement of Problem $P^h$. The regularization is needed only for the case with the singular potential. 
        \item[-] \textit{Step 2:} Study of the limit problem for the regularized scheme as the regularization parameter tends to zero.
        \item[-] \textit{Step 3:} Extension by induction of the well-posedness result to any $n=1,\dots,N$, starting from $n=1$ with given $\phi_h^0$. Derivation of the stability estimates \eqref{stes} and \eqref{stes2} for \eqref{ph}.
    \end{itemize}
    \textit{Step 1.} Given $\alpha>0$, in the case with the singular potential let us introduce the regularization 
    \[
    \Psi_{\alpha}(\cdot)=\Psi_{1,\alpha}(\cdot)+\Psi_{2,\alpha}(\cdot):=\Psi_{1,S,\alpha}(s)+\bar{\Psi}_{2,S}(s), \;\, \forall \, s\in \mathbb{R}.
    \]
    In order to unify the analysis in the cases with the regular and the singular potentials, let us also introduce the notation $\Psi_{\alpha}(\cdot)=\Psi_{1,\alpha}(\cdot)+\Psi_{2,\alpha}(\cdot):=\Psi_{1,R}(\cdot)+\Psi_{2,R}(\cdot)$ in the regular case. Note that no dependence on the regularization parameter $\alpha$ is present in the latter case. 
    
    We introduce the following regularized problem, for a fixed $n\in\{1,\dots,N\}$.
    \medskip
    
\noindent
\textbf{Problem} $\mathbf{P_{\alpha}^{h,n}}$: given $\alpha >0$ and $\phi_h^{n-1}\in S^h$, with $\lVert \phi_h^{n-1}\rVert_{H^1(\Omega)}\leq C_{n-1}$ and, in the case with the singular potential, with $-1\leq \phi_h^{n-1} \leq 1$ and $\overline{\phi}_h^{n-1}\in (-1,1)$, find $(\phi_h^n,\mu_h^n)\in S^h\times S^h$ such that, for all $(\chi,\xi)\in S^h\times S^h$, 
\begin{equation}
    \label{pha}
    \begin{cases}
        \displaystyle \left(\frac{\phi_h^n-\phi_h^{n-1}}{\Delta t},\chi\right)^h+(\nabla \mu_h^{n},\nabla \chi)=0,\\[8pt]
        (\mu_h^{n},\xi)^h=(\nabla \phi_h^n,\nabla \xi)+\left(\Psi'_{1,\alpha}(\phi_h^n)+\Psi'_{2,\alpha}(\phi_h^{n-1}),\xi\right)^h+\lambda\left(\nabla \phi_h^n\cdot \nabla \phi_h^{n-1},\xi\right).
    \end{cases}
\end{equation}
The existence of a unique solution to Problem $P_{\alpha}^{h,n}$, given $\phi_h^{n-1}$ as in the hypothesis, is obtained via a fixed-point argument. 

Setting $\phi_h^{n,0}:=\phi_h^{n-1}$, we introduce the map $\mathcal{T}:S^h\to S^h$, with $\mathcal{T}(\phi_h^{n,i-1})=\phi_h^{n,i}$, for any $i\in \mathbb{N}^+$, defined as the solution of the variational formulation
\begin{equation}
    \label{phai}
    \begin{cases}
        \displaystyle \left(\frac{\phi_h^{n,i}-\phi_h^{n-1}}{\Delta t},\chi\right)^h+(\nabla \mu_h^{n,i},\nabla \chi)=0,\\[8pt]
        (\mu_h^{n,i},\xi)^h=(\nabla \phi_h^{n,i},\nabla \xi)+\left(\Psi'_{1,\alpha}(\phi_h^{n,i})+\Psi'_{2,\alpha}(\phi_h^{n-1}),\xi\right)^h+\lambda\left(\nabla \phi_h^{n,i-1}\cdot \nabla \phi_h^{n-1},\xi\right),
    \end{cases}
\end{equation}
for all $(\chi,\xi)\in S^h\times S^h$. A fixed point of $\mathcal{T}$ is a solution of Problem $P_{\alpha}^{h,n}$. Note that, since $\phi_h^{n-1},\phi_h^{n,i-1}\in S^h\subset W^{1,\infty}(\Omega)$, the third $L^2$ inner product on the right-hand side in the second equation of \eqref{phai} is well defined. 
\medskip

The existence of a unique solution to \eqref{phai} is obtained through a variational principle. As a consequence of the first equation of \eqref{phai}, we observe that, given $\phi_h^{n-1}\in S^h$, we look for a solution $\phi_h^{n,i}$ in the closed and convex space $V^h(\phi_h^{n-1}):=\{\chi \in S^h: (\chi-\phi_h^{n-1},1)^h=0\}$. Moreover, from the first equation of \eqref{phai} and from \eqref{eqn:greendiscr}, we get 
\begin{equation}
\label{mui}
\mu_h^{n,i}=-\mathcal{G}_h\left(\frac{\phi_h^{n,i}-\phi_h^{n-1}}{\Delta t}\right)+\nu_h^{n,i}, \;\, \text{with}\;\, \nu_h^{n,i}=\frac{(\mu_h^{n,i},1)^h}{|\Omega|}.
\end{equation}
Substituting in the second equation of \eqref{phai}, we can rewrite \eqref{phai} as follows: given $\phi_h^{n-1},\phi_h^{n,i-1}\in S^h$, find $\phi_h^{n,i}\in V^h(\phi_h^{n-1})$ and the Lagrange multiplier $\nu_h^{n,i}\in \mathbb{R}$ such that, for all $\xi \in S^h$,
\begin{align}
    \label{kkt}
    & \notag (\nabla \phi_h^{n,i},\nabla \xi)+\left(\frac{\mathcal{G}_h(\phi_h^{n,i}-\phi_h^{n-1})}{\Delta t},\xi\right)^h+\left(\Psi'_{1,\alpha}(\phi_h^{n,i})+\Psi'_{2,\alpha}(\phi_h^{n-1}),\xi\right)^h+\lambda\left(\nabla \phi_h^{n,i-1}\cdot \nabla \phi_h^{n-1},\xi\right)\\
    &\qquad -\nu_h^{n,i}(1,\xi)^h=0.
\end{align}
We note that \eqref{kkt} represents, together with $\phi_{h}^{n,i}\in V^{h}(\phi_{h}^{n-1})$, the Karush-Kuhn-Tucker optimality condition for the minimization problem
\begin{multline}
\label{lagrangian}
\displaystyle \inf_{v_{h}\in S^h}\sup_{\nu\in \mathbb{R}}\biggl\{
\lVert \nabla v_h\rVert^2+\frac{1}{\Delta t}||\nabla {\mathcal{G}}_{h}(v_{h}-\phi_{h}^{n-1}) ||^{2}\\
\displaystyle +2(\Psi_{1,\alpha}(v_{h})+\Psi'_{2,\alpha}(\phi_h^{n-1})v_{h},1)^h +2\lambda \left(\left(\nabla \phi_h^{n,i-1}\cdot \nabla \phi_h^{n-1}\right),v_h\right)
-\nu(v_{h},1)^h
\biggr\},
\end{multline}
being $\nu \in \mathbb{R}$ the Lagrange multiplier associated to the mass conservation constraint.
Noting the convexity and the coercivity of $\Psi_{1,\alpha}(\cdot)$ (see equations \eqref{csplitr} and \eqref{regpsi}), the primal form associated to the Lagrangian \eqref{lagrangian} is a convex, proper, lower semi continuous and coercive function from the closed and convex set $V^h(\phi_{h}^{n-1})$ to $\mathbb{R}$. 
Hence, we infer from the Kuhn-Tucker Theorem (see, e.g., theorem $5.1$ in \cite{temam}) the existence of $\phi_{h}^{n,i}\in V^h(\phi_{h}^{n-1})$, solution to the primal problem, and Lagrange multiplier $\nu_h^{n,i}\in \mathbb{R}$, for each $i\in \mathbb{N}^+$. Therefore, from \eqref{mui} we have the existence of a solution $(\phi_{h}^{n,i},\mu_{h}^{n,i})$ to \eqref{phai}.

Let us now prove the uniqueness. If, for fixed $i\geq 1$ and given $\phi_h^{n-1},\phi_h^{n,i-1}\in S^h$, \eqref{kkt} has two solutions $(\phi_{h,j}^{n,i},\nu_{h,j}^{n,i})$, $j=1,2$, defining $d_h^{n,i}:=\phi_{h,1}^{n,i}-\phi_{h,2}^{n,i}$, $\iota_h^{n,i}:=\nu_{h,1}^{n,i}-\nu_{h,2}^{n,i}$ and taking the difference between the respective equations, we have 
\[
    (\nabla d_h^{n,i},\nabla \xi)+\left(\frac{\mathcal{G}_h d_h^{n,i}}{\Delta t},\xi\right)^h+\left(\Psi'_{1,\alpha}(\phi_{h,1}^{n,i})-\Psi'_{1,\alpha}(\phi_{h,2}^{n,i}),\xi\right)^h-\iota_h^{n,i}(1,\xi)^h=0.
\]
Choosing $\xi\equiv \Delta t \, d_h^{n,i}$ in the previous equation, employing the monotonicity of $\Psi'_{1,\alpha}(\cdot)$ and observing that $(d_h^{n,i},1)^h=0$, we find
\[
\Delta t\lVert \nabla d_h^{n,i} \rVert^2+\lVert \nabla \mathcal{G}_h d_h^{n,i} \rVert^2\leq 0,
\]
from which we obtain that $\phi_{h,1}^{n,i}\equiv \phi_{h,2}^{n,i}$ in $S^h$. The uniqueness of $\nu_h^{n,i}$, and hence of $\mu_h^{n,i}$, is a straightforward consequence of \eqref{kkt}, \eqref{mui} and of the uniqueness of $\phi_h^{n,i}$.
\medskip

Now we prove the existence of a unique fixed point for the map $\mathcal{T}$. Let us denote $e^{i}:=\phi_h^{n,i}-\phi_h^{n,i-1}$, $e^{i-1}:=\phi_h^{n,i-1}-\phi_h^{n,i-2}$, $m^{i}:=\mu_h^{n,i}-\mu_h^{n,i-1}$. Taking the difference between \eqref{phai} at $i$ and at $i-1$, we obtain 
\begin{equation*}
    \begin{cases}
        \displaystyle \left(\frac{e^i}{\Delta t},\chi\right)^h+(\nabla m^i,\nabla \chi)=0,\\[8pt]
        (m^i,\xi)^h=(\nabla e^i,\nabla \xi)+\left(\Psi'_{1,\alpha}(\phi_h^{n,i})-\Psi'_{1,\alpha}(\phi_h^{n,i-1}),\xi\right)^h+\lambda\left(\nabla e^{i-1}\cdot \nabla \phi_h^{n-1},\xi\right),
    \end{cases}
\end{equation*}
for all $(\chi,\xi)\in S^h\times S^h$. Let us consider $\chi=\mathcal{G}_h e^i$, and, noting that $e^i\in V^h$, $\xi=\Delta t\, e^i$. Using \eqref{eqn:greendiscr} and the monotonicity of $\Psi'_{1,\alpha}(\cdot)$, we deduce that
\begin{equation*}
\lVert \nabla \mathcal{G}_he^i\rVert^2+\Delta t \lVert \nabla e^i\rVert^2\leq -\lambda \Delta t (\nabla e^{i-1}\cdot \nabla\phi_h^{n-1},e^i)\leq |\lambda|\Delta t \lVert \nabla e^{i-1}\rVert\, \lVert\nabla \phi_h^{n-1} \rVert_{L^{\infty}(\Omega)}\lVert e^i\rVert.
\end{equation*}
In the case with the regular potential, by using \eqref{eqn:interp1}, we have
\begin{equation}
\label{gh1}
\lVert\nabla \phi_h^{n-1} \rVert_{L^{\infty}(\Omega)}\leq \bar{C}h^{-\frac{d}{2}}C_{n-1}=:G(h),
\end{equation}
where $\bar{C}$ is the constant in \eqref{eqn:interp1}.
In the case with the singular potential, using again \eqref{eqn:interp1} and the fact that $-1\leq \phi_h^{n-1}\leq 1$, we find
\begin{equation}
\label{gh2}
\lVert\nabla \phi_h^{n-1} \rVert_{L^{\infty}(\Omega)}\leq \bar{C}h^{-1}\lVert \phi_h^{n-1}\rVert_{L^{\infty}(\Omega)}\leq \bar{C}h^{-1}=:G(h).
\end{equation}
By using \eqref{l2ggh} and the Young inequality, we obtain 
\begin{align*}
\lVert \nabla \mathcal{G}_he^i\rVert^2+\Delta t \lVert \nabla e^i\rVert^2
&\leq  G(h)|\lambda| \Delta t \lVert \nabla e^{i-1}\rVert\,\lVert e^i\rVert
\\ 
&\leq G(h)|\lambda| \Delta t \lVert \nabla e^{i-1}\rVert\,\lVert \nabla \mathcal{G}_he^i\rVert^{\frac{1}{2}}\lVert \nabla e^i\rVert^{\frac{1}{2}}\\
& \quad \leq \frac{1}{4}\lVert \nabla \mathcal{G}_he^i\rVert^2+\frac{1}{3}\Delta t \lVert \nabla e^i\rVert^2+\frac{2}{3}\Delta t \left(\frac{3\sqrt{3}}{4\sqrt{4}}G(h)^2\lambda^2(\Delta t)^{\frac{1}{2}}\right)\lVert \nabla e^{i-1}\rVert^2.
\end{align*}
Finally, by using the Poincar\'{e}--Wirtinger inequality we conclude that
\[
\lVert e_i\rVert_{H^1(\Omega)}^2\leq \left(\frac{3\sqrt{3}}{4\sqrt{4}}(C_p^2+1)G(h)^2\lambda^2(\Delta t)^{\frac{1}{2}}\right)\lVert e_{i-1}\rVert_{H^1(\Omega)}^2,
\]
where $C_p$ is the Poincar\'{e} constant. If the term between parenthesis is less than one, namely, if
\begin{equation}
    \label{dtban}
    \Delta t < 
    \begin{cases}
        \displaystyle \frac{64}{27(C_p^2+1)\bar{C}^4C_{n-1}^4\lambda^4}h^{2d} \quad \text{in the case with regular potential},\\[5pt]
        \displaystyle \frac{64}{27(C_p^2+1)\bar{C}^4\lambda^4}h^{4} \quad \quad \quad \;\, \text{in the case with singular potential},
    \end{cases}
\end{equation}
we have that $\mathcal{T}$ is a contraction, and then there exists a unique fixed point for $\mathcal{T}$, which gives the unique solution of Problem $P_{\alpha}^{h,n}$. 
\medskip

\noindent
\textit{Step 2:} We now study the limit problem as $\alpha\to 0$ and prove a well-posedness result for problem \eqref{ph}. This is only needed in the case with the singular potential. To this aim, we perform a-priory estimates, which hold as well for the regular case. Let us take $\chi=\mu_h^{n+1}+\frac{1}{\sqrt{\Delta t}}(\phi_h^n-\phi_h^{n+1})$ and $\xi=\phi_h^n-\phi_h^{n-1}$ in \eqref{pha}. By using the Cauchy--Schwarz and Young inequalities, we deduce that
 \begin{align*}
 & \frac{1}{\sqrt{\Delta t}}\lVert \phi_h^n-\phi_h^{n-1}\rVert^2+\Delta t \lVert \nabla \mu_h^n\rVert^2+\frac{1}{2}\lVert \nabla \phi_h^n\rVert^2+\frac{1}{2}\lVert \nabla (\phi_h^n-\phi_h^{n-1})\rVert^2+(\Psi_{\alpha}(\phi_h^n),1)^h\\
 & \quad \leq \frac{1}{2}\lVert \nabla \phi_h^{n-1}\rVert^2+(\Psi_{\alpha}(\phi_h^{n-1}),1)^h-\sqrt{\Delta t}(\nabla \mu_h^n,\nabla (\phi_h^n-\phi_h^{n-1}))-\lambda(\nabla \phi_h^n\cdot \nabla \phi_h^{n-1}, \phi_h^n-\phi_h^{n-1})\\
 & \quad \leq \frac{1}{2}\lVert \nabla \phi_h^{n-1}\rVert^2+(\Psi_{\alpha}(\phi_h^{n-1}),1)^h +\frac{\Delta t}{2}\lVert \nabla \mu_h^n\rVert^2+\frac{1}{2}\lVert \nabla (\phi_h^n-\phi_h^{n-1})\rVert^2\\
 & \qquad +|\lambda|\lVert \nabla \phi_h^n \rVert \,\lVert \nabla \phi_h^{n-1} \rVert_{L^{\infty}(\Omega)}\lVert \phi_h^n-\phi_h^{n-1}\rVert.
 \end{align*}
Then, by exploiting \eqref{gh1}, \eqref{gh2}, we have 
\begin{align*}
 & \frac{1}{\sqrt{\Delta t}}\lVert \phi_h^n-\phi_h^{n-1}\rVert^2+\frac{\Delta t}{2} \lVert \nabla \mu_h^n\rVert^2+\frac{1}{2}\lVert \nabla \phi_h^n\rVert^2+(\Psi_{\alpha}(\phi_h^{n}),1)^h
 \\
 &\leq \frac{1}{2}\lVert \nabla \phi_h^{n-1}\rVert^2+(\Psi_{\alpha}(\phi_h^{n-1}),1)^h
  +|\lambda|G(h)\lVert \nabla \phi_h^n \rVert \,\lVert \phi_h^n-\phi_h^{n-1}\rVert
  \\
  &\leq \frac{1}{2}\lVert \nabla \phi_h^{n-1}\rVert^2+(\Psi_{\alpha}(\phi_h^{n-1}),1)^h+\frac{1}{2\sqrt{\Delta t}}\lVert \phi_h^n-\phi_h^{n-1}\rVert^2
 +\frac{\sqrt{\Delta t}}{2}\lambda^2G(h)^2\lVert \nabla \phi_h^n\rVert^2,
 \end{align*}
which entails that
 \begin{align}
     \label{ape1}
     &\notag \frac{1}{2\sqrt{\Delta t}}\lVert \phi_h^n-\phi_h^{n-1}\rVert^2+\frac{\Delta t}{2} \lVert \nabla \mu_h^n\rVert^2+\frac{1}{2}\left(1-\sqrt{\Delta t}\lambda^2G(h)^2\right)\lVert \nabla \phi_h^n\rVert^2+(\Psi_{\alpha}(\phi_h^{n}),1)^h\\
     & \quad \leq \frac{1}{2}\lVert \nabla \phi_h^{n-1}\rVert^2+(\Psi_{\alpha}(\phi_h^{n-1}),1)^h.
 \end{align}
We observe that, in the case with the regular potential, if $\lVert 
 \phi_h^{n-1}\rVert_{H^1(\Omega)}\leq C_{n-1}$, the right hand side of \eqref{ape1} is bounded thanks to \eqref{csplitr}, \eqref{eqn:interp1} and the Sobolev embedding. Similarly, in the case with the singular potential, if $\lVert 
 \phi_h^{n-1}\rVert_{H^1(\Omega)}\leq C_{n-1}$ and $-1\leq \phi_h^{n-1}\leq 1$, the right hand side of \eqref{ape1} is uniformlu bounded in $\alpha$ thanks to \eqref{csplits}. 
Hence, recalling \eqref{psibelow2}, \eqref{psibelow3}, the Poincar\'e--Wirtinger inequality, upon noting that $(\phi_h^n,1)^h=(\phi_h^{n-1},1)^h$, and setting
\begin{equation}
    \label{dtape}
    \Delta t <
    \begin{cases}
        \displaystyle \frac{h^{2d}}{\bar{C}^4C_{n-1}^4\lambda^4} \quad \text{in the case with regular potential},\\
        \displaystyle \frac{h^{4}}{\bar{C}^4\lambda^4} \quad \quad \quad \;\, \text{in the case with singular potential},
    \end{cases}
\end{equation}
we conclude that
 \begin{equation}
 \label{alpha1}
 \lVert \phi_h^{n}\rVert_{H^1(\Omega)}\leq C_n \;\, \text{uniformly in}\;\, \alpha.
 \end{equation} In addition, in the case with the singular potential, thanks to \eqref{eqn:interp2},  \eqref{psibelow} and \eqref{ape1}, we get that
\begin{equation}
    \label{plusminush}
    \lVert[\phi_h^n-1]_+\rVert^2+\lVert[-1-\phi_h^n]_+\rVert^2\leq C\alpha(2-\alpha).
\end{equation}
Let us now consider the specific case with the singular potential. We take $\xi=\phi_h^n-\overline{\phi}_h^n$ in \eqref{pha}, where $\overline{\phi}_h^n=|\Omega|^{-1}(\phi_h^n,1)^h$. Recalling that $-1\leq \phi_h^{n-1}\leq 1$ and \eqref{reg2}, and by using the Poincar\'e--Wirtinger and the Cauchy-Schwarz inequalities, \eqref{gh2} and \eqref{ape1}, we obtain
\begin{align}
\label{psipes}
& \notag \left(\Psi'_{1,\alpha}(\phi_h^n),\phi_h^n-\bar{\phi}_h^n\right)^h\\
&\quad =\left(\mu_h^n-\bar{\mu}_h^n,\phi_h^n-\bar{\phi}_h^n\right)^h-\lVert \nabla \phi_h^n\rVert^2+\frac{\theta_0}{2}\left(\phi_h^{n-1},\phi_h^n-\bar{\phi}_h^n\right)^h -\lambda\left(\nabla \phi_h^n\cdot \nabla \phi_h^{n-1},\phi_h^n-\bar{\phi}_h^n\right) \notag
\\
& \quad \leq C\left(\lVert \nabla \mu_h^n\rVert+1\right)\lVert \nabla \phi_h^n\rVert+Ch^{-1}\lVert \nabla \phi_h^n\rVert^2\leq C\left((\Delta t)^{-\frac{1}{2}},h^{-1}\right).
\end{align}
Thanks to $\overline{\phi}_h^n=\overline{\phi}_h^0\in (-1,1)\subset \text{dom}\left(\Psi'_1\right)$, and observing that $\Psi_{1,\alpha}$ is convex with $\Psi'_{1,\alpha}(0)=0$ (see \eqref{regpsip}), following the same arguments as in \cite[p. 908]{GMS} we can prove that, for each node $j\in \mathcal{I}$ of $\mathcal{T}_h$, there exist $C_1>0, C_2\in \mathbb{R}$ such that
\[
C_1|\Psi'_{1,\alpha}(\phi_h^n(\mathbf{x}_j))|\leq \Psi'_{1,\alpha}(\phi_h^n(\mathbf{x}_j)) \left(\phi_h^n(\mathbf{x}_j)-\bar{\phi}_h^n\right)+C_2.
\]
Multiplying the previous inequality by $(1,\chi_j)$,  summing over $j\in \mathcal{I}$, and using \eqref{psipes}, we conclude that
\begin{equation}
    \label{psipun}
    \left(|\Psi'_{1,\alpha}(\phi_h^n)|,1\right)^h\leq C\left((\Delta t)^{-\frac{1}{2}},h^{-1}\right)
\end{equation}
uniformly in $\alpha$. Taking now $\xi=1$ in \eqref{pha}, and employing \eqref{psipun}, the hypotheses on $\phi_h^{n-1}$ and \eqref{ape1}, we obtain 
\[|(\mu_h^n,1)^h|\leq C\left((\Delta t)^{-\frac{1}{2}},h^{-1}\right).\]
The Poincar\'e--Wirtinger inequality, together with \eqref{ape1}, implies that
\begin{equation}
    \label{alpha2}
    \lVert \mu_h^n\rVert_{H^1(\Omega)}\leq C\left((\Delta t)^{-\frac{1}{2}},h^{-1}\right)\;\, \text{uniformly in}\;\, \alpha.
\end{equation}
Finally, taking $\xi=\pi^h\left(\Psi_{1,\alpha}'(\phi_h^n)\right)$ in \eqref{pha}, and using \eqref{eqn:convg}, \eqref{gh2}, \eqref{alpha1}, \eqref{alpha2}, the hypotheses on $\phi_h^{n-1}$ and \eqref{eqn:interp2}, we infer that
\begin{equation}
    \label{alpha3}
    \lVert \pi^h\left(\Psi_{1,\alpha}'(\phi_h^n)\right)\rVert^2\leq\lVert \Psi_{1,\alpha}'(\phi_h^n)\rVert_h^2\leq C\lVert \mu_h^n\rVert^2+C\lVert \nabla \phi_h^{n-1}\rVert_{L^{\infty}(\Omega)}^2+C \leq C\left((\Delta t)^{-1},h^{-2}\right),
\end{equation}
uniformly in $\alpha$. From the uniform bounds \eqref{alpha1}, \eqref{alpha2}, \eqref{alpha3} and the Bolzano--Weierstrass theorem, it follows that, for any sequence $\alpha_k\to 0$, there exists a subsequence $\alpha_k'\to 0$ and $\phi_h^n,\mu_h^n,\xi_h^n\in S^h$ such that, for $k\to +\infty$,
\begin{align}
    \label{phiconv}
    &\phi_{h,\alpha_k'}^n \to \phi_h^n, \qquad  \nabla \phi_{h,\alpha_k'}^n \to \nabla\phi_h^n,\\
    \label{muconv}
    &\mu_{h,\alpha_k'}^n \to \mu_h^n, \qquad \nabla \mu_{h,\alpha_k'}^n \to \nabla\mu_h^n,\\
    \label{psiconv}
    &\pi^h\left(\Psi'_{1,\alpha_k'}(\phi_{h,\alpha_k'}^n)\right)\to \xi_h^n,
\end{align}
where we explicitly reported the dependence on $\alpha$ of the solution of \eqref{pha}. Passing to the limit in \eqref{plusminush} we also have that $|\phi_h^n|\leq 1$ in $\Omega$. Furthermore, in light of \eqref{phiconv}, since $\Psi'_{1,\alpha}(\cdot)$ is maximal monotone (recall that it is a continuous non-decreasing function defined on $\mathbb{R}$), standard results on maximal monotone operators (see e.g. \cite[Proposition 1.1, p.42]{Barbu}) imply that $\xi_h^n\equiv \pi^h\left(\Psi'_{1}(\phi_{h}^n)\right)$ In addition, by the logarithmic singularity in $\Psi'_{1,S}$, it follows that $|\phi_h^n|<1$. The latter also implies that $\overline{\phi}_h^n\in (-1,1)$.

Thanks to the previous convergence results, it is easy to pass to the limit as $\alpha\to 0$ in \eqref{pha}, proving that the limit point defined in \eqref{phiconv}-\eqref{muconv} is a solution of \eqref{ph}. The uniqueness of the solution of \eqref{ph} can be inferred similarly as for the solution of \eqref{pha}. In particular, given $\phi_h^{n-1}$, Let us assume that there exist two solutions $(\phi_{h,j}^{n},\mu_{h,j}^{n})$, $j=1,2$, to \eqref{ph}, and let us define $\Phi_h^{n}:=\phi_{h,1}^{n}-\phi_{h,2}^{n}$, $\Sigma_h^{n}:=\mu_{h,1}^{n}-\mu_{h,2}^{n}$.
Taking the difference between the respective equations, choosing $\chi\equiv \mathcal{G}_h(\Phi_h^n), \xi\equiv \Delta t \Sigma_h^{n}$, and employing the monotonicity of $\Psi'_{1,S}(\cdot)$, we get 
\[
\Delta t\lVert \nabla \Phi_h^{n} \rVert^2+\lVert \nabla \mathcal{G}_h (\Phi_h^{n}) \rVert^2\leq |\lambda| \Delta t \lVert \nabla \Phi_h^n\rVert\, \lVert \nabla \phi_h^{n-1}\rVert_{L^{\infty}(\Omega)}\lVert \Phi_h^n\rVert.
\]
Hence, using \eqref{l2ggh}, \eqref{gh2} and the Young inequality, we obtain that
\begin{align}
\Delta t\lVert \nabla \Phi_h^{n} \rVert^2+\lVert \nabla \mathcal{G}_h (\Phi_h^{n}) \rVert^2 
&
\leq |\lambda| C\frac{\Delta t}{h} 
\lVert \nabla \Phi_h^n\rVert^{\frac32}
\lVert \nabla \mathcal{G}_h(\Phi_h^n)\rVert^{\frac{1}{2}} 
\notag
\\
& 
\leq \Delta t \frac{3}{4}\lVert \nabla \Phi_h^{n} \rVert^2+\Delta t \frac{|\lambda|^4\bar{C}^4}{4h^4}\lVert \nabla \mathcal{G}_h (\Phi_h^{n}) \rVert^2.
\label{phih!}
\end{align}
Thanks to \eqref{dtape} we thus obtain that $\phi_{h,1}^{n}\equiv \phi_{h,2}^{n}$ in $S^h$. The uniqueness of $\mu_h^{n}$ is then a consequence of \eqref{ph} and of the uniqueness of $\phi_h^{n}$.
\medskip

\noindent
\textit{Step 3:} We now extend inductively the well-posedness of \eqref{ph} to each level $n=1,\dots,N$, starting from $n=1$. 
Setting
\begin{equation}
    \label{fh}
    F(h):=
    \begin{cases}
    \begin{aligned}
    &\displaystyle \min\left(\displaystyle \frac{4\sqrt{4}}{3\sqrt{3}\sqrt{C_p^2+1}},1\right)\min_{n=1\dots,N}\left(\frac{h^{d}}{\bar{C}^2C_{n-1}^2}\right) \quad &\text{regular potential},\\
        &\min\left(\displaystyle \frac{4\sqrt{4}}{3\sqrt{3}\sqrt{C_p^2+1}},1\right)\frac{h^{2}}{\bar{C}^2} \quad & \text{singular potential},
        \end{aligned}
    \end{cases}
\end{equation}
and taking $\Delta t < \lambda^{-4}(F(h))^2$, we obtain the well posedness of \eqref{ph} for any $n=1,\dots,N$.  Then, proceeding as in \eqref{ape1}, we consider $\chi=\mu_h^{n+1}+\frac{1}{\sqrt{\Delta t}}(\phi_h^n-\phi_h^{n+1})$ and $\xi=\phi_h^n-\phi_h^{n-1}$ in \eqref{ph}, which leads to the validity of \eqref{stes} for any $n=1,\dots,N$, given $\phi_h^{n-1}$. Repeating the same arguments as those employed for the regularized problem, we infer that, if $\lVert 
 \phi_h^{n-1}\rVert_{H^1(\Omega)}\leq C_{n-1}$, then $\lVert 
 \phi_h^{n}\rVert_{H^1(\Omega)}\leq C_n$, and, in the case with the singular potential, if $-1\leq \phi_h^{n-1}\leq 1$ and $\overline{\phi}_h^{n-1}\in(-1,1)$, then $-1\leq \phi_h^{n}\leq 1$ and $\overline{\phi}_h^n\in(-1,1)$.
We finally observe that, at the level $n=1$, we have by hypothesis that $\lVert 
 \phi_h^0\rVert_{H^1(\Omega)}\leq C_0$ and, in the case with the singular potential, $-1\leq \phi_h^0\leq 1$ and $\overline{\phi}_h^0\in(-1,1)$. Hence, we can extend inductively the well posedness for any $n=1,\dots,N$, starting from $n=1$.
 \medskip
 
 \noindent
 We now consider the case with the singular potential \eqref{csplits} and prove \eqref{stes2}. Let us take $\chi\equiv \phi_h^n$ and $\xi\equiv -\Delta t\,\Delta_h\phi_h^n$ in \eqref{ph}. By exploiting \eqref{eqn:lapldiscr} and \eqref{eqn:convg}, we obtain 
 \begin{align*}
     & \frac{\lVert \phi_h^n\rVert_h^2}{2}+\frac{\lVert \phi_h^n-\phi_h^{n-1}\rVert_h^2}{2}\underbrace{-\left(\pi^h(\Psi'_1(\phi_h^n)),\Delta_h\phi_h^n\right)^h}_{\geq 0}+\Delta t\lVert \Delta_h\phi_h^n\rVert_h^2\\
     & \leq \frac{\lVert \phi_h^{n-1}\rVert_h^2}{2} +\Delta t \theta_0\lVert\phi_h^{n-1}\rVert_h\,\lVert\Delta_h\phi_h^n\rVert_h +|\lambda|\Delta t \lVert \nabla \phi_h^n\rVert_{L^4(\Omega)}\lVert \nabla \phi_h^{n-1}\rVert_{L^4(\Omega)}\lVert \Delta_h\phi_h^n\rVert.
 \end{align*}
 We recall that $-1\leq \phi_h^{n}\leq 1$ for any $n=0,\dots,N$. By using \eqref{gndisc2}, with corresponding constant $C_g$, as well as \eqref{eqn:interp2}, the Cauchy--Schwarz and Young inequalities, we get 
 \begin{align*}
     & \frac{\lVert \phi_h^n\rVert_h^2}{2}+\frac{\lVert \phi_h^n-\phi_h^{n-1}\rVert_h^2}{2}+\Delta t\lVert \Delta_h\phi_h^n\rVert_h^2
     \\
     &\quad \leq \frac{\lVert \phi_h^{n-1}\rVert_h^2}{2}+\Delta t \theta_0^2|\Omega|+\frac{\Delta t}{2}\lVert\Delta_h\phi_h^n\rVert_h^2  +\lambda^2\Delta t C_g^4 \lVert \Delta_h \phi_h^n\rVert_h \lVert \Delta_h \phi_h^{n-1}\rVert_h
     \\
     &\quad \leq \frac{\lVert \phi_h^{n-1}\rVert_h^2}{2}+C\Delta t+\frac{\Delta t}{2}\left(1+\lambda^2C_g^4\right)\lVert \Delta_h \phi_h^n\rVert_h^2+\frac{\Delta t}{2}\lambda^2C_g^4\lVert \Delta_h \phi_h^{n-1}\rVert_h^2.
 \end{align*}
 Summing over $n=1,\dots,N$, and using \eqref{discrinv}, we infer that
 \begin{align*}
     & \frac{\lVert \phi_h^N\rVert_h^2}{2}+\frac{\Delta t}{2}\left(1-\lambda^2C_g^4\right)\sum_{n=1}^{N}\lVert \Delta_h\phi_h^n\rVert_h^2
     \\
     &\quad \leq \frac{\lVert \phi_h^{0}\rVert_h^2}{2}+C+\frac{\Delta t}{2}\lambda^2C_g^4\lVert\Delta_h\phi_h^0\rVert_h^2  +\frac{\Delta t}{2} \lambda^2C_g^4 \sum_{n=1}^{N}\lVert \Delta_h \phi_h^n\rVert_h^2
     \\
     &\quad \leq C+C\frac{\Delta t}{h^2}\lVert\nabla\phi_h^0\rVert^2+ \frac{\Delta t}{2} \lambda^2C_g^4 \sum_{n=1}^{N}\lVert \Delta_h \phi_h^n\rVert_h^2.
 \end{align*}
 Since $\lVert \phi_h^0\rVert_{H^1(\Omega)}\leq C$ and  $\Delta t<\lambda^{-4}(F(h))^2$, we conclude that
 \begin{align*}
     & \frac{\lVert \phi_h^N\rVert_h^2}{2}+\frac{\Delta t}{2}\left(1-2\lambda^2C_g^4\right)\sum_{n=1}^{N}\lVert \Delta_h\phi_h^n\rVert_h^2\leq C+C\frac{(F(h))^2}{h^2}\leq C.
 \end{align*}
Assuming that $|\lambda|<(\sqrt{2}C_g^2)^{-1}$, we then deduce \eqref{stes2}, which is uniform in $\Delta t$ and $h$.
\end{proof}
\begin{remark}
    \label{rem:conv}
    We emphasize that the estimate \eqref{stes2} is uniform in the discretization parameters $\Delta t$ and $h$. However, the lack of a uniform bound for $\Psi'(\phi_h^n)$ in some $L^p$ space, with $p>1$, prevents us to pass to the limit in \eqref{ph} and obtain an existence result for a suitable notion of a weak-distributional solution satisfying a primal weak formulation of the active CH equation. 
\end{remark}

\subsection{Higher order estimates and convergence analysis in the singular case}
In this section we derive higher order a-priori estimates for the solution to \eqref{ph} in the case with the singular potential, which are uniform in the discretization parameters $\Delta t,h$. These estimates are valid only locally in time, i.e. there exists a $\bar{N}\leq N$ such that they are valid for $n=1,\dots,\bar{N}$. These a-priori estimates will let us identify the limit system of \eqref{ph} as $\Delta t,h\to 0$ and prove an existence and uniqueness result for a local in time weak solution to \eqref{ac-CH}-\eqref{ac-CH-bc}. 

The higher a-priori estimates are valid under the following higher regularity assumption for the initial condition:
\begin{itemize}
    \item[(ICR)] 
    $\phi_0$ satisfies $(IC)$ and is such that $\mu_0\in H^1(\Omega)$, where $\mu_0=-\Delta \phi_0+\Psi'(\phi_0)+\lambda |\nabla \phi_0|^2$.
\end{itemize}
Setting $\phi_h^0=\hat{P}_h(\phi_0)$ and $\mu_h^0=\hat{P}_h(\mu_0)$, we have the following result.
\begin{lemma}
    \label{lem:hoe}
   Let the Assumptions $(ICR)-(A0)-(A1)-(A2)$ be satisfied, and let us also assume that $\Delta t < \lambda^{-4}(F(h))^{2}\sim h^4$ and $|\lambda|<(\sqrt{6}C_g^2)^{-1}$. Let $(\phi_h^n,\mu_h^n)\in S^h\times S^h$ be the unique solution to \eqref{ph}. Then, there exists a positive integer $\bar{N}\leq N$, which depends only on proper norms of the initial condition, and a positive constant $C$ independent on $h,\Delta t$, such that the following higher order stability estimate holds
   \begin{align}
   \label{hoe}
   &\sup_{n\in \{1,\dots,\bar{N}\}}\left \{\lVert \phi_h^n\rVert_{H^1(\Omega)}^2+\lVert \mu_h^n\rVert_{H^1(\Omega)}^2+\lVert \pi^h (\Psi'(\phi_h^n)) \rVert^2+\left\lVert\frac{\phi_h^n-\phi_h^{n-1}}{\Delta t}\right\rVert^2+\lVert \Delta_h \phi_h^n\rVert^2\right\}\leq C.
   \end{align}
\end{lemma}
\begin{proof}
First we need to show some preliminary estimates. Let us take $\chi\equiv \mathcal{G}_h\left(\frac{\phi_h^n-\phi_h^{n-1}}{\Delta t}\right)$ in \eqref{ph}. Using \eqref{eqn:greendiscr} and the Cauchy--Schwarz inequality, we easily obtain that
    \begin{equation}
        \label{prel1}
        \left \lVert \nabla \mathcal{G}_h\left(\frac{\phi_h^n-\phi_h^{n-1}}{\Delta t}\right) \right \rVert\leq \lVert \nabla \mu_h^n \rVert.
    \end{equation}
    Taking $\chi \equiv \phi_h^n-\phi_h^{n-1}$ in \eqref{ph}, using \eqref{eqn:interp1}, \eqref{eqn:interp2} and the fact that $\Delta t \sim h^4$, we have 
    \begin{equation}
        \label{prel2}
        \lVert \phi_h^n-\phi_h^{n-1} \rVert\leq Ch^3\lVert \nabla \mu_h^n \rVert.
    \end{equation}
    We now consider $\xi\equiv \Delta_h \phi_h^n$ in \eqref{ph}. Rearranging terms and exploiting \eqref{eqn:convg} and \eqref{eqn:interp2}, we obtain 
    \begin{align*}
        & \lVert \Delta_h \phi_h^n\rVert^2\underbrace{-\left(\Psi'_1(\phi_h^n),\Delta_h \phi_h^n\right)^h}_{\geq 0}\\
        & \quad \leq \left(\nabla \mu_h^n,\nabla \phi_h^n\right)-\theta_0\left(\phi_h^{n-1},\Delta_h\phi_h^n\right)^h +\lambda \left(|\nabla \phi_h^n|^2,\Delta_h \phi_h^n\right)-\lambda \left(\nabla \phi_h^n\cdot \nabla (\phi_h^n-\phi_h^{n-1}),\Delta_h \phi_h^n\right).
    \end{align*}
    Considering now that $|\phi_h^{n-1}|\leq 1$ and employing \eqref{eqn:interp1}, \eqref{gndisc2}, \eqref{prel2}, the Cauchy--Schwarz and Young inequalities, we deduce that
\begin{align*}
         \lVert \Delta_h \phi_h^n\rVert^2&\leq \lVert\nabla \mu_h^n\rVert\,\lVert\nabla \phi_h^n\rVert+\frac{1}{2}\lVert \Delta_h \phi_h^n\rVert^2+C+\frac{3}{2}\lambda^2C_g^4\lVert \Delta_h \phi_h^n\rVert^2 +\frac{3}{2}\lambda^2\, \lVert\nabla \phi_h^n\rVert_{L^4(\Omega)}^2h^{-\frac{(d+4)}{2}}\lVert \phi_h^n-\phi_h^{n-1}\rVert^2
        \\
        & 
        \leq \lVert\nabla \mu_h^n\rVert\,\lVert\nabla \phi_h^n\rVert+\frac{1}{2}\lVert \Delta_h \phi_h^n\rVert^2+C +3\lambda^2C_g^4\lVert \Delta_h \phi_h^n\rVert^2 +C\Delta t\, h^{4-d}\lVert\nabla \mu_h^n\rVert^4,
    \end{align*}
    from which we conclude that, if $|\lambda|<(\sqrt{6}C_g^2)^{-1}$,
    \begin{equation}
       \label{prel3}
       \lVert \Delta_h \phi_h^n\rVert^2\leq C\lVert\nabla \mu_h^n\rVert\,\lVert\nabla \phi_h^n\rVert+C\Delta t\, h^{4-d}\lVert\nabla \mu_h^n\rVert^4+C.
    \end{equation}
 We proceed now by taking $\chi\equiv \mu_h^n$, $\xi \equiv \phi_h^n-\phi_h^{n-1}$ in \eqref{ph}. Thanks to the convexity of $\Psi_1(\cdot)$ and the concavity of $\Psi_2(\cdot)$, and using the Poincar\'{e}--Wirtinger inequality and \eqref{eqn:interp1}, we have 
 \begin{align*}
     &\frac{\lVert \nabla \phi_h^n\rVert^2}{2}+\frac{\lVert \nabla (\phi_h^n-\phi_h^{n-1})\rVert^2}{2}+(\Psi(\phi_h^n),1)^h+\Delta t \lVert \nabla \mu_h^n\rVert^2\\
     & \qquad 
     \leq \frac{\lVert \nabla \phi_h^{n-1}\rVert^2}{2}+(\Psi(\phi_h^{n-1}),1)^h
     -\lambda \left(|\nabla \phi_h^n|^2,\phi_h^n-\phi_h^{n-1}\right)+\lambda \left(\nabla \phi_h^n\cdot \nabla (\phi_h^n-\phi_h^{n-1}),\phi_h^n-\phi_h^{n-1}\right)\\
     & \qquad 
     \leq \frac{\lVert \nabla \phi_h^{n-1}\rVert^2}{2}+(\Psi(\phi_h^{n-1}),1)^h+C|\lambda|\,\lVert\nabla \phi_h^n\rVert_{L^4(\Omega)}^2\lVert \nabla (\phi_h^n-\phi_h^{n-1})\rVert
     \\
     &\qquad \quad +C\lVert\nabla \phi_h^n\rVert_{L^4(\Omega)}h^{-\frac{(d+4)}{4}}\lVert \phi_h^n-\phi_h^{n-1}\rVert^2.
 \end{align*}
 Using now \eqref{gndisc2}, \eqref{prel2} and the Young inequality, we obtain 
 \begin{align}
     \label{hoe1}
     & \notag \frac{\lVert \nabla \phi_h^n\rVert^2}{2}+\frac{\lVert \nabla (\phi_h^n-\phi_h^{n-1})\rVert^2}{2}+(\Psi(\phi_h^n),1)^h+\Delta t \lVert \nabla \mu_h^n\rVert^2\\
     &\quad 
     \leq \frac{\lVert \nabla \phi_h^{n-1}\rVert^2}{2}+(\Psi(\phi_h^{n-1}),1)^h
     +\epsilon \frac{\lVert \nabla (\phi_h^n-\phi_h^{n-1})\rVert^2}{\Delta t}  \notag
     \\
     &\qquad+C\Delta t \lVert \Delta_h \phi_h^n \rVert^2+C\Delta t\,h^{1-\frac{d}{4}}\lVert \Delta_h \phi_h^n \rVert^{\frac{1}{2}}\lVert \nabla \mu_h^n \rVert^2,
 \end{align}
 where $\epsilon>0$ is a coefficient which will be chosen later on. 
 
 Then, we consider the difference between the second equation of \eqref{ph} at time levels $n$ and $n-1$, where we set $\phi_h^{-1}=\phi_h^0$. Taking then $\chi \equiv \mu_h^n-\mu_h^{n-1}$ and $\xi\equiv (\phi_h^n-\phi_h^{n-1})/\Delta t$ in the resulting system, we find
 \begin{align*}
     & \frac{\lVert \nabla \mu_h^n\rVert^2}{2}+\frac{\lVert \nabla (\mu_h^n-\mu_h^{n-1})\rVert^2}{2}+\frac{\lVert \nabla (\phi_h^n-\phi_h^{n-1})\rVert^2}{\Delta t}+\underbrace{\left(\Psi'_1(\phi_h^n)-\Psi'_1(\phi_h^{n-1}),\frac{\phi_h^n-\phi_h^{n-1}}{\Delta t}\right)^h}_{\geq 0}\\
     & \quad =\frac{\lVert \nabla \mu_h^{n-1}\rVert^2}{2}+\underbrace{\theta_0\left(\phi_h^{n-1}-\phi_h^{n-2},\frac{\phi_h^n-\phi_h^{n-1}}{\Delta t}\right)^h}_{I_1}-\underbrace{\lambda \left(\nabla(\phi_h^n-\phi_h^{n-2})\cdot \nabla \phi_h^{n-1},\frac{\phi_h^n-\phi_h^{n-1}}{\Delta t}\right)}_{I_2}.
 \end{align*}
 We observe that
 \[
 \theta_0\left(\phi_h^{n-1}-\phi_h^{n-2},\frac{\phi_h^n-\phi_h^{n-1}}{\Delta t}\right)^h=-\theta_0 \Delta t \left\lVert \frac{\phi_h^n-\phi_h^{n-1}}{\Delta t}\right\rVert_h^2+\theta_0\left(\frac{\phi_h^{n}-\phi_h^{n-2}}{\Delta t},\phi_h^n-\phi_h^{n-1}\right)^h.
 \]
 Taking the difference between the first equation of \eqref{ph} at time levels $n$ and $n-1$ and choosing $\chi\equiv \phi_h^n-\phi_h^{n-1}$, we can deduce from the Cauchy--Schwarz and the Young inequalities that
 \[
 |I_1|\leq \epsilon\frac{\lVert \nabla (\phi_h^n-\phi_h^{n-1})\rVert^2}{\Delta t}+C\Delta t\left(\lVert \nabla \mu_h^n \rVert^2+\lVert \nabla \mu_h^{n-1} \rVert^2\right). 
 \]
 Next, we rewrite $I_2$ as
 \[
 I_2=\underbrace{\lambda \left(\nabla(\phi_h^n-\phi_h^{n-1})\cdot \nabla \phi_h^{n-1},\frac{\phi_h^n-\phi_h^{n-1}}{\Delta t}\right)}_{J_1}+\underbrace{\lambda \left(\nabla(\phi_h^{n-1}-\phi_h^{n-2})\cdot \nabla \phi_h^{n-1},\frac{\phi_h^n-\phi_h^{n-1}}{\Delta t}\right)}_{J_2}.
 \]
 Using \eqref{gngeneral} with $j=0, p=4, m=1, r=2, q=2$,
 togetheer with \eqref{l2ggh}, \eqref{prel1}, \eqref{eqn:interp2}, \eqref{gndisc2} and the Young inequality, we obtain 
 \begin{align*}
 & |J_1|\leq \lVert \nabla (\phi_h^n-\phi_h^{n-1})\rVert\,\lVert \nabla \phi_h^{n-1}\rVert_{L^4(\Omega)}\left \lVert \frac{\phi_h^n-\phi_h^{n-1}}{\Delta t}\right \rVert_{L^4(\Omega)}\\
 & \quad \leq \lVert \nabla (\phi_h^n-\phi_h^{n-1})\rVert^{\frac{4+d}{4}}\lVert \Delta_h \phi_h^{n-1} \rVert^{\frac{1}{2}}\left \lVert \frac{\phi_h^n-\phi_h^{n-1}}{\Delta t}\right \rVert^{\frac{4-d}{4}}\\
 & \quad \leq \lVert \nabla (\phi_h^n-\phi_h^{n-1})\rVert^{\frac{12+d}{8}}\lVert \Delta_h \phi_h^{n-1} \rVert^{\frac{1}{2}}\left \lVert \nabla \mu_h^n\right \rVert^{\frac{4-d}{8}}\\
 & \quad \leq \epsilon\frac{\lVert \nabla (\phi_h^n-\phi_h^{n-1})\rVert^2}{\Delta t}+C\Delta t\lVert \Delta_h \phi_h^{n-1} \rVert^{\frac{8}{4-d}}\lVert \nabla \mu_h^n \rVert^2.
 \end{align*}
 Similarly,
 \begin{align*}
 & |J_2|\leq \lVert \nabla (\phi_h^{n-1}-\phi_h^{n-2})\rVert\,\lVert \nabla \phi_h^{n-1}\rVert_{L^4(\Omega)}\left \lVert \frac{\phi_h^n-\phi_h^{n-1}}{\Delta t}\right \rVert_{L^4(\Omega)}\\
 & \quad \leq \epsilon'\frac{\lVert \nabla (\phi_h^{n-1}-\phi_h^{n-2})\rVert^2}{\Delta t}+C\Delta t\lVert \Delta_h \phi_h^{n-1} \rVert\,\left\lVert \frac{\nabla (\phi_h^{n}-\phi_h^{n-1})}{\Delta t}\right\rVert^{\frac{d+4}{4}}\lVert \nabla \mu_h^n \rVert^{\frac{4-d}{4}}\\
 & \quad \leq \epsilon'\frac{\lVert \nabla (\phi_h^{n-1}-\phi_h^{n-2})\rVert^2}{\Delta t}+\epsilon\frac{\lVert \nabla (\phi_h^{n}-\phi_h^{n-1})\rVert^2}{\Delta t}+C\Delta t\lVert \Delta_h \phi_h^{n-1} \rVert^{\frac{8}{4-d}}\lVert \nabla \mu_h^n \rVert^{2},
 \end{align*}
 where $\epsilon'>0$ is a coefficient which will be chosen properly. 
 
 Collecting the above results, we end up with
 \begin{align}
\label{hoe2}
& \notag  \frac{\lVert \nabla \mu_h^n\rVert^2}{2}+\frac{\lVert \nabla (\mu_h^n-\mu_h^{n-1})\rVert^2}{2}+(1-3\epsilon)\frac{\lVert \nabla (\phi_h^n-\phi_h^{n-1})\rVert^2}{\Delta t}
\\
&\quad  \leq \frac{\lVert \nabla \mu_h^{n-1}\rVert^2}{2}+\epsilon'\frac{\lVert \nabla (\phi_h^{n-1}-\phi_h^{n-2})\rVert^2}{\Delta t} \notag
\\
     & \qquad +C\Delta t\left(\lVert \nabla \mu_h^n \rVert^2+\lVert \nabla \mu_h^{n-1} \rVert^2\right)+C\Delta t\lVert \Delta_h \phi_h^{n-1} \rVert^{\frac{8}{4-d}}\lVert \nabla \mu_h^n \rVert^2.
 \end{align}
 Summing \eqref{hoe1} and \eqref{hoe2}, choosing $\epsilon <\frac{1}{4}$ and $\epsilon'=1-\epsilon =: k$, we have
 \begin{align*}
     & \frac{\lVert \nabla \phi_h^n\rVert^2}{2}+\frac{\lVert \nabla (\phi_h^n-\phi_h^{n-1})\rVert^2}{2}+(\Psi(\phi_h^n),1)^h+\frac{\lVert \nabla \mu_h^n\rVert^2}{2}+\frac{\lVert \nabla (\mu_h^n-\mu_h^{n-1})\rVert^2}{2}\\
     & \quad +k\frac{\lVert \nabla (\phi_h^n-\phi_h^{n-1})\rVert^2}{\Delta t}+\Delta t \lVert \nabla \mu_h^n\rVert^2
     \\
     &\qquad 
     \leq \frac{\lVert \nabla \phi_h^{n-1}\rVert^2}{2}+(\Psi(\phi_h^{n-1}),1)^h +\frac{\lVert \nabla \mu_h^{n-1}\rVert^2}{2}
     \\
     &\qquad \quad+k\frac{\lVert \nabla (\phi_h^{n-1}-\phi_h^{n-2})\rVert^2}{\Delta t}+C\Delta t \lVert \Delta_h \phi_h^n \rVert^2+C\Delta t\,h^{1-\frac{d}{4}}\lVert \Delta_h \phi_h^n \rVert^{\frac{1}{2}}\lVert \nabla \mu_h^n \rVert^2,\\
     & \qquad \quad  +C\Delta t\left(\lVert \nabla \mu_h^n \rVert^2+\lVert \nabla \mu_h^{n-1} \rVert^2\right)+C\Delta t\lVert \Delta_h \phi_h^{n-1} \rVert^{\frac{8}{4-d}}\lVert \nabla \mu_h^n \rVert^2.
 \end{align*}
 Summing the previous inequality for $n=1,\dots,N$, and considering that $\phi_h^{-1}=\phi_h^0$, we arrive at
 \begin{align}
 \label{hoe3}
 &\notag \frac{\lVert \nabla \phi_h^N\rVert^2}{2}+(\Psi(\phi_h^N),1)^h+\frac{\lVert \nabla \mu_h^N\rVert^2}{2}+k\frac{\lVert \nabla (\phi_h^N-\phi_h^{N-1})\rVert^2}{\Delta t}+\Delta t \sum_{n=1}^N\lVert \nabla \mu_h^n\rVert^2\\
 &\notag \quad \leq \frac{\lVert \nabla \phi_h^{0}\rVert^2}{2}+(\Psi(\phi_h^{0}),1)^h +\frac{\lVert \nabla \mu_h^{0}\rVert^2}{2}+C\Delta t \sum_{n=1}^N\left(\lVert \Delta_h \phi_h^n \rVert^2+\lVert \nabla \mu_h^n \rVert^2+\lVert \nabla \mu_h^{n-1} \rVert^2\right)\\
 & \qquad +C\Delta t \sum_{n=1}^N\left(\lVert \Delta_h \phi_h^{n-1} \rVert^{\frac{8}{4-d}}\lVert \nabla \mu_h^n \rVert^2+\Delta t\,h^{1-\frac{d}{4}}\lVert \Delta_h \phi_h^n \rVert^{\frac{1}{2}}\lVert \nabla \mu_h^n \rVert^2\right).
 \end{align}
 Let us introduce the function
 \[
 H^n:=\frac{\lVert \nabla \phi_h^n\rVert^2}{2}+\frac{\lVert \nabla \mu_h^n\rVert^2}{2}.
 \]
 Noting that $\Psi(\phi_h^N)$ is bounded from below by a negative constant, using Assumption $(ICR)$ and observing from \eqref{prel3} that 
 $$
 \lVert \Delta_h \phi_h^n\rVert^2\leq CH^n+C\Delta t\,h^{4-d}(H^n)^2+C\leq C+C(H^n)^2,
 $$ 
 we derive from the previous inequality that
 \begin{equation}
 \label{hoe4}
 H^N\leq C_0+C\Delta t \sum_{n=1}^N\left(1+\left(H^n\right)^{\frac{12-d}{4-d}}\right),
 \end{equation}
 where $C_0$ is a positive constant depending  on the norms of the initial condition. We observe that \eqref{hoe4} has the form of \eqref{bihari1}, with the constant $C$ and the sequence $a_n$ in \eqref{bihari1} equal to $C_0+CT$ and $C\Delta t$ respectively. By imposing \eqref{ap:5}, which gives an implicit bound on $\Delta t$ in terms of $C_0,C$ and $(12-d)/(4-d)$, we infer from \eqref{bihari2} that there exists an $\bar{N}\leq N$, which depends only $C_0,C,d$, such that
 \begin{equation}
 \label{hoe5}
 \sup_{n\in \{1,\dots,\bar{N}\}}\left \{\lVert \nabla \phi_h^n\rVert^2+\lVert \nabla \mu_h^n\rVert^2\right\}\leq C.
 \end{equation}
 Here we are assuming that, given $\Delta t \sim h^4$, $h$ is sufficiently small to ensure that \eqref{ap:5} is satisfied. Therefore, thanks to \eqref{prel3} and \eqref{hoe5} we obtain the last bound in \eqref{hoe}.
\medskip

 \noindent
  We now take $\xi\equiv \phi_h^n-\phi_h^{n-1}$ in \eqref{ph} and, similarly to \eqref{psipes}, we exploit \eqref{gndisc2}, \eqref{hoe5} and the last bound in \eqref{hoe} to obtain
\begin{align*}
& \sup_{n\in \{1,\dots,\bar{N}\}}\left \{\left(\Psi'_+(\phi_h^n),\phi_h^n-\bar{\phi}_h^n\right)^h\right\}\\
& \quad \leq \sup_{n\in \{1,\dots,\bar{N}\}} \left \{C(\lVert \nabla \mu_h^n\rVert+1)\lVert \nabla \phi_h^n\rVert+C\lVert \Delta_h \phi_h^n \rVert^2+C\lVert \Delta_h \phi_h^{n-1} \rVert^2\right\}\leq C.
\end{align*}
Then, with similar arguments as those already employed in the proof of Theorem \ref{thm:wps}, we obtain the first, the second and the third bounds in \eqref{hoe}. Finally, taking $\chi\equiv (\phi_h^n-\phi_h^{n-1})/\Delta t$ in \eqref{ph}, using the Cauchy--Schwarz inequality, \eqref{hoe3} and \eqref{hoe5}, we infer the fourth bound in \eqref{hoe}. 
\end{proof}

We now associate to the sequence of discrete solutions $(\phi_h^n,\mu_h^n)$ of Problem $P^h$, for $h,\Delta t>0$, $n\in \{1,\dots,\bar{N}\}$, the following piecewise constant and piecewise linear time interpolants over the interval $[0, \bar{T}]$, where $\bar{T}:=\Delta t \bar{N}$:
\begin{align}
    \label{cpcinterp}
    & \notag \Phi_h^{+}(t):=\phi_h^n, \quad \Phi_h^{-}(t):=\phi_h^{n-1}, \quad M_h^+(t):=\mu_h^n \quad M_h^-(t):=\mu_h^{n-1},\\
    & \Phi_h(t):=\frac{t-t^{n-1}}{\Delta t}\phi_h^n+\frac{t^n-t}{\Delta t}\phi_h^{n-1},
\end{align}
for $t \in (t_{n-1} , t_n ]$, $n = 1, \dots, \bar{N}$. The discrete weak formulation \eqref{ph} becomes: 
\begin{equation}
    \label{phint}
    \begin{cases}
        \left(\partial_t \Phi_h,\chi\right)^h+(\nabla M_h^+,\nabla \chi)=0,\\[5pt]
        (M_h^+,\xi)^h=(\nabla \Phi_h^+,\nabla \xi)+\left(\Psi'_1(\Phi_h^+)+\Psi'_2(\Phi_h^{-}),\xi\right)^h+\lambda\left(\nabla \Phi_h^+\cdot \nabla \Phi_h^{-},\xi\right),
    \end{cases}
\end{equation}
valid for all $(\chi,\xi)\in S^h\times S^h$ and for all $t\in (0,\bar{T}]$, with $\Phi_h(0)=\phi_h^0$. The next lemma provides the convergence results, based on the estimates
obtained in Lemma \ref{lem:hoe}, which are needed to study the limit of Problem $P^h$ as $h,\Delta t \to 0$. 
\begin{lemma}
    \label{lem:conv}
    Under the assumptions of Lemma \ref{lem:hoe}, there exist a subsequence of continuous and piecewise constant in time interpolants, which we still label by the index $h$, and functions 
    \[\phi \in L^{\infty}(0,\bar{T};W^{1,4}(\Omega))\cap W^{1,\infty}(0,\bar{T};L^2(\Omega)),\;\, \mu \in L^{\infty}(0,\bar{T};H^1(\Omega)),\;\, \chi \in L^{\infty}(0,\bar{T};L^2(\Omega)),
    \]
    with $-1<\phi<1$ a.e. in $\Omega \times (0,\bar{T})$,
    such that, for $h\to 0$ (and hence $\Delta t \sim h^4 \to 0$), the following convergence results hold:
    \begin{align}
    \label{conv1} & \Phi_h, \Phi_h^{\pm} \overset{\ast}{\rightharpoonup} \phi \quad \text{in} \quad L^{\infty}\left(0,\bar{T};H^1(\Omega)\right),\\
    \label{conv2} & M_h^{\pm} \overset{\ast}{\rightharpoonup} \mu \quad \text{in} \quad L^{\infty}\left(0,\bar{T};H^1(\Omega)\right),\\
    \label{conv3} & \partial_t\Phi_h \overset{\ast}{\rightharpoonup} \partial_t\phi \quad \text{in} \quad L^{\infty}\left(0,\bar{T};L^2(\Omega)\right),\\
    \label{conv4} & -\Delta_h \Phi_h^{\pm} \overset{\ast}{\rightharpoonup} \chi \quad \text{in} \quad L^{\infty}\left(0,\bar{T};L^2(\Omega)\right),\\
    \label{conv5} & \Phi_h^{\pm} \overset{\ast}{\rightharpoonup} \phi \quad \text{in} \quad L^{\infty}\left(0,\bar{T};W^{1,4}(\Omega)\right),\\
    \label{conv6} &  \Phi_h, \Phi_h^{\pm} {\rightarrow} \phi \quad \text{in} \quad L^{\infty}\left(0,\bar{T};L^r(\Omega)\right)\;\; \text{and} \;\; \text{a.e. in} \; \; \Omega\times (0,\bar{T}),\\
    \label{conv7} & \nabla \Phi_h^{\pm} {\rightarrow} \nabla \phi \quad \text{in} \quad L^{2}\left(\Omega \times (0,\bar{T});\mathbb{R}^d\right)\;\; \text{and} \;\; \text{a.e. in} \; \; \Omega\times (0,\bar{T}),\\
    \label{conv8} & \pi^h [\Psi'(\Phi_h^{\pm})] \overset{\ast}{\rightharpoonup} \Psi'(\phi) \quad \text{in} \quad L^{\infty}\left(0,\bar{T};L^2(\Omega)\right),
    \end{align}
    where $r\geq 1$ for $d=2$ and $r\in [1,6)$ for $d=3$.
\end{lemma}
\begin{proof}
The convergence properties \eqref{conv1}, \eqref{conv2}, \eqref{conv3} and \eqref{conv4} are direct consequences of \eqref{hoe}, \eqref{cpcinterp} and of the Banach--Alaoglu theorem. 
We note that the interpolants $\Phi_h^{\pm}$ converge to the same limit as $\Phi_h$. This can be proved by showing that
\begin{equation}
\label{lcconv}
\lVert \Phi_h-\Phi_h^{\pm}\rVert_{L^{\infty}(0,\bar{T};H^1(\Omega)}\to 0 \quad \text{as} \quad h,\Delta t \to 0. 
\end{equation}
Indeed, we can write
\[
\Phi_h-\Phi_h^+=\frac{t-t^{n}}{\Delta t}(\phi_h^{n}-\phi_h^{n-1}), \quad \Phi_h-\Phi_h^-=\frac{t-t^{n-1}}{\Delta t}(\phi_h^{n}-\phi_h^{n-1}).
\]
Using the latter formula and \eqref{hoe3}, we conclude that
\[
\lVert \Phi_h-\Phi_h^{\pm}\rVert_{L^{\infty}(0,\bar{T};H^1(\Omega))}\leq \Delta t \sup_{n\in \{1,\dots,\bar{N}\}}\frac{\lVert \nabla (\phi_h^n-\phi_h^{n-1}) \rVert^2}{\Delta t}\to 0\quad \text{as} \quad \Delta t \to 0.
\]
The convergence properties \eqref{conv1}, \eqref{conv3}, \eqref{lcconv} and the Aubin--Lions lemma imply \eqref{conv6}, while \eqref{conv5} is a consequence of \eqref{gndisc2}, \eqref{hoe} and of the Banach--Alaoglu theorem. 
\medskip

We are now left to prove \eqref{conv7} and \eqref{conv8}, which are nontrivial. Note in particular that $\Phi_h^{\pm}\notin H^2(\Omega)$, so the strong convergence result \eqref{conv7}, which is necessary to pass to the limit in the active term in \eqref{phint}$_2$, cannot be derived from the Aubin--Lions lemma starting from \eqref{conv4}. We obtain \eqref{conv7} by a duality argument. Given an arbitrary $v\in L^2(0,\bar{T};H^1(\Omega))$, let us take $\chi \equiv P_h(v(t))$, for a.e. $t\in (0,\bar{T})$, in \eqref{eqn:lapldiscr}, written for $v=\Phi_h^{\pm}$, integrate in time over $(0,\bar{T})$ and pass to the limit as $h\to 0$ in the resulting equation. By using \eqref{eqn:interp3}, the fact that $P_h v \to v$ strongly in $L^2(0,\bar{T};H^1(\Omega))$ as $h \to 0$, \eqref{conv1} and \eqref{conv4}, we obtain 
\begin{equation}
    \label{rr1}
    (\nabla \phi, \nabla v)_{L^2(\Omega\times (0,\bar{T}))} = (\chi, v)_{L^2(\Omega\times (0,\bar{T}))} \quad \forall \, v \in L^2(0,\bar{T};H^1(\Omega)).
\end{equation}
Choosing now $\chi \equiv \Phi_h^{\pm}$ in \eqref{eqn:lapldiscr}, written for $v=\Phi_h^{\pm}$, an integration in time over $(0,\bar{T})$ yields
\[
\lVert \nabla \Phi_h^{\pm}\rVert_{L^2(\Omega\times (0,\bar{T}))}^2 = (-\Delta_h \Phi_h^{\pm}, \Phi_h^{\pm})_{L^2(\Omega\times (0,\bar{T}))}.
\]
Passing to the limit as $h \to 0$, using \eqref{conv4}, \eqref{conv6} and \eqref{rr1}, we conclude that
\[
\lim_{h \to 0} \lVert \nabla \Phi_h^{\pm}\rVert_{L^2(\Omega\times (0,\bar{T}))}^2 = (\chi, \phi)_{L^2(\Omega\times (0,\bar{T}))}=\lVert \nabla \phi\rVert_{L^2(\Omega\times (0,\bar{T}))}^2,
\]    
which, together with the weak convergence \eqref{conv1}, implies \eqref{conv7} by the Radon--Riesz property of Hilbert spaces. 
\medskip

For what concerns \eqref{conv8}, \eqref{hoe} implies that $\lVert \pi^h[\Psi'(\Phi_h^{\pm})]\rVert_{L^{\infty}(0,\bar{T};L^2(\Omega))}\leq C$ uniformly in $h,\Delta t$. Since $\Psi_2'(\Phi_h^{\pm})=-\theta_0\Phi_h^{\pm}$, this entails that $\lVert \pi^h[\Psi_1'(\Phi_h^{\pm})]\rVert_{L^{\infty}(0,\bar{T};L^2(\Omega))}\leq C$. We now exploit this  to bound $\Psi_1'(\Phi_h^{\pm})$. Given an element $K\in \mathcal{T}_h$, thanks to the monotonicity of $\Psi_1'$ and to the fact that $\Phi_h^{\pm}$ is linear on $K$, using moreover \eqref{eqn:interp1} localized on $K$, we have 
\[
\lVert \Psi'_1(\Phi_h^{\pm})\lVert_{L^2(K)}^2 \leq \vert{}K\vert{} \max_{i} \vert{}\Psi'_1(\Phi_h^{\pm}(\mathbf{x}_i))\vert{}^2\leq C  \lVert\pi^h[ \Psi'_1(\Phi_h^{\pm})]\lVert_{L^2(K)}^2.
\]
Summing over $K$, we deduce that $\lVert \Psi_1'(\Phi_h^{\pm})\rVert_{L^{\infty}(0,\bar{T};L^2(\Omega))}\leq C$. Hence, the Banach--Alaoglu theorem gives that there exists a $\xi \in L^{\infty}(0,\bar{T};L^2(\Omega))$ such that, up to a subsequence,
\[
\Psi_1'(\Phi_h^{\pm})\overset{\ast}{\rightharpoonup} \xi \quad \text{in} \quad L^{\infty}\left(0,\bar{T};L^2(\Omega)\right).
\]
The previous convergence property together with \eqref{conv6}, which in particular implies that $\Phi_h^{\pm}\to \phi$ in $L^1(0,\bar{T};L^2(\Omega))$, gives that
\[
\lim_{h\to 0}\int_0^{\bar{T}}(\Psi_1'(\Phi_h^{\pm}),\Phi_h^{\pm})\, \mathrm{d}t=\int_0^{\bar{T}}(\xi,\phi)\, \mathrm{d}t.
\]
The maximal monotonicity of $\Psi'_1$ 
then implies that $\xi \equiv \Psi'_1(\phi)\in L^{\infty}(0,\bar{T};L^2(\Omega))$ (see e.g. \cite[Proposition 1.1, p.42]{Barbu}). Due to the logarithmic term in $\Psi'_1$, we have as a consequence that $-1<\phi<1$ a.e. in $\Omega \times (0,\bar{T})$. Finally, \eqref{eqn:interp5}, \eqref{conv6} and the properties $1<\Phi_h^{\pm},\phi<1$ a.e. in $\Omega\times (0,\bar{T})$ imply that
\begin{equation}
\label{pihpsi}
\pi^h \Psi'_1(\Phi_h^{\pm})-\Psi'_1(\Phi_h^{\pm})\overset{\ast}{\rightharpoonup} 0\quad \text{in} \quad L^{\infty}\left(0,\bar{T};L^2(\Omega)\right),
\end{equation}
which, together with the fact that $\Psi'_2(\Phi_h^{\pm})=-\theta_0\Phi_h^{\pm}$, conclude the proof of \eqref{conv8}.
\end{proof}

With the convergence results of Lemma \ref{lem:conv} it is possible to pass to the limit as $h,\Delta t \to 0$ in the discrete system \eqref{phint} and to identify the limit point $(\phi,\mu)$ as the unique   local in time weak solution of \eqref{ac-CH}-\eqref{ac-CH-bc}. In particular, the weak convergence \eqref{conv5} and the strong convergence \eqref{conv7} allow us pass to the limit in the active term in \eqref{phint}. The identification of the limit of the other terms in \eqref{phint} via the convergence results of Lemma \ref{lem:conv} is straightforward. The techniques to deal with the lumped products are standard and rely on the estimates \eqref{eqn:interp3}, \eqref{eqn:interp4} (see e.g. \cite{CE1992}). Hence, we conclude this section by stating the existence theorem  without delving into the calculations.
\begin{theorem}
    \label{thm:limitpoint}
    Under the assumptions of Lemma \ref{lem:hoe}, the limit point of Lemma \ref{lem:conv} satisfies the weak formulation
    \begin{equation}
    \label{pint}
    \begin{cases}
        \left(\partial_t \phi,\chi\right)+(\nabla \mu,\nabla \chi)=0,\\[5pt]
        (\mu,\xi)=(\nabla \phi,\nabla \xi)+\left(\Psi'(\phi),\xi\right)+\lambda\left(|\nabla \phi|^2,\xi\right),
    \end{cases}
\end{equation}
valid for all $(\chi,\xi)\in H^1(\Omega)$ and a.e. $t\in (0,\bar{T}]$, with $\phi(\cdot,0)=\phi_0$. Hence, it is the unique local in time weak solution of \eqref{ac-CH}-\eqref{ac-CH-bc} over the time interval $(0,\bar{T})$, where the final time $\bar{T}$ depends on the data of the problem.
\end{theorem}

\begin{remark}
The uniqueness of the solution of the weak formulation \eqref{pint} is proved with the same procedure as performed in Section $2$ to obtain \eqref{fe4}.
\end{remark}

\section{Numerical results}
In this section we report the results of numerical simulations, obtained by solving \eqref{ph} via a Newton method, in two space dimensions for different test cases.

In particular, we consider, as in \cite{WITT}, $\Omega=[0,256]^2$, and we vary $\lambda$ throughout the test cases. For $\lambda=0$ we consider $h=1$ and $\Delta t=0.01$, while for $\lambda\neq 0$ we take $h=\lambda$ and $\Delta t=0.01\,h^4/\lambda^4$. In Test Case $1$ we will consider the case with the regular potential \eqref{csplitr} and we will vary the value of $\lambda$ in the set $\{0,0.1,1,2\}$, in order to observe the different phase separation dynamics associated to increasing values of the active strength $\lambda$. In Test Case $2$ we will compare the phase separation dynamics obtained in the cases with regular potential \eqref{csplitr} and with singular potential \eqref{csplits}, both for $\lambda=1$. Finally, in Test Case $3$ we will show the growth law for domain length scale $L(t)$ associated to the coarsening dynamics in the case with the regular potential \eqref{csplitr} for $\lambda\in{0.1,1}$.

\subsection{Test Case $1$}
We consider the initial condition $\phi_0 = 0 \pm 0.1\iota$, where $\iota$ is a random perturbation uniformly distributed in the interval $[0, 1]$. In Figure \ref{fig:ps} we show the numerical results at different time points throughout the phase separation dynamics, up to late times at which we can observe the coarsening dynamics of the separated domain subregions, for the cases $\lambda=0,0.1,1,2$.

\begin{figure}[ht!]
\includegraphics[width=0.8\linewidth]
{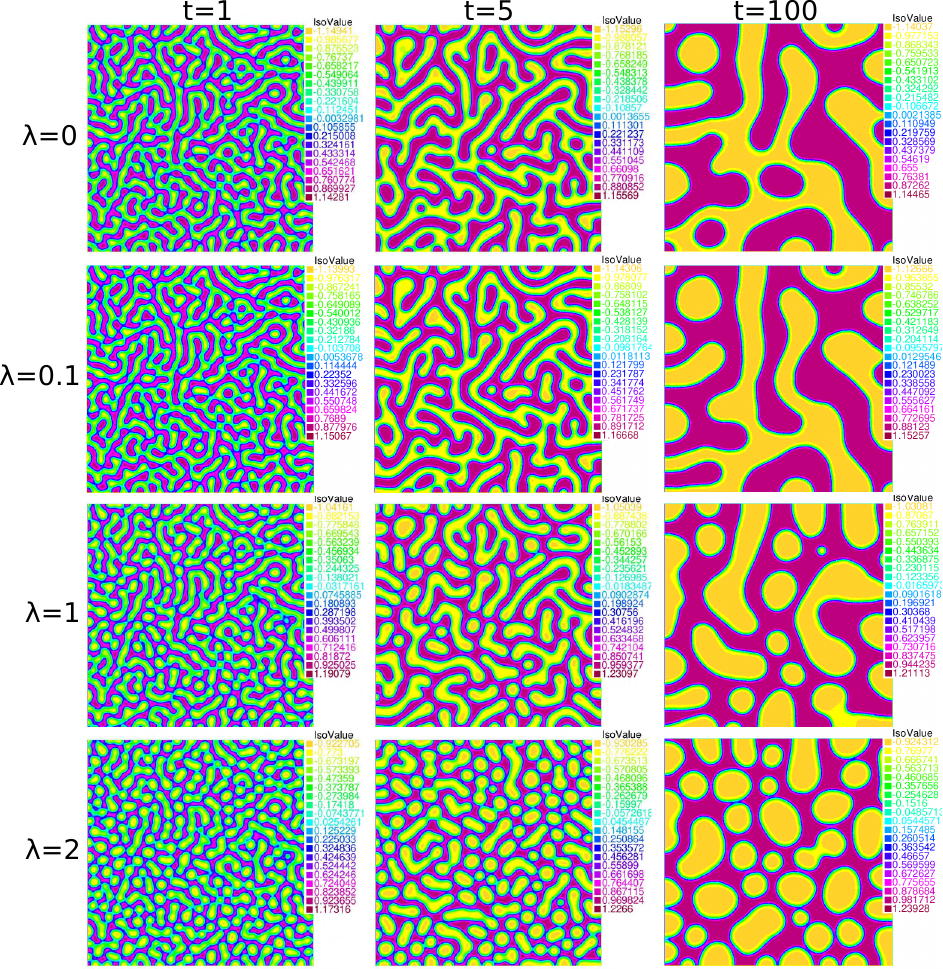}
\centering
\caption{Snapshots of evolving phase separation in two spatial dimensions at times $t=1,5,100$ for $\lambda=0,1,2$. The initial condition is a small uniformly distributed random perturbation around the value $\phi_0=0$.}
\label{fig:ps}
\end{figure}
We observe from figure \ref{fig:ps} that the phase separation dynamics for the $\phi$ variable consists, for $\lambda=0$, in the formation of maze-like bicontinuous structures alternating between the two pure phases $\phi \sim -1$ and $\phi \sim 1$. For $\lambda >0$, the activity-induced shift \eqref{rootslamb} and mass conservation induce an imbalance between the fractions of $\phi_a$ and $\phi_c$ at late times, with the formation of a droplet morphology of clusters of $\phi\sim \phi_a$ immersed in a background of $\phi\sim \phi_c$ phase. This effect is more evident with increasing values of $\lambda$, and is in accordance with the numerical results reported in \cite{WITT,PATTANAYAK2021}.

\subsection{Test Case $2$}
We now compare the phase separation dynamics between the cases with the regular and the singular potentials, considering the initial condition $\phi_0 = 0 \pm 0.1\iota$ and $\lambda=1$. In the case with the singular potential, we take $\theta=\frac{1}{2}$, $\theta_0=1$ and we solve the regularized problem \eqref{pha} with $\alpha=0.001$, in order to avoid numerical problems due to the singularities at the pure phases. In Figure \ref{fig:psrs} we show the numerical results at different time points throughout the phase separation dynamics.

\begin{figure}[ht!]
\includegraphics[width=0.7\linewidth]
{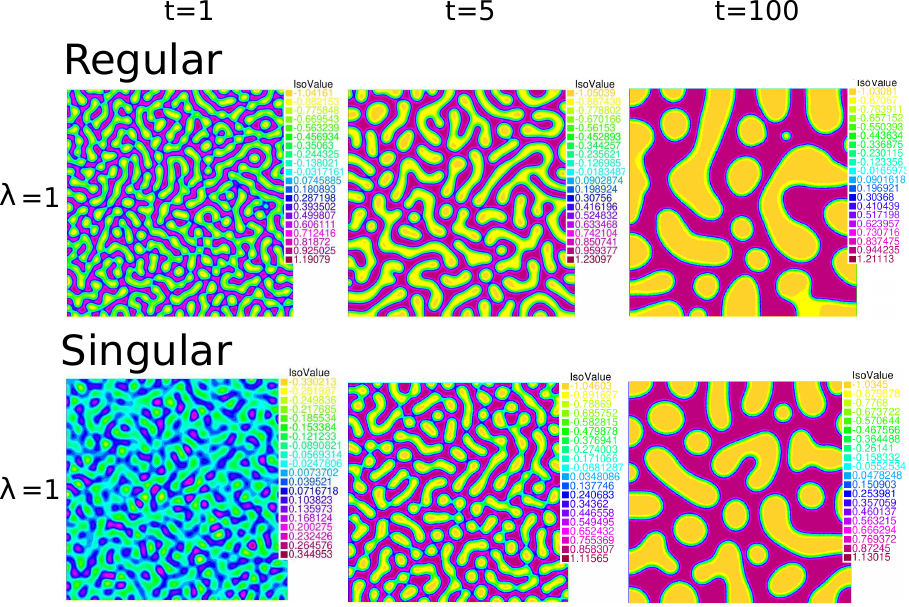}
\centering
\caption{Snapshots of evolving phase separation in two spatial dimensions at times $t=1,5,100$ for $\lambda=1$, in the case with regular (top) and singular (bottom) potential. The initial condition is a small uniformly distributed random perturbation around the value $\phi_0=0$.}
\label{fig:psrs}
\end{figure}
We observe from figure \ref{fig:psrs} that the phase separation dynamics in the regular and singular cases are qualitatively identical, consisting in the formation of a droplet morphology with shifted core values, which can be identified in both cases by the uncommon tangent construction. In the case with the singular potential we observe a time delay in the phase separation dynamics, which is slower with respect to the case with regular potential.

\subsection{Test Case $3$}
We consider an ensemble of $20$ independent test cases with regular potential, $\lambda=2$ and with initial conditions given by small uniformly distributed random independent perturbations around the value $\phi_0=0$. We run the numerical simulations up to late times till $t\sim 5\cdot 10^5$. We then calculate the domain length scale $L(t)$, following \cite{Agosti,WITT}, as the inverse of the first moment of the spherically and ensemble averaged structure factor. In \cite{PATTANAYAK2021} the long-time coarsening dynamics, up to times $t\sim 5\cdot 10^4$, of the same system (with periodic boundary conditions) was extensively numerically investigated, providing evidences for a crossover in the power law $L(t)\sim t^{\frac{1}{z}}$ from $z=3$ at early times to $z=4$ at late times. The crossover occurred earlier for larger values of $\lambda$. In \cite{WITT}, the numerical investigation of the growth law for $L(t)$ at longer times, up to $t\sim 10^5$, led to the observation of even slower asymptotic growth than $L(t)\sim t^{\frac{1}{4}}$ in the case $|\lambda|=2$ (see Figure $2$ in \cite{WITT}). In figure \ref{fig:lt} we report the plot $L(t)$ vs $t$ for $\lambda=2$ up to late times $t\sim 5\cdot 10^5$, comparing it with the $t^{\frac{1}{3}}$ and $t^{\frac{1}{4}}$ growth laws.

\begin{figure}[ht!]
\includegraphics[width=0.8\linewidth]
{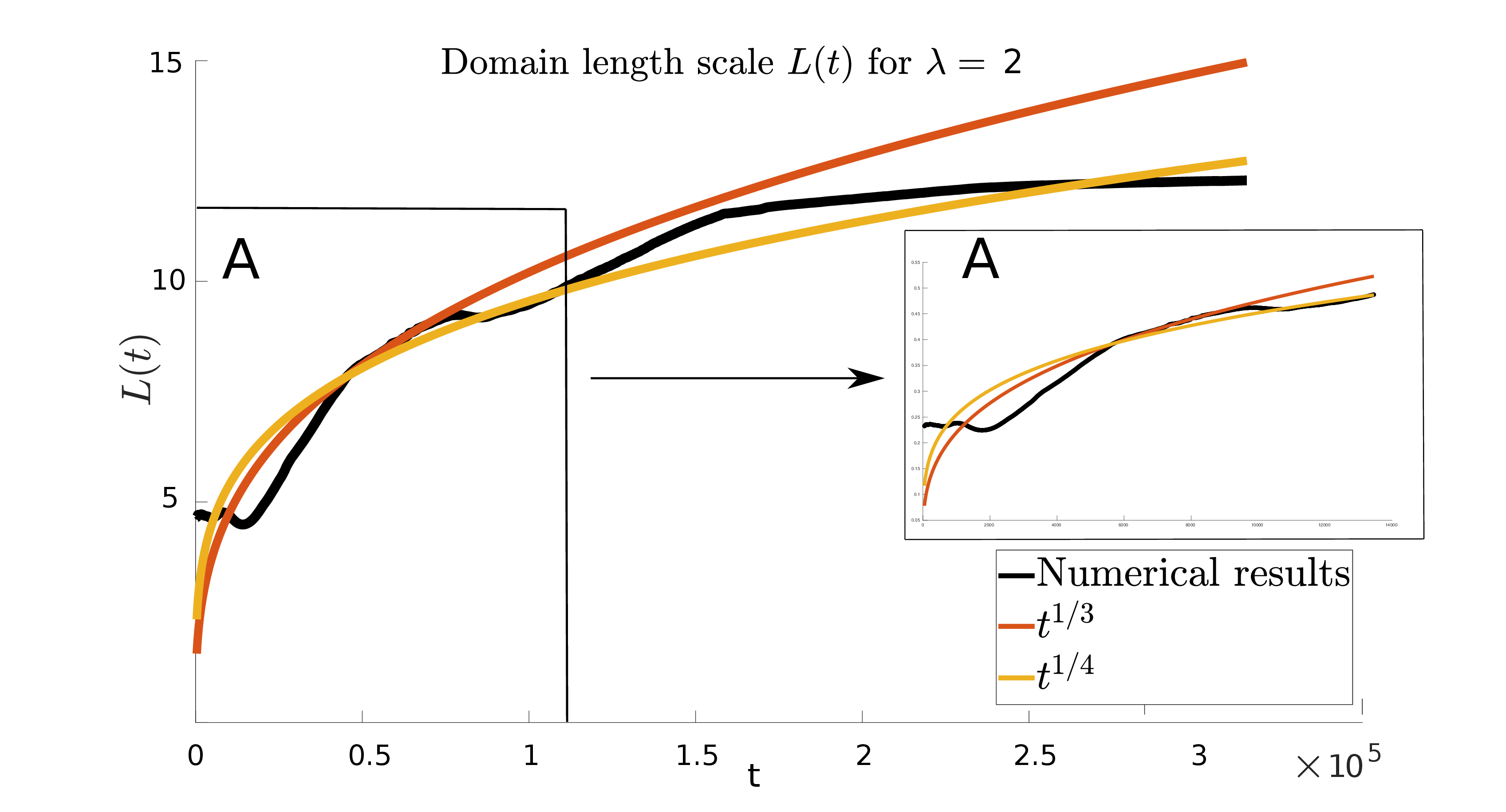}
\centering
\caption{Plot of the domain length scale $L(t)$ for $\lambda=2$ in two spatial dimensions. The growth laws $L(t)\sim t^{\frac{1}{3}}$ (LS law) and $L(t)\sim t^{\frac{1}{4}}$ are reported for comparison, together with an inset (A) highlighting the growth profile at earlier times. The initial conditions in the ensemble are small uniformly distributed random independent perturbations around the value $\phi_0=0$.}
\label{fig:lt}
\end{figure}
We observe from figure \ref{fig:lt} that a crossover from the $L(t)\sim t^{\frac{1}{3}}$ to the $L(t)\sim t^{\frac{1}{4}}$
growth laws at late times is indeed observed also in our test case (see the inset A). This confirms the observations reported in \cite{PATTANAYAK2021}. Going even further in time, we observe, after a transitory growth regime, a slowdown of the growth law, as observed in \cite{WITT}, which eventually converges to saturation, confirming the theoretical results found in section \ref{pls}. We conclude by noting that, since for lower values of $\lambda$ the crossover in the growth law for $L(t)$ and the possible relaxation to saturation happen at longer times, we actually found too computationally demanding to investigate saturation occurrence for $\lambda <2$. This could be also the reason why no slowdown of the growth law $L(t)\sim t^{\frac{1}{4}}$ was observed in \cite{WITT} for $|\lambda|<2$ in the time window up to $t\sim 10^5$. We leave the design of more computationally efficient numerical implementations capable of investigating the full long-time coarsening dynamics for values $|\lambda|<2$ for future investigations.

\appendix
\section{Detailed calculations for the static droplet solution at II order}
\label{ap:shift}
Following \eqref{heterocliniclambda}, we assume that there exists an heteroclinic trajectory for \eqref{kink2xilamb} which is a quartic curve of the form
\begin{align}
    \label{heteroclinicxilamb}  
    \notag \chi=&\left(A_{1,1}+B_{1,1}\lambda \xi\right)(f-f_a)(f-f_c)+\lambda A_{1,2}(f-f_a)(f-f_c)^2+\lambda A_{2,1}(f-f_a)^2(f-f_c)\\
    &+\lambda^2A_{2,2}(f-f_a)^2(f-f_c)^2+\lambda^2A_{1,3}(f-f_a)(f-f_c)^3+\lambda^2A_{3,1}(f-f_a)^3(f-f_c),
\end{align}
where $A_{1,1},A_{1,2},A_{2,1},A_{2,2},A_{1,3},A_{3,1},B_{1,1}$ are real coefficients to be determined by inserting \eqref{heteroclinicxilamb} in \eqref{kink2xilamb} and using the principle of polynomial identity. We observe that the term $B_{1,1}\lambda \xi(f-f_a)(f-f_c)$ in \eqref{heteroclinicxilamb} is needed to deal with contributions of order $O(\lambda\xi)$ in the term $-\xi\chi$ in the right hand side of \eqref{kink2xilamb}$_2$.
Using \eqref{rootsxilamb} in \eqref{heteroclinicxilamb}, and keeping only terms at most of order $O(\lambda^2), O(\xi^2), O(\lambda\xi)$, we obtain
that
\begin{align}
\label{heteroclinicxi2lamb}
    \notag \chi=&\left(A_{1,1}+B_{1,1}\lambda \xi\right)(f^2-1)-A_{1,1}\mu_{s,1}\lambda f-A_{1,1}\rho \xi f+A_{1,1}\mu_{s,1}^2\lambda^2+A_{1,1}\rho^2\xi^2+2A_{1,1}\mu_{s,1}\rho\lambda\xi\\
    & \notag +\lambda A_{1,2}(f+1)(f-1)^2+\lambda A_{2,1}(f+1)^2(f-1)-\lambda^2\frac{\mu_{s,1}}{2}A_{1,2}(f-1)^2-\lambda\xi\frac{\rho}{2}A_{1,2}(f-1)^2\\
    & \notag -\lambda^2\frac{\mu_{s,1}}{2}A_{2,1}(f+1)^2-\lambda\xi\frac{\rho}{2}A_{2,1}(f+1)^2-\lambda^2\mu_{s,1}(A_{1,2}+A_{2,1})(f^2-1)-\lambda\xi\rho (A_{1,2}+A_{2,1})(f^2-1)\\ 
    & +\lambda^2A_{2,2}(f+1)^2(f-1)^2+\lambda^2A_{1,3}(f+1)(f-1)^3+\lambda^2A_{3,1}(f+1)^3(f-1).
\end{align}
Substituting \eqref{heteroclinicxi2lamb} in \eqref{kink2xilamb}$_{2}$ and keeping only terms at most of order $O(\lambda^2), O(\xi^2), O(\lambda\xi)$, we obtain that 
\begin{align}
\label{heteroclinicxi5lamb}
    \notag \frac{d \chi}{d\tilde{r}}=& f(f^2-1)-A_{1,1}\xi(f^2-1)+A_{1,1}\mu_{s,1}\lambda \xi f+A_{1,1}\rho \xi^2f-\lambda\xi A_{1,2}(f+1)(f-1)^2\\
    &\notag -\lambda\xi A_{2,1}(f+1)^2(f-1)+\lambda A_{1,1}^2(f^2-1)^2-2\lambda A_{1,1}^2(\mu_{s,1}\lambda+\rho\xi)f(f^2-1)\\
    & +2\lambda^2A_{1,1}A_{1,2}(f+1)^2(f-1)^3+2\lambda^2A_{1,1}A_{2,1}(f+1)^3(f-1)^2-\mu_{s,1}\lambda-\rho \xi.
\end{align}
Taking the derivative of \eqref{heteroclinicxi2lamb} with respect to $\tilde{r}$, still keeping only terms at most of order $O(\lambda^2), O(\xi^2), O(\lambda\xi)$, and equating the result to \eqref{heteroclinicxi5lamb}, we obtain that 
\begin{align*}
&A_{1,1}=- \frac{1}{\sqrt{2}}, \quad \rho=\frac{1}{3A_{1,1}}=-\frac{\sqrt{2}}{3}, \quad A_{1,2}=A_{2,1}=\frac{A_{1,1}}{10}=-\frac{1}{10\sqrt{2}}, \quad \mu_{s,1}=\frac{4}{15},\\
&A_{2,2}=-\frac{17}{300\sqrt{2}}, \quad A_{1,3}=A_{3,1}=\frac{1}{200\sqrt{2}}, \quad B_{1,1}=-\frac{1}{20A_{1,1}}=\frac{\sqrt{2}}{20}.
\end{align*}
We obtain that, at second order in $\lambda,\xi$, the static droplet solution must solve the following boundary value problem:
\begin{equation}
    \label{chiodexilamb}
    \begin{cases}
    \frac{df}{d\tilde{r}}=- \frac{1}{\sqrt{2}}(f^2-1)-\frac{\lambda}{5\sqrt{2}}f\left(f^2-\frac{7}{3}\right)-\frac{1}{3}\xi f-\frac{\lambda^2}{450\sqrt{2}}(21f^4-87f^2+74)-\frac{\sqrt{2}}{9}\xi^2\\
    \qquad -\frac{\lambda\xi}{180}\left((18-9\sqrt{2})f^2-(40-9\sqrt{2})\right),\\
    f(-\infty)\to - 1+\frac{2}{15}\lambda-\frac{\sqrt{2}}{6}\xi+\frac{2}{75}\lambda^2+\frac{1}{12}\xi^2-\frac{\sqrt{2}}{15}\lambda\xi,\\
    f(+\infty)\to 1+\frac{2}{15}\lambda-\frac{\sqrt{2}}{6}\xi-\frac{2}{75}\lambda^2-\frac{1}{12}\xi^2+\frac{\sqrt{2}}{15}\lambda\xi.
    \end{cases}
\end{equation}
In order to solve \eqref{chiodexilamb}, we assume that its solution can be written as 
\begin{equation}
\label{f012345}
f(\tilde{r})=f_0(\tilde{r})+\lambda f_1(\tilde{r})+\xi f_2(\tilde{r})+\lambda^2f_3(\tilde{r})+\xi^2 f_4(\tilde{r})+\lambda \xi f_5(\tilde{r})+o(\lambda^2)+o(\xi^2)+o(\lambda\xi).
\end{equation}
At order $O(1)$ we get the boundary value problem
\begin{equation*}
    \begin{cases}
    \frac{df_0}{d\tilde{r}}=- \frac{1}{\sqrt{2}}(f_0^2-1),\\
    f_0(-\infty)\to - 1,\\
    f_0(+\infty)\to  1,
    \end{cases}
\end{equation*}
which has the solution
\begin{equation}
\label{f0}
f_0(\tilde{r})=\tanh\left(\frac{\tilde{r}}{\sqrt{2}}\right).
\end{equation}
At order $O(\lambda)$ we get the boundary value problem
\begin{equation*}
    \begin{cases}
    \frac{df_1}{d\tilde{r}}=- \sqrt{2}\tanh\left(\frac{\tilde{r}}{\sqrt{2}}\right)f_1(\tilde{r})-\frac{1}{5\sqrt{2}}\tanh\left(\frac{\tilde{r}}{\sqrt{2}}\right)\left(\tanh^2\left(\frac{\tilde{r}}{\sqrt{2}}\right)-\frac{7}{3}\right),\\
    f_1(-\infty)\to \frac{2}{15},\\
    f_1(+\infty)\to \frac{2}{15},
    \end{cases}
\end{equation*}
which has the solution
\begin{equation}
\label{f1}
f_1(\tilde{r})=\frac{1}{5}\frac{\log \left(\cosh \left(\frac{\tilde{r}}{\sqrt{2}}\right)\right)}{\cosh^2\left(\frac{\tilde{r}}{\sqrt{2}}\right)}+\frac{2}{15}.
\end{equation}
At order $O(\xi)$ we get the boundary value problem
\begin{equation*}
    \begin{cases}
    \frac{df_2}{d\tilde{r}}=- \sqrt{2}\tanh\left(\frac{\tilde{r}}{\sqrt{2}}\right) f_2(\tilde{r})-\frac{1}{3}\tanh\left(\frac{\tilde{r}}{\sqrt{2}}\right),\\
    f_2(-\infty)\to -\frac{\sqrt{2}}{6},\\
    f_2(+\infty)\to -\frac{\sqrt{2}}{6},
    \end{cases}
\end{equation*}
which has the constant solution $f_2\equiv -\frac{\sqrt{2}}{6}$.
At order $O(\lambda^2)$ we get the boundary value problem
\begin{equation*}
    \begin{cases}
    \frac{df_3}{d\tilde{r}}=- \sqrt{2}f_0(\tilde{r})f_3(\tilde{r})-\frac{1}{\sqrt{2}}f_1(\tilde{r})\left(f_1(\tilde{r})+\frac{3}{5}f_0^2(\tilde{r})-\frac{7}{15}\right)-\frac{7}{150\sqrt{2}}f_0^4(\tilde{r})+\frac{29}{150\sqrt{2}}f_0^2(\tilde{r})-\frac{37}{225\sqrt{2}},\\
    f_3(-\infty)\to \frac{2}{75},\\
    f_3(+\infty)\to -\frac{2}{75},
    \end{cases}
\end{equation*}
which has the solution
\begin{align}
\notag f_3(\tilde{r})=&\frac{\tanh\left(\frac{\tilde{r}}{\sqrt{2}}\right)}{25\cosh^2\left(\frac{\tilde{r}}{\sqrt{2}}\right)}\left[\log \left(\cosh \left(\frac{\tilde{r}}{\sqrt{2}}\right)\right)-\log^2 \left(\cosh \left(\frac{\tilde{r}}{\sqrt{2}}\right)\right)-\frac{1}{6}\right]\\
\label{f3}&-\frac{2}{75}\tanh\left(\frac{\tilde{r}}{\sqrt{2}}\right)-\frac{13\sqrt{2}}{300}\frac{\tilde{r}}{\cosh^2\left(\frac{\tilde{r}}{\sqrt{2}}\right)}.
\end{align}
At order $O(\xi^2)$ we get the boundary value problem
\begin{equation*}
    \begin{cases}
    \frac{df_4}{d\tilde{r}}=- \sqrt{2}\tanh\left(\frac{\tilde{r}}{\sqrt{2}}\right)f_4(\tilde{r})-\frac{\sqrt{2}}{12},\\
    f_4(-\infty)\to \frac{1}{12},\\
    f_4(+\infty)\to -\frac{1}{12},
    \end{cases}
\end{equation*}
which has the solution 
\begin{equation}
\label{f4}
f_4(\tilde{r})=-\frac{1}{12}\tanh\left(\frac{\tilde{r}}{\sqrt{2}}\right)-\frac{\tilde{r}}{12\sqrt{2}\cosh^2\left(\frac{\tilde{r}}{\sqrt{2}}\right)}.
\end{equation}
Finally, at order $O(\lambda\xi)$ we get the boundary value problem
\begin{equation*}
    \begin{cases}
    \frac{df_5}{d\tilde{r}}=- \sqrt{2}\tanh\left(\frac{\tilde{r}}{\sqrt{2}}\right)f_5(\tilde{r})+\frac{\sqrt{2}}{20}\tanh^2\left(\frac{\tilde{r}}{\sqrt{2}}\right)+\frac{8-3\sqrt{2}}{60},\\
    f_5(-\infty)\to -\frac{\sqrt{2}}{15},\\
    f_5(+\infty)\to \frac{\sqrt{2}}{15},
    \end{cases}
\end{equation*}
which has the solution
\begin{equation}
\label{f5}
f_5(\tilde{r})=\frac{\sqrt{2}}{15}\tanh\left(\frac{\tilde{r}}{\sqrt{2}}\right)+\frac{4-3\sqrt{2}}{60}\frac{\tilde{r}}{\cosh^2\left(\frac{\tilde{r}}{\sqrt{2}}\right)}.
\end{equation}

\bigskip

\noindent
\textbf{Acknowledgments.} AA and AG are members of Gruppo Nazionale per l'Analisi Ma\-te\-ma\-ti\-ca, la Probabilit\`{a} e le loro Applicazioni (GNAMPA), Istituto Nazionale di Alta Matematica (INdAM). 





\begin{thebibliography}{99}


\bibitem{Agosti} \au{A. Agosti, P.F. Antonietti, P. Ciarletta, M. Grasselli, M. Verani}, \ti{A {C}ahn--{H}illiard-type equation with application to tumor growth dynamics}, 
\jou{Math. Meth. Appl. Sci.} \no{40}{7598--7626}{2017}

\bibitem{Barbu}
{\au V. Barbu},
{\bk Nonlinear semigroups and differential equations in Banach spaces}, 
\eds{Editura Academiei Republicii Socialiste
Rom\^{a}nia, Bucharest; Noordhoff International Publishing, Leiden.}{1976}

\bibitem{Barrett1} \au{J. W. Barrett, J. F. Blowey and H. Garcke}, \ti{Finite element approximation of the Cahn--Hilliard equation with degenerate mobility}, 
\jou{SIAM J. Numer. Anal.} \no{37}{286--318}{1999}

\bibitem{Barrett2} \au{J. W. Barrett, R. N\"{u}rnberg, and V. Styles}, \ti{Finite element approximation of a phase field model for void electromigration}, 
\jou{SIAM J. Num. Anal.} \textbf{42}(2) (2004), 738--772.

\bibitem{Bihari} \au{I. Bihari}, \ti{A generalisation of a lemma of Bellman and its application to uniqueness problems of differential equations}, 
\jou{Math. Acad. Sci.
Hungar.} \no{7}{81--94}{1956}

\bibitem{Bray} \au{A.J. Bray}, \ti{Theory of phase ordering kinetics}, 
\jou{Adv. Phys.} \no{43}{357--459}{1994}

\bibitem{Brenner} \au{S. C. Brenner and L. R. Scott}, {\bk The Mathematical Theory of Finite Element Methods}, 
\eds{Springer-Verlag}{New York}{2008}

\bibitem{BREZIS2010}
{\au H. Brezis},
{\bk Functional Analysis, Sobolev Spaces and partial
Differential Equations},
\eds{Springer-Verlag}{New York}{2010}


\bibitem{Brezis} \au{H. Brezis and P. Mironescu}, \ti{Gagliardo-Nirenberg inequalities and non-inequalities: the full story}, 
\jou{Ann. Inst. H. Poincar\'{e} - Anal. Non Lin\'{e}aire} \no{35}{1355--1376}{2018}

\bibitem{Burekovic} \au{S. Burekovi\'{c}, F. De Luca, M. E. Cates, C. Nardini}, \ti{Active {C}ahn--{H}illiard theory for non-equilibrium phase separation: quantitative macroscopic predictions and a microscopic derivation}, 
\jou{arXiv:2601.16539} (2026).


\bibitem{CH} \au{J.W. Cahn, J.E. Hilliard}, 
\ti{Free energy of a nonuniform
system. I. Interfacial free energy}, \jou{J. Chem. Phys.} \textbf{28}
(1958), 258--267.

\bibitem{CH2} \au{J.W. Cahn, J.E. Hilliard}, 
\ti{Spinodal decomposition: a
reprise}, \jou{Acta Metallurgica} \textbf{19} (1971), 151--161.


\bibitem{Calgaro} \au{C. Calgaro, C. Canc\e`s, E. Creus\'{e}}, \ti{Discrete Gagliardo-Nirenberg inequality and application to the finite volume
approximation of a convection–diffusion equation with a Joule effect term}, 
\jou{IMA J. Numer. Anal.} \textbf{44} (2024), 2394--2436.

\bibitem{Ciavaldini} \au{J.F. Cialvaldini}, 
\ti{Analyse Numerique d'un Probleme de Stefan a Deux Phases Par une Methode d'Elements Finis}, \jou{SIAM J. Numer. Anal.} \textbf{12}{3} (1975), 464--487.

\bibitem{CT2018}
{\au M.E. Cates, E. Tjhung},
{\ti Theories of binary fluid mixtures: from phase-separation kinetics to
active emulsions}, 
{\jou J.\ Fluid Mech.} 
\no{836}{P1}{2018}

\bibitem{CE1992}
{\au M.I.M. Copetti, C. Elliott},
{\ti Numerical analysis of the {C}ahn--{H}illiard equation with a logarithmic free energy}, 
{\jou Numer. Math.} 
\no{63}{39-65}{1992}



\bibitem{Dauge}
{\au M. Dauge},
{\ti Neumann and mixed problems on curvilinear polyhedra}, 
{\jou ntegr. equ. oper. theory} 
\no{15}{227--261}{1992}


\bibitem{E} {\au C.M. Elliott}, {\ti The Cahn-Hilliard model for the
kinetics of phase separation}, Mathematical models for phase change problems (\'{O}bidos, 1988), 35--73, Internat. Ser. Numer. Math. \textbf{88}, Birkh\"{a}user, Basel, 1989.


\bibitem{GMS} \au{G. Gilardi, A. Miranville, G., Schimperna}, 
\ti{On the Cahn–Hilliard equation with irregular potentials
and dynamic boundary conditions}, 
\jou{Commun. Pure Appl. Anal.} \no{8}{881--912}{2009}


\bibitem{Grisvard}
{\au P. Grisvard},
{\bk Elliptic Problems in Non Smooth Domains}, 
\eds{Monogr. Stud. Math., vol. 24, Pitman}{1985}


\bibitem{KS} \au{Keener, J., Sneyd, J.}, 
\bk{Mathematical Physiology. I: Cellular Physiology}, 
Springer New York, NY, ISSN 0939-6047, 2010.

\bibitem{Leoni} \au{G. Leoni}, {\bk A First Course in Sobolev Spaces: Second Edition. Graduate Studies in Mathematics. 181}, 
\eds{American Mathematical Society}{2017}

\bibitem{Liu} \au{Y. Liu, W. Chen, C. Wang, 
S. M. Wise}, 
\ti{Error analysis of a mixed finite element method for a Cahn--Hilliard--Hele--Shaw system}, 
\jou{Numer. Math.} \no{135}{679--709}{2017}

\bibitem{LS} \au{I. M. Lifshitz, V. V. Slyozov}, 
\ti{The kinetics of precipitation from supersaturated solid solutions}, 
\jou{J. Phys. Chem. Solids} \textbf{19}(1--2)(1961), 35--50.


\bibitem{MZ} \au{A. Miranville, S. Zelik}, 
\ti{Robust exponential attractors
for Cahn-Hilliard type equations with singular potentials}, 
\jou{Math.
Methods. Appl. Sci.} \textbf{27} (2004), 545--582.

\bibitem{Mbook} {\au A. Miranville}, {\bk The Cahn-Hilliard Equation: Recent
Advances and Applications}, CBMS-NSF Regional Conf. Ser. in Appl. Math. 
\textbf{95}, SIAM, Philadelphia, PA., 2019.

\bibitem{PATTANAYAK2021}
{\au S. Pattanayak, S. Mishra, S. Puri}, 
{\ti Ordering kinetics in the active model B}, 
{\jou Phys. Rev. E} 
\no{104}{014606}{2021}

\bibitem{PATTANAYAK2021-2}
{\au S. Pattanayak, S. Mishra, S. Puri},
{\ti Domain Growth in the Active Model B: Critical and Off-critical Composition},
{\jou Soft Materials} 
\no{19}{286--296}{2021}

\bibitem{Pachpatte}
{\au B. G. Pachpatte},
{\ti Integral Inequalities of the Bihari type},
{\jou Math. Inequal. Appl.} 
\no{5}{4}{649--657}{2002}

\bibitem{Suli}
{\au D. Kay, V. Styles and E. Suli},
{\ti Discontinuous Galerkin Finite Element Approximation of the Cahn--Hilliard Equation with Convection},
{\jou SIAM J. NUMER. ANAL.} 
\no{47}{4}{2660--2685}{2009}

\bibitem{Quarteroni} 
\au{A. Quarteroni and A. Valli}, 
{\bk Numerical Approximation of Partial Differential Equations},
\eds{Springer-Verlag}{Berlin}{2008}

\bibitem{T} 
\au{R. Temam}, 
{\bk Infinite-dimensional dynamical systems in mechanics and physics},
\eds{Springer-Verlag}{New York}{1997}

\bibitem{temam}
\au{R. Temam}, 
{\bk Convex Analysis and Variational Problems},
\eds{SIAM}{1999}

\bibitem{TJHUNG2018}
{\au E. Tjhung, C. Nardini, M. E. Cates}, 
{\ti Cluster phases and bubbly phase separation in active fluids: reversal of the Ostwald process}, 
{\jou Phys. Rev. X}
\no{8}{031080}{2018}

\bibitem{TJHUNG2015}
{\au E. Tjhung, A. Tiribocchi, D. Marenduzzo, M. E. Cates},
{\ti A minimal physical model captures the shapes of crawling cells},
{\jou Nat. Commun.}
\no{6}{5420}{2015}

\bibitem{WITT}
{\au R. Wittkowski, A. Tiribocchi, J. Stenhammar, R. J. Allen, D. Marenduzzo, M. E. Cates},
{\ti Scalar $\varphi^4$ field theory for active-particle phase separation},
{\jou Nat. Commun.}
\no{5}{4351}{2014}

\end{thebibliography}
\end{document}